\documentclass[10pt]{amsart}

\usepackage[
paper=a4paper,
headsep=15pt,text={150mm,235mm},includehead,centering
]{geometry}

\usepackage{float,array,multirow,calc,booktabs}

\usepackage[bookmarks]{hyperref}
\hypersetup{colorlinks=true,linkcolor=blue,citecolor=blue}

\usepackage{amssymb,amsxtra,bbm}
\usepackage[shortlabels]{enumitem}
\setlist[enumerate,1]{label=\textup{(\arabic*)}}

\usepackage{graphicx,color}

\usepackage{tikz}
\usetikzlibrary{
  cd,
  calc,
  positioning,
  arrows,
  decorations.pathreplacing,
  decorations.markings,
}

\definecolor{todo-background-color}{gray}{0.95}
\usepackage[
  textwidth=1.1in,
  backgroundcolor=todo-background-color,
  bordercolor=black,
  linecolor=black,
  textsize=footnotesize
]{todonotes}

\iftrue
\makeatletter
\def\@settitle{%
  \vspace*{10pt}
  \begin{flushleft}%
    \LARGE\bfseries
    \strut\@title\strut
  \end{flushleft}%
}
\def\@setauthors{%
  \begingroup
  \def\thanks{\protect\thanks@warning}%
  \trivlist
  \raggedright
  \large \@topsep28\p@\relax
  \advance\@topsep by -\baselineskip
  \item\relax
  \author@andify\authors
  \def\\{\protect\linebreak}%
  \authors
  \ifx\@empty\contribs
  \else
    ,\penalty-3 \space \@setcontribs
    \@closetoccontribs
  \fi
  \normalfont
  \endtrivlist
  \endgroup
}
\def\@setaddresses{\par
  \nobreak \begingroup
  \small\raggedright
  \def\author##1{\nobreak\addvspace\smallskipamount}%
  \def\\{\unskip, \ignorespaces}%
  \interlinepenalty\@M
  \def\address##1##2{\begingroup
    \par\addvspace\bigskipamount\noindent
    \@ifnotempty{##1}{(\ignorespaces##1\unskip) }%
    {\ignorespaces##2}\par\endgroup}%
  \def\curraddr##1##2{\begingroup
    \@ifnotempty{##2}{\nobreak\noindent\curraddrname
      \@ifnotempty{##1}{, \ignorespaces##1\unskip}\/:\space
      ##2\par}\endgroup}%
  \def\email##1##2{\begingroup
    \@ifnotempty{##2}{\nobreak\noindent E-mail address%
      \@ifnotempty{##1}{, \ignorespaces##1\unskip}\/:\space
      \ttfamily##2\par}\endgroup}%
  \def\urladdr##1##2{\begingroup
    \def~{\char`\~}%
    \@ifnotempty{##2}{\nobreak\noindent\urladdrname
      \@ifnotempty{##1}{, \ignorespaces##1\unskip}\/:\space
      \ttfamily##2\par}\endgroup}%
  \addresses
  \endgroup
  \global\let\addresses=\@empty
}
\def\@setabstracta{%
  \ifvoid\abstractbox
  \else
    \skip@20pt \advance\skip@-\lastskip
    \advance\skip@-\baselineskip \vskip\skip@
    \box\abstractbox
    \prevdepth\z@ 
    \vskip-22pt
  \fi
}
\renewenvironment{abstract}{%
  \ifx\maketitle\relax
    \ClassWarning{\@classname}{Abstract should precede
      \protect\maketitle\space in AMS document classes; reported}%
  \fi
  \global\setbox\abstractbox=\vtop \bgroup
    \normalfont\small
    \list{}{\labelwidth\z@
      \leftmargin0pc \rightmargin\leftmargin
      \listparindent\normalparindent \itemindent\z@
      \parsep\z@ \@plus\p@
      
    }%
    \item[\hskip\labelsep\bfseries\abstractname.]%
}{%
  \endlist\egroup
  \ifx\@setabstract\relax \@setabstracta \fi
}

\def\ps@headings{\ps@empty
  \def\@evenhead{%
    \setTrue{runhead}%
    \normalfont\scriptsize
    \rlap{\thepage}\hfill
    \def\thanks{\protect\thanks@warning}%
    \leftmark{}{}}%
  \def\@oddhead{%
    \setTrue{runhead}%
    \normalfont\scriptsize
    \def\thanks{\protect\thanks@warning}%
    \rightmark{}{}\hfill \llap{\thepage}}%
  \let\@mkboth\markboth
}\ps@headings

\def\section{\@startsection{section}{1}%
  \z@{-1.4\linespacing\@plus-.5\linespacing}{.8\linespacing}%
  {\normalfont\bfseries\Large}}
\def\subsection{\@startsection{subsection}{2}%
  \z@{-.8\linespacing\@plus-.3\linespacing}{.5\linespacing\@plus.2\linespacing}%
  {\normalfont\bfseries\large}}
\def\subsubsection{\@startsection{subsubsection}{3}%
  \z@{.7\linespacing\@plus.2\linespacing}{-1.5ex}%
  {\normalfont\bfseries}}
\def\@secnumfont{\bfseries}

\renewcommand\contentsnamefont{\bfseries\large}
\def\@starttoc#1#2{\begingroup
  \setTrue{#1}%
  \par\removelastskip\vskip\z@skip
  \@startsection{}\@M\z@{\linespacing\@plus\linespacing}%
    {.5\linespacing}{
      \contentsnamefont}{#2}%
  \ifx\contentsname#2%
  \else \addcontentsline{toc}{section}{#2}\fi
  \makeatletter
  \@input{\jobname.#1}%
  \if@filesw
    \@xp\newwrite\csname tf@#1\endcsname
    \immediate\@xp\openout\csname tf@#1\endcsname \jobname.#1\relax
  \fi
  \global\@nobreakfalse \endgroup
  \addvspace{32\p@\@plus14\p@}%
  \let\tableofcontents\relax
}
\def\contentsname{Contents}
\def\l@section{\@tocline{1}{.5ex}{0mm}{5pc}{}}
\def\l@subsection{\@tocline{2}{0pt}{2em}{5pc}{}}
\makeatother

\fi

\def\to{\mathchoice{\longrightarrow}{\rightarrow}{\rightarrow}{\rightarrow}}
\makeatletter
\newcommand{\shortxra}[2][]{\ext@arrow 0359\rightarrowfill@{#1}{#2}}
\def\longrightarrowfill@{\arrowfill@\relbar\relbar\longrightarrow}
\newcommand{\longxra}[2][]{\ext@arrow 0359\longrightarrowfill@{#1}{#2}}
\renewcommand{\xrightarrow}[2][]{\mathchoice{\longxra[#1]{#2}}%
  {\shortxra[#1]{#2}}{\shortxra[#1]{#2}}{\shortxra[#1]{#2}}}
\makeatother

\makeatletter
\def\Nopagebreak{\@nobreaktrue\nopagebreak}

\makeatother

\newtheorem{theorem}{Theorem}[section]
\newtheorem{theoremalpha}{Theorem}
\newtheorem{proposition}[theorem]{Proposition}
\newtheorem{corollary}[theorem]{Corollary}
\newtheorem{corollaryalpha}[theoremalpha]{Corollary}
\newtheorem{lemma}[theorem]{Lemma}

\theoremstyle{definition}
\newtheorem{definition}[theorem]{Definition}

\newtheorem{remark}[theorem]{Remark}

\newtheoremstyle{theorem-giventitle}
        {}{}              
        {\itshape}                      
        {}                              
        {\bfseries}                     
        {.}                             
        { }                             
        {\thmnote{\bfseries#3}}
\theoremstyle{theorem-giventitle}
\newtheorem{theorem-named}{}
\newtheoremstyle{definition-giventitle}
        {}{}              
        {}                      
        {}                              
        {\bfseries}                     
        {.}                             
        {.7em}                             
        {\thmnote{\bfseries#3}}
\theoremstyle{definition-giventitle}
\newtheorem{question-named}{}
\newtheorem{step-named}{}
\newtheorem{claim-named}{}

\newenvironment{subproof}{\begin{trivlist}\item}{\end{trivlist}}

\numberwithin{equation}{section}

\def\Z{\mathbb{Z}}
\def\Q{\mathbb{Q}}

\def\C{\mathbb{C}}

\def\cA{\mathcal{A}}
\def\cC{\mathcal{C}}

\def\cR{\mathcal{R}}

\def\Ker{\operatorname{Ker}}
\def\Coker{\operatorname{Coker}}
\def\Im{\operatorname{Im}}
\def\Hom{\mathrm{Hom}}

\def\Ext{\mathrm{Ext}}

\def\id{\mathrm{id}}

\def\Aut{\mathrm{Aut}}

\def\sm{\setminus}
\makeatletter
\def\tpmod#1{{\@displayfalse\pmod{#1}}}
\makeatother

\def\mathbinover#1#2{\mathbin{\mathop{#1}\limits_{#2}}}

\def\cupover#1{\mathbinover{\cup}{#1}}
\def\colim{\operatorname*{colim}}

\def\setminus{\smallsetminus}

\def\isomto{\mathrel{\smash{\xrightarrow{\smash{\lower.4ex\hbox{$\scriptstyle{\cong}$}}}}}}
\def\isomfrom{\mathrel{\smash{\xleftarrow{\smash{\lower.4ex\hbox{$\scriptstyle{\cong}$}}}}}}

\def\sbmatrix#1{\left[\begin{smallmatrix}#1\end{smallmatrix}\right]}

\def\cyc#1{\tfrac{1}{#1}\Z/\Z}
\def\frt#1{(\tfrac{1}{#1}\Z)^2}

\begin{document}

\vspace*{0mm}

\title%
{Transfinite Milnor invariants for 3-manifolds}

\author{Jae Choon Cha}
\address{
  Center for Research in Topology\\
  POSTECH\\
  Pohang 37673\\
  Republic of Korea
}
\email{jccha@postech.ac.kr}

\author{Kent E. Orr}
\address{Department of Mathematics\\
  Indiana University\\
  Bloomington, Indiana 47405
  \\USA}
\email{korr@indiana.edu}

\def\subjclassname{\textup{2010} Mathematics Subject Classification}
\expandafter\let\csname subjclassname@1991\endcsname=\subjclassname
\expandafter\let\csname subjclassname@2000\endcsname=\subjclassname
\subjclass{%
}

\begin{abstract}
  In his 1957 paper, John Milnor introduced link invariants which measure the
  homotopy class of the  longitudes of a link relative to the lower central
  series of the link group.  Consequently, these invariants determine the lower
  central series quotients of the link group. This work has driven decades of
  research with profound influence.  One of Milnor's original problems remained
  unsolved: to extract similar invariants from the transfinite lower central
  series of the link group. We reformulate and extend Milnor's invariants in the
  broader setting of $3$-manifolds, with his original invariants as special
  cases. We present a solution to Milnor's problem for general 3-manifold
  groups, developing a theory of transfinite invariants and realizing nontrivial
  values.
 \end{abstract}

\maketitle

\setcounter{tocdepth}{1}
\tableofcontents

\section{Introduction}
\label{section:introduction}

In John Milnor's 1954 Ph.D.\ thesis~\cite{Milnor:1957-1}, he introduced link
invariants obtained from the lower central series of the fundamental group.
Milnor's work vastly extended the classical linking number, and has influenced
decades of fundamental research.

Roughly speaking, Milnor's invariants inductively measure whether the
fundamental group of the exterior of a given link has the same lower central
series quotients as that of the free group~\cite{Milnor:1957-1}.  Another key
feature of Milnor's invariant, due to Stallings~\cite{Stallings:1965-1}, is
invariance under link concordance, and more generally under homology cobordism
of the link exterior. Invariance under homology cobordism  seeds a fundamental
connection between Milnor's invariants and the topology of 4-manifolds.

Although seldom noted, the first part of Milnor's paper~\cite{Milnor:1957-1}
concerns fundamental groups of exteriors of links in an {\em arbitrary
3-manifold}, while the latter part of Milnor's paper, as well as most subsequent
research of others, focuses on the special case of links in~$S^3$.

The following problems posed by Milnor in~\cite{Milnor:1957-1} have remained
unsolved for more than 60 years.

\begin{question-named}
  [Milnor's Problem \textmd{~\cite[p.~52, Problem~(b)]{Milnor:1957-1}}]

  Find a method of attacking the \emph{transfinite} lower central series
  quotients  and extracting information from it.
\end{question-named}

That is, develop a transfinite lower central series version of
Milnor's invariants which contains non-vacuous information.

Recall that the \emph{transfinite lower central series} of a group $G$
consists of subgroups $G_\kappa$ indexed by ordinals $\kappa$ and defined by
\[
  G_\kappa = \begin{cases}
    G &\text{if $\kappa=1$,}\\
    [G,G_{\kappa-1}] &\text{if $\kappa>1$ is a discrete ordinal,}\\
    \bigcap_{\lambda<\kappa} G_\lambda &\text{if $\kappa$ is a limit ordinal.}
  \end{cases}
\]

We acknowledge past progress toward a solution to Milnor's problem. The first
viable candidate for a transfinite invariant was given in work of the second
author~\cite{Orr:1987-1}.  He presented a reformulation of the original Milnor
link invariants by introducing a homotopy theoretic approach.
Papers~\cite{Orr:1987-1,Igusa-Orr:2001-1} answered numerous problems
from~\cite{Milnor:1957-1}, including the realizabilty and independence of Milnor
invariants.  But Orr's ``transfinite'' invariant of links continues to
resist computation.  (See recent progress by E. D. Farjoun and R. Mikhailov
in~\cite{Farjoun-Mikhailov:2018-1}.)

J.~Levine refined Orr's transfinite invariant by developing the fundamental
notion of ``algebraic closure of groups''~\cite{Levine:1989-2, Levine:1989-1}.
This arose, in part, from harvesting key insights from work of M.
Gutierrez~\cite{Gutierrez:1979-1} and P.
Vogel~\cite{Vogel:1978-1,LeDimet:1988-1}.  With his breakthrough, Levine proved
further realization and geometric characterization results. Unfortunately,
Levine's refinement resists computation as well.

In particular, it remains open whether the invariants in~\cite{Orr:1987-1,
Levine:1989-2, Levine:1989-1} always vanish for links with vanishing classical
Milnor invariants!

\subsubsection*{Our contribution}
\label{section:contribution}

In this paper, we develop new families of \emph{transfinite} invariants for
closed, orientable \emph{$3$-manifolds}.  For one family of these invariants we
find striking parallels to Milnor's link invariants, leading
us to name that family of invariants \emph{Milnor invariants of $3$-manifolds}.
The Milnor invariants we introduce are indexed by arbitrary ordinal numbers
called the \emph{length} of the invariant.  This allows one to extend the
integer grading in Milnor's original work.  Our invariants include classical
Milnor invariants as a special case.

We show that our invariants are highly nontrivial even at infinite ordinals.
Thus, \emph{we view these invariants as presenting a solution to Milnor's
problem within the broad context of oriented closed $3$-manifolds.}

Indeed, we define four closely related invariants.  The invariants we call the
Milnor invariants is denoted by $\bar\mu_\kappa(M)$, where $\kappa$ is the
length. The invariant $\bar\mu_\kappa(M)$ has the following features.

\begin{enumerate}[(i)]
  \item\label{item:feature-lcsq}
  \emph{Determination of lower central series quotients}: $\bar\mu_\kappa(M)$
  inductively determine the isomorphism classes of the lower central series
  quotients, as do Milnor's link invariants.  Furthermore, this inductive
  process extends to transfinite ordinals.
  \item\label{item:feature-H-cob}
  \emph{Homology cobordism invariance}: $\bar\mu_\kappa(M)$ is invariant under
  homology cobordism, as are Milnor's link invariants.
  \item\label{item:feature-link}
  \emph{Specialization to Milnor's link invariants}: $\bar\mu_\kappa(M)$ with
  finite $\kappa$ determines Milnor's link invariants, when $M$ is the
  zero-surgery on a link in~$S^3$.
  \item\label{item:feature-grope}
  \emph{Obstructions to gropes:}  like those of links, $\bar\mu_\kappa(M)$ is an
  obstruction to building gropes.  Moreover, this extends to the transfinite
  length case, using an appropriate notion of transfinite gropes.
  \item\label{item:feature-realization}
  \emph{Realization:} $\bar\mu_\kappa(M)$ lives in a set
  $\cR_\kappa(\Gamma)/{\approx}$, whose elements are explicitly characterized.
  Every element in $\cR_\kappa(\Gamma)/{\approx}$ is realized as
  $\bar\mu_\kappa(M)$ for some closed 3-manifold~$M$.
\end{enumerate}

This shows that many {\em fundamental characterizing properties} of Milnor's
link invariants generalize to our 3-manifold Milnor invariants, thereby
extending Milnor's theory across all ordinals and $3$-manifolds.

Using realizability from~\ref{item:feature-realization},  we show the
aforementioned result that the transfinite theory is \emph{highly nontrivial
even at infinite ordinals}---we exhibit infinitely many explicit 3-manifolds $M$
with vanishing $\bar\mu_\kappa(M)$ for all finite $\kappa$ but have
non-vanishing, pairwise distinct $\bar\mu_\omega(M)$ for the first transfinite
ordinal~$\omega$.

We also define and study a ``universal'' transfinite invariant, which
generalizes Levine's link invariant~\cite{Levine:1989-1} over algebraic closures
to the case of 3-manifolds.  We prove that this universal invariant is highly
nontrivial, even for 3-manifolds for which all transfinite Milnor invariants
vanish.  As mentioned earlier, for links, whether Levine's invariant can be
non-zero remains open.

We define two additional invariants central to our paper.  The following section
describes all four invariants and provides precise statements of our main
results as well as
applications,~\ref{item:feature-lcsq}--\ref{item:feature-realization}.

The new results of this paper, especially the framework of transfinite
invariants, opens multiple avenues for future research.  We discuss a small
portion of these, including potential applications to link concordance and to
Whitney towers, at the end of the paper.

\subsection*{Acknowledgements}

The first named author was partly supported by NRF grant 2019R1A3B2067839. The
second named author was partly supported by Simons Foundation Grants 209082 and
430351.  The second author gratefully acknowledges support from SFB 1085 `Higher
Invariants' funded by the Deutsche Forschungsgemeinschaft DFG,  University of
Regensburg.

Subsequent to the development of our theory, Sergei Ivanov and Roman Mikhailov
have begun studying the Bousfield-Kan completion of
$3$-manifolds~\cite{Ivanov-Mikhailov:2019-1}. Their work seems to relate
mysteriously with this paper.  \emph{Their result inspired our use of the examples
$M_k$ in Section~\ref{section:modified-tb}.}  We thank Sergei and Roman for
bringing these examples to our attention.

We thank the referee who read our manuscript with keen understanding and insight, and offered numerous valuable comments.

\section{Statements of main results}
\label{section:main-results}

In this section, we describe our main results.  In
Section~\ref{subsection:intro-localization}, we provide a quick review of
homology localization.  In
Sections~\ref{subsection:definition-invariant}--\ref{subsection:results-universal-inv},
we define four invariants for 3-manifolds and present their key features,
including 3-manifold Milnor invariants.  In
Sections~\ref{subsection:results-torus-bundle}--\ref{subsection:results-nontrivial-milnor},
we discuss examples exhibiting rich information extracted from these invariants.

Throughout this paper, we consider only compact oriented manifolds unless stated
otherwise explicitly. The notation $H_*(-)$ denotes homology with integral
coefficients.

\subsection{Homology localization of groups}
\label{subsection:intro-localization}

We begin with a brief introduction to the role of \emph{locally finite homology
localization} of groups, also known as \emph{algebraic closure}. Readers who are
already familiar with this might prefer to skip to the last paragraph (or the
last sentence) of this subsection.

The invariance of the original Milnor invariants under concordance and homology
cobordism follows from a well known result of Stallings that the lower central
quotients $\pi_1(-)/\pi_1(-)_k$ are preserved under homology equivalence of
spaces for all $k<\infty$~\cite{Stallings:1965-1}.  (See also
\cite{Casson:1975-1}.) By contrast, the transfinite lower central quotient
$\pi_1(-)/\pi_1(-)_\kappa$ is not invariant under homology cobordism (or
homology equivalence).  For instance, this follows from an example of
Hillman~\cite[p.~56--57]{Hillman:1981-1}.

To extract information invariant under concordance of links and homology
cobordism of 3-manifolds, we follow an approach suggested in work of
Vogel~\cite{Vogel:1978-1} and Levine~\cite{Levine:1989-1,Levine:1989-2}, using
homology localization of groups.

In general, localization is defined for a given collection $\Omega$ of morphisms
in a category~$\cC$.  Briefly, a localization designates a functor $E\colon\cC
\to \cC$ equipped with a natural transformation $A=1_\cC(A) \to E(A)$ such that

\begin{enumerate}[(i)]
\item $E(\phi)$ is an equivalence for all morphisms $\phi$ in~$\Omega$, and
\item $E$ is universal (initial) among those satisfying~(i).
\end{enumerate}

A precise definition will be stated in Section~\ref{section:localization}.
Observe that a homology equivalence $X\to Y$ of spaces gives rise to a group
homomorphism $\pi_1(X)\to \pi_1(Y)$ which induces an isomorphism on $H_1(-)$ and
an epimorphism on $H_2(-)$.  We call a group homomorphism with this homological
property \emph{2-connected}. Due to an unpublished manuscript of
Vogel~\cite{Vogel:1978-1} and an independent approach of
Levine~\cite{Levine:1989-2,Levine:1989-1}, there exists a localization, in the
category of groups, for the collection of 2-connected homomorphisms $\phi\colon
A\to B$ with $A$ and $B$ finitely presented.  (For those who are familiar with
Bousfield's $H\Z$ localization~\cite{Bousfield:1974-1,Bousfield:1975-1}, we
remark that the key difference between Vogel-Levine from the $H\Z$ case is the
finite presentability of $A$ and $B$, which turns out to provide a crucial
advantage for applications to compact manifolds.)  

We observe that Levine's version in~\cite{Levine:1989-1} differs slightly from
what we use here.  With applications to link concordance in mind, he adds the
additional requirement that the image $\phi(A)$ normally generates~$B$. (This
reflects the property that meridians for a link normally generate the link group
$\pi_1(S^3\sm L)$.)  Levine was aware of both notions of localization.  The
first detailed exposition of what we use can be found in~\cite{Cha:2004-1}. We
denote this homology localization by $G\to \widehat G$ in this paper.  See
Section~\ref{section:localization} for more details.

The following two properties of the homology localization $\widehat G$ are
essential for our purpose.  For brevity, denote the transfinite lower central
subgroup $(\widehat G)_\kappa$ by~$\widehat G_\kappa$.
\begin{enumerate}[(i)]
  \item A 2-connected homomorphism $G \to \Gamma$ between finitely presented
  groups induces an isomorphism on~$\widehat G/\widehat G_\kappa \to \widehat
  \Gamma/\widehat \Gamma_\kappa$ for every ordinal~$\kappa$.
  \item When $G$ is finitely presented, $\widehat G/\widehat G_k\cong G/G_k$ for
  all $k$ finite.
\end{enumerate}

See Section~\ref{section:localization}, especially
Corollary~\ref{corollary:localization-basic-consequences}.

So, $\widehat G/\widehat G_\kappa$ is a transfinite generalization of the finite
lower central quotients $G/G_k$, which remains invariant under homology
cobordism of compact manifolds for every ordinal~$\kappa$. In this regard,
$\widehat G/\widehat G_\kappa$ is a correct generalization of $G/G_k$ for
studies related to homology cobordism, concordance, and disk embedding in
dimension 4.  From now on, ``transfinite lower central quotient'' in this paper
means $\widehat G/\widehat G_\kappa$, instead of~$G/G_\kappa$, where $\widehat
G$ is the integer coefficient Vogel-Levine homology localization as constructed
in~\cite{Cha:2004-1}.

\subsection{Definition of the transfinite invariants}
\label{subsection:definition-invariant}

Milnor's original work~\cite{Milnor:1957-1} compares the lower central quotients
$\pi/\pi_k$ of a link group $\pi=\pi_1(S^3\sm L)$ with that of the trivial link,
namely the free nilpotent quotients~$F/F_k$, inductively on~$k$. We provide a
relative theory, comparing the lower central quotients of other $3$-manifolds to
that of a fixed $3$-manifold we choose arbitrarily. For instance, when studying
links, we can begin with $0$-surgery on a nontrivial link, and compare its lower
central series quotients to that of other links. By replacing a $3$-manifold
group with its localization, we extend this theory throughout the {\em
transfinite} lower central series.

Fix a closed 3-manifold $Y$, which will play the role analogous to the trivial
link in Milnor's work.  Denote $\Gamma=\pi_1(Y)$.  Suppose $M$ is another closed
3-manifold with $\pi=\pi_1(M)$.  Our invariants compare the transfinite lower
central quotients with that of~$\Gamma$.  

Indeed, we define and study four invariants of~$M$: 
\begin{enumerate}
  \item a $\theta$-invariant $\theta_\kappa(M)$ defined as a 3-dimensional
  homology class,
  \item a reduced version of the $\theta_\kappa(M)$ living in a certain
  ``cokernel,''
  \item 3-manifold Milnor invariant $\bar\mu_\kappa(M)$, and
  \item a universal $\theta$-invariant $\widehat\theta(M)$. 
\end{enumerate}
The first three invariants are indexed by arbitrary ordinals~$\kappa$.  In
Sections~\ref{subsection:definition-invariant}--\ref{subsection:results-universal-inv},
we describe the definitions and state their key features.

We begin with $\theta_\kappa(M)$.  Fix an arbitrary ordinal~$\kappa$, and
suppose the 3-manifold group $\pi$ admits an isomorphism $f\colon
\widehat\pi/\widehat\pi_\kappa \isomto \widehat\Gamma/\widehat\Gamma_\kappa$.
The goal is to determine whether the next stage quotient
$\widehat\pi/\widehat\pi_{\kappa+1}$ is isomorphic
to~$\widehat\Gamma/\widehat\Gamma_{\kappa+1}$.

The following definition is motivated from work of the second
author~\cite{Orr:1989-1} and Levine~\cite{Levine:1989-2, Levine:1989-1}.  (See also~\cite{Heck:2009-1}.)

\begin{definition}
  \label{definition:theta-kappa}
  Let $\theta_\kappa(M,f)\in H_3(\widehat\Gamma/\widehat\Gamma_\kappa)$ be
  the image of the fundamental class $[M]\in H_3(M)$, under the
  composition
  \[
    H_3(M) \to H_3(\pi) \to H_3(\widehat\pi) \to
    H_3(\widehat\pi/\widehat\pi_\kappa) \xrightarrow[\cong]{f_*}
    H_3(\widehat\Gamma/\widehat\Gamma_\kappa).
  \]
  We call $\kappa$ the \emph{length} of the invariant~$\theta_\kappa$.
  We will generally write $\theta_\kappa(M)$, in order to avoid an excess of notation, and will write $\theta(M,f)$ when we need to emphasize the choice of~$f$.
\end{definition}

The value of $\theta_\kappa(M)$ in $H_3(\widehat\Gamma/\widehat\Gamma_\kappa)$
depends on the choice of $f\colon \widehat\pi/\widehat\pi_\kappa \isomto
\widehat\Gamma/\widehat\Gamma_\kappa$, and could be denoted $\theta_\kappa(M,
f).$  We choose to omit the reference to $f$ to simplify notation, but we
emphasize to the reader that this indeterminacy is often nontrivial.

If we choose to remove indeterminacy, we can do so by comparing possible choices
for~$f$.  Doing so, we obtain an invariant of $3$-manifolds defined from
$\theta_\kappa(M)$ by taking the value of $\theta_\kappa(M)$ in the orbit space
$H_3(\widehat\Gamma/\widehat\Gamma_\kappa) /
\Aut(\widehat\Gamma/\widehat\Gamma_\kappa)$ of the action of automorphisms of
$\widehat\Gamma/\widehat\Gamma_\kappa$, thus providing an alternative definition
of $\theta_\kappa(M)$ which is independent of the choice of~$f$.  It turns out
that both versions (with and without indeterminacy) are useful, as we discuss
below.  We will refer to these invariants as the
\emph{$\theta_\kappa$-invariants of $M$} (relative to~$\Gamma$).

\subsection{Invariance under homology cobordism}
\label{subsection:results-homology-cob}

\begin{theoremalpha}
  \label{theorem:results-homology-cob}
  The class $\theta_\kappa(M)$ is invariant under homology cobordism.  More
  precisely, if $M$ and $N$ are homology cobordant 3-manifolds with
  $\pi=\pi_1(M)$ and $G=\pi_1(N)$, then for every ordinal $\kappa$,
  the following hold.

  \begin{enumerate}

    \item\label{item:homology-cob-inv-lcsq}
    There is an isomorphism $\phi\colon \widehat{G}/\widehat{G}_\kappa \isomto
    \widehat{\pi}/\widehat{\pi}_\kappa$.  Consequently $\theta_\kappa(N)$ is
    defined if and only if $\theta_\kappa(M)$ is defined.

    \item\label{item:homology-cob-inv-theta}
    If $f\colon \widehat{\pi}/\widehat{\pi}_\kappa \isomto \widehat\Gamma/\widehat\Gamma_\kappa$ is an isomorphism, then $\theta_\kappa(M,f) = \theta_\kappa(N,f\circ\theta)$ in $H_3(\widehat\Gamma/\widehat\Gamma_\kappa)$, where $\phi$ is the isomorphism in~\ref{item:homology-cob-inv-lcsq}.

    \item\label{item:homology-cob-inv-theta-mod-aut}
    If $\theta_\kappa(M)$ and $\theta_\kappa(N)$ are defined using
    arbitrary isomorphisms $\widehat\pi/\widehat\pi_\kappa \isomto
    \widehat\Gamma/\widehat\Gamma_\kappa$ and $\widehat G/\widehat G_\kappa
    \isomto \widehat\Gamma/\widehat\Gamma_\kappa$, then 
    $\theta_\kappa(M)=\theta_\kappa(N)$ in
    $H_3(\widehat\Gamma/\widehat\Gamma_\kappa) /
    \Aut(\widehat\Gamma/\widehat\Gamma_\kappa)$.

  \end{enumerate}
\end{theoremalpha}

We remark that the isomorphism $\phi$ in~\ref{item:homology-cob-inv-lcsq}
and~\ref{item:homology-cob-inv-theta} depends on a choice of a homology
cobordism.  The statement~\ref{item:homology-cob-inv-theta-mod-aut} provides
an invariant independent of choice.

The proof of Theorem~\ref{theorem:results-homology-cob} is given in
Section~\ref{section:homology-cobordism-invariance}.  It is a
straightforward consequence of the definition of the invariant and basic
properties of homology localization.

\subsection{Determination of transfinite lower central quotients}
\label{subsection:results-determination}

Define the set of homology classes which are realizable by $\theta_\kappa$ to be
\begin{equation}
  \cR_\kappa(\Gamma) = \bigg\{
    \theta\in H_3(\widehat\Gamma/\widehat\Gamma_\kappa) \,\,\bigg|\,\,
  \begin{tabular}{@{}c@{}}
    $\theta=\theta_\kappa(M)$ for some closed 3-manifold $M$ 
    \\
    equipped with $\widehat{\pi_1(M)}/\widehat{\pi_1(M)}_\kappa \isomto \widehat\Gamma/\widehat\Gamma_\kappa$
  \end{tabular}
  \bigg\}.
  \label{equation:realizable-classes}
\end{equation}
Not all homology classes are necessarily realizable.  That is,
$\cR_\kappa(\Gamma) \ne H_3(\widehat\Gamma/\widehat\Gamma_\kappa)$ in general.
Nor is $\cR_\kappa(\Gamma)$ necessarily a subgroup.  See
Theorem~\ref{theorem:main-realization} below, and
Sections~\ref{section:tb-finite-computation}
and~\ref{section:tb-transfinite-computation}.

Nonetheless, one can
straightforwardly verify that the projection
$\widehat\Gamma/\widehat\Gamma_{\kappa+1} \to
\widehat\Gamma/\widehat\Gamma_{\kappa}$ induces a function
$\cR_{\kappa+1}(\Gamma) \to \cR_\kappa(\Gamma)$.  Although
$\Coker\{\cR_{\kappa+1}(\Gamma) \to \cR_\kappa(\Gamma)\}$ is not well defined in
the usual way because of the lack of a natural group structure, we can define a
notion of \emph{vanishing in the cokernel} as follows: 

\begin{definition}
  \label{definition:vanishing-in-coker}
  We say that a class $\theta\in \cR_\kappa(\Gamma)$ \emph{vanishes in}
  $\Coker\{\cR_{\kappa+1}(\Gamma) \to \cR_\kappa(\Gamma)\}$ if $\theta$ lies in
  the image of $\cR_{\kappa+1}(\Gamma) \to \cR_\kappa(\Gamma)$.
\end{definition}

That is, the invariant $ \theta_\kappa(M)$ vanishes in the cokernel if there is
a closed $3$-manifold $N$ for which $\theta_{\kappa + 1}(N)$ is defined
(relative to~$\Gamma$) and the image of $\theta_{\kappa + 1}(N)$ is
$\theta_\kappa(M)$ under the quotient induced homomorphism below.
\[
\begin{tikzcd}[column sep=-1ex]
  \theta_{\kappa+1}(N)\arrow[d,mapsto]
  & {\in} & \cR_{\kappa+1}(\Gamma)  \ar[d]
  & {\subset} & H_3(\widehat{\Gamma}/\widehat{\Gamma}_{\kappa + 1})
  \\
  \theta_\kappa(M)
  & {\in} & \cR_\kappa(\Gamma)
  & {\subset} & H_3(\widehat{\Gamma}/\widehat{\Gamma}_\kappa)
\end{tikzcd}
\]

We now state the second main result.

\begin{theoremalpha}
  \label{theorem:main-determination-lcsq-lift}
  Suppose $M$ is a closed 3-manifold and $\pi=\pi_1(M)$ is endowed with
  an isomorphism $f\colon \widehat\pi/\widehat\pi_\kappa \isomto
  \widehat\Gamma/\widehat\Gamma_\kappa$.  Then the following are equivalent.
  \begin{enumerate}
    \item There exists a lift $\widehat\pi/\widehat\pi_{\kappa+1}
    \isomto \widehat\Gamma/\widehat\Gamma_{\kappa+1}$ of $f$ which is an
    isomorphism.
    \item The invariant $\theta_\kappa(M)$ vanishes in
    $\Coker\{\cR_{\kappa+1}(\Gamma) \to \cR_\kappa(\Gamma)\}$.
  \end{enumerate}
\end{theoremalpha}

As stated in Theorem~\ref{theorem:main-determination-lcsq-nonlift} below, it
is possible to remove the restriction in
Theorem~\ref{theorem:main-determination-lcsq-lift} that the next stage
isomorphism $\widehat\pi/\widehat\pi_{\kappa+1} \cong
\widehat\Gamma/\widehat\Gamma_{\kappa+1}$ is a lift, by taking the value of
$\theta_\kappa(M)$ modulo the action of
$\Aut(\widehat\Gamma/\widehat\Gamma_\kappa)$, which is independent of the choice
of~$\widehat\pi/\widehat\pi_\kappa \isomto
\widehat\Gamma/\widehat\Gamma_\kappa$. To state the result, we use the following
definition:  a class $\theta$ \emph{vanishes in} $\Coker\{\cR_{\kappa+1}(\Gamma)
\to \cR_\kappa(\Gamma)/\Aut(\widehat\Gamma/\widehat\Gamma_\kappa)\}$ if it lies
in the image of the composition $\cR_{\kappa+1}(\Gamma) \to \cR_\kappa(\Gamma)
\to \cR_\kappa(\Gamma)/\Aut(\widehat\Gamma/\widehat\Gamma_\kappa)$.

\begin{theoremalpha}
  \label{theorem:main-determination-lcsq-nonlift}
  Suppose $M$ is a closed 3-manifold with $\pi=\pi_1(M)$ which admits an
  isomorphism $\widehat\pi/\widehat\pi_\kappa \isomto
  \widehat\Gamma/\widehat\Gamma_\kappa$. Then the following are equivalent.
  \begin{enumerate}
    \item\label{item:iso-kappa+1-nonlift} $\widehat\pi/\widehat\pi_{\kappa+1}$
    is isomorphic to $\widehat\Gamma/\widehat\Gamma_{\kappa+1}$ \textup{(}via
    any isomorphism not required to be a lift\textup{)}.
    \item\label{item:theta-kappa-vanishes-mod-aut} The invariant
    $\theta_\kappa(M)$ vanishes in $\Coker\{\cR_{\kappa+1}(\Gamma) \to
    \cR_{\kappa}(\Gamma) / \Aut(\widehat\Gamma/\widehat\Gamma_\kappa)\}$.  
  \end{enumerate}
\end{theoremalpha}

The proof is straightforward, using
Theorem~\ref{theorem:main-determination-lcsq-lift}\@.

\begin{proof}[Proof of Theorem~\ref{theorem:main-determination-lcsq-nonlift}]
  Suppose $g\colon \widehat\pi/\widehat\pi_{\kappa+1} \isomto
  \widehat\Gamma/\widehat\Gamma_{\kappa+1}$ is an isomorphism.  
  Let $g_0\colon \widehat\pi/\widehat\pi_\kappa \isomto
  \widehat\Gamma/\widehat\Gamma_\kappa$ be the isomorphism induced by~$g$, and consider $\theta_{\kappa+1}(M)=\theta_{\kappa+1}(M,g)$ and $\theta_\kappa(M)=\theta_\kappa(M,g_0)$.
  Then
  $\theta_\kappa(M)$ is the image of $\theta_{\kappa+1}(M)$ under
  $\cR_{\kappa+1}(\Gamma) \to \cR_{\kappa}(\Gamma)$.  This shows
  (1)${}\Rightarrow{}$(2).  
  
  For the converse, suppose the invariant $\theta_\kappa(M)=\theta_\kappa(M,f)$ vanishes in the cokernel of $\cR_{\kappa+1}(\Gamma) \to \cR_\kappa(\Gamma)/\Aut(\widehat\Gamma/\widehat\Gamma_\kappa)$ where $f\colon \widehat\pi/\widehat\pi_\kappa \isomto \widehat\Gamma/\widehat\Gamma_\kappa$ is an arbitrary isomorphism.
  By composing
  an automorphism on $\widehat\Gamma/\widehat\Gamma_\kappa$ with $f$, we may
  assume that $\theta_\kappa(M)$ lies in the image of $\cR_{\kappa+1}(\Gamma)$.
  By Theorem~\ref{theorem:main-determination-lcsq-lift}, there is a lift
  $\widehat\pi/\widehat\pi_{\kappa+1} \isomto
  \widehat\Gamma/\widehat\Gamma_{\kappa+1}$ of~$f$.
\end{proof}

The notion of vanishing in the cokernel generalizes to an equivalence relation
$\sim$ on the set $\cR_\kappa(\Gamma)$, which we describe below.  Recall that if
$\theta\in \cR_\kappa(\Gamma)$, we have $\theta=\theta_\kappa(M)$ for some
closed 3-manifold $M$ equipped with an isomorphism $f\colon
\widehat{\pi_1(M)}/\widehat{\pi_1(M)}_\kappa \isomto
\widehat\Gamma/\widehat\Gamma_\kappa$.  Let $I_\theta$ be the image of the
composition
\[
  \cR_{\kappa+1}(\pi_1(M)) \to \cR_{\kappa}(\pi_1(M))
  \xrightarrow[f_*]{\cong} \cR_{\kappa}(\Gamma).
\]
We show that $\{I_\theta \mid \theta\in \cR_\kappa(\Gamma)\}$ is a partition of
the set~$\cR_\kappa(\Gamma)$ in
Lemma~\ref{lemma:partition-on-realizable-classes}. Consider the associated
equivalence relation:

\begin{definition}
  \label{definition:equiv-rel-realzable-classes}
  Define $\sim$ on~$\cR_\kappa(\Gamma)$ by $\theta\sim\theta'$ if $\theta' \in
  I_\theta$.
\end{definition}

We prove the following result in
Section~\ref{subsection:proof-main-determination-rel}.

\begin{corollaryalpha}
  \label{corollary:main-determination-equiv-rel}
  Suppose $M$ and $N$ are closed 3-manifolds with $\pi=\pi_1(M)$ and
  $G=\pi_1(N)$, which are equipped with isomorphisms
  $\widehat\pi/\widehat\pi_\kappa \isomto \widehat\Gamma/\widehat\Gamma_\kappa$
  and $\widehat G/\widehat G_\kappa \isomto
  \widehat\Gamma/\widehat\Gamma_\kappa$.  Then, there is an isomorphism
  $\widehat\pi/\widehat\pi_{\kappa+1} \isomto \widehat G/\widehat G_{\kappa+1}$
  which is a lift of the composition $\widehat\pi/\widehat\pi_\kappa \isomto
  \widehat\Gamma/\widehat\Gamma_\kappa \isomto \widehat G/\widehat G_\kappa$ if
  and only if $\theta_\kappa(M) \sim \theta_\kappa(N)$ in~$\cR_\kappa(\Gamma)$.
\end{corollaryalpha}

Note that for a class $\theta\in \cR_\kappa(\Gamma)$, we have $\theta \sim
\theta_\kappa(Y)$ if and only if $\theta$ vanishes in the cokernel of
$\cR_{\kappa+1}(\Gamma) \to \cR_\kappa(\Gamma)$.  Here $\theta_\kappa(Y)$ is
defined using the identity map $\widehat{\pi_1(Y)}/\widehat{\pi_1(Y)}_\kappa \to
\widehat\Gamma/\widehat\Gamma_\kappa$.  So,
Corollary~\ref{corollary:main-determination-equiv-rel} generalizes
Theorem~\ref{theorem:main-determination-lcsq-lift}.

\subsection{Milnor invariants of $3$-manifolds}
\label{subsection:results-milnor-inv}

Now we define Milnor invariants of 3-manifolds.  It combines the features of
Theorem~\ref{theorem:main-determination-lcsq-nonlift} and
Corollary~\ref{corollary:main-determination-equiv-rel} in a natural way.  Once
again, we remind the reader of our hypothesis.  We fix a 3-manifold $Y$ and let
$\Gamma = \pi_1(Y)$.  We assume that $M$ is a 3-manifold with $\pi = \pi_1(M)$,
$\kappa$ is an ordinal, and we have an isomorphism $f \colon
\widehat\pi/\widehat\pi_\kappa \isomto \widehat\Gamma/\widehat\Gamma_\kappa$.
The invariant $\theta_\kappa(M) \in \cR_\kappa(\Gamma)$ was defined in
Definition~\ref{definition:theta-kappa}.  Here $\cR_\kappa(\Gamma)  \subset
H_3(\widehat\Gamma/\widehat\Gamma_\kappa)$ is the subset of realizable classes
defined by~\eqref{equation:realizable-classes}.

The following is a coarser version of the equivalence relation~$\sim$ on
$\cR_\kappa(\Gamma)$ in Definition~\ref{definition:equiv-rel-realzable-classes}.

\begin{definition}\label{definition:milnor-equivalence-rel}
  Let $\theta$, $\theta'\in \cR_\kappa(\Gamma)$.  Write $\theta \approx \theta'$
  if there is $\gamma \in \Aut(\widehat{\Gamma}/\widehat{\Gamma}_\kappa)$ such
  that $\gamma_*(\theta') \sim \theta$. That is, choosing a 3-manifold $M$
  equipped with an isomorphism $f\colon
  \widehat{\pi_1(M)}/\widehat{\pi_1(M)}_\kappa \isomto
  \widehat{\Gamma}/\widehat{\Gamma}_\kappa$ that satisfies
  $\theta=\theta_\kappa(M)$, we have $\theta \approx \theta'$ if and only if
  there is $\gamma \in \Aut(\widehat{\Gamma}/\widehat{\Gamma}_\kappa)$ such that
  \[
    \gamma_*(\theta') \in \Im\big\{\cR_{\kappa+1}(\pi_1(M)) \rightarrow
    \cR_\kappa(\pi_1(M)) \xrightarrow[f_*]{\cong} \cR_{\kappa}(\Gamma)\big\}.
  \]
\end{definition}

Since ${\sim}$ is an equivalence relation and
$\Aut(\widehat\Gamma/\widehat\Gamma_\kappa)$ is a group, it follows that
${\approx}$ is an equivalence relation too.

\begin{definition}\label{definition:milnor-invariant}
  The \emph{Milnor invariant of length $\kappa$} for $M$ is defined by
  \[
    \bar\mu_\kappa(M) := [\theta_\kappa(M)] \in \cR_{\kappa}(\Gamma)/{\approx}.
  \]
  Here $[\theta_\kappa(M)]$ is the equivalence class
  of~$\theta_\kappa(M)\in\cR_\kappa(\Gamma)$ under~${\approx}$.
\end{definition}

We have that $\theta$ vanishes in
$\Coker\{\cR_{\kappa+1}(\Gamma) \to
\cR_{\kappa}(\Gamma)/\Aut(\widehat\Gamma/\widehat\Gamma_\kappa)\}$ in the sense
of Section~\ref{subsection:results-determination} if and only if $\theta\approx
\theta_\kappa(Y)$ in~$\cR_\kappa(\Gamma)$.  If
$\theta_\kappa(M)\approx\theta_\kappa(Y)$, we say that
\emph{$\bar\mu_\kappa(M)$ vanishes}, or \emph{$M$ has vanishing Milnor invariant
of length~$\kappa$}.

\begin{theoremalpha}
  \label{theorem:main-milnor-inv}
  Let $M$ be a 3-manifold such that
  $\widehat{\pi_1(M)}/\widehat{\pi_1(M)}_\kappa \cong
  \widehat{\Gamma}/\widehat{\Gamma}_\kappa$.  Then $\bar\mu_\kappa(M)$ is a
  well-defined homology cobordism invariant, and the following are equivalent. 
  \begin{enumerate}
    \item\label{item:milnor-inv-vanish} $\bar{\mu}_{\kappa}(M)$ vanishes.
    \item\label{item:milnor-inv-lcsq-isom}
    $\widehat{\pi_1(M)}/\widehat{\pi_1(M)}_{\kappa+1} \cong
    \widehat\Gamma/\widehat\Gamma_{\kappa+1}$ \textup{(}via any isomorphism not
    required to be a lift\textup{)}.
    \item\label{item:milnor-inv-next-defined} The invariant
    $\bar\mu_{\kappa+1}(M)$ is defined.
  \end{enumerate}
  
  In addition, for $M$ and $N$ such that
  $\widehat{\pi_1(M)}/\widehat{\pi_1(M)}_\kappa \cong
  \widehat{\pi_1(N)}/\widehat{\pi_1(N)}_\kappa \cong
  \widehat{\Gamma}/\widehat{\Gamma}_\kappa$, the following two conditions are
  equivalent.
  \begin{enumerate}[resume]
    \item\label{item:milnor-inv-equal} $\bar\mu_{\kappa}(M) =
    \bar\mu_{\kappa}(N)$ in $\cR_\kappa(\Gamma)/{\approx}$.
    \item\label{item:milnor-inv-two-lcsq-isom}
    $\widehat{\pi_1(M)}/\widehat{\pi_1(M)}_{\kappa+1} \cong
    \widehat{\pi_1(N)}/\widehat{\pi_1(N)}_{\kappa+1}$.
  \end{enumerate}
\end{theoremalpha}

\begin{proof}
  The equivalence of
  \ref{item:milnor-inv-vanish}--\ref{item:milnor-inv-next-defined} is the
  conclusion of Theorem~\ref{theorem:main-determination-lcsq-nonlift}.

  Suppose \ref{item:milnor-inv-equal} holds.  Fix an isomorphism $f\colon \widehat{\pi_1(M)}/\widehat{\pi_1(M)}_\kappa \isomto \widehat{\Gamma}/\widehat{\Gamma}_\kappa$.
  By definition, $\theta_\kappa(M,f) \sim \gamma_*(\theta_\kappa(N,g))$ for some $\gamma\in\Aut(\widehat{\Gamma}/\widehat{\Gamma}_\kappa)$ and $g\colon \widehat{\pi_1(N)}/\widehat{\pi_1(N)}_\kappa \isomto \widehat{\Gamma}/\widehat{\Gamma}_\kappa$.
  Then $\gamma_*(\theta_\kappa(N,g)) = \theta_\kappa(N,\gamma\circ g)$.
  By Corollary~\ref{corollary:main-determination-equiv-rel}, it follows that \ref{item:milnor-inv-two-lcsq-isom} holds.
  Conversely, when \ref{item:milnor-inv-two-lcsq-isom} holds, let $\phi$ be the induced isomorphism $\widehat{\pi_1(N)}/\widehat{\pi_1(N)}_{\kappa} \isomto \widehat{\pi_1(M)}/\widehat{\pi_1(M)}_{\kappa}$.
  Since $\phi$ lifts, $\theta_\kappa(M,f) \sim \theta_\kappa(N,f\circ\phi)$ by Corollary~\ref{corollary:main-determination-equiv-rel}\@. So \ref{item:milnor-inv-equal} holds.
\end{proof}

Examples showing the nontriviality of the 3-manifold $\bar\mu_\kappa$-invariant
of transfinite length are given in Section~\ref{subsection:results-nontrivial-milnor}.
See Theorem~\ref{theorem:main-tb-milnor}\@.

Section~\ref{section:classical-milnor-invariant} explains how classical Milnor
invariants are special cases of the above theory associated to finite ordinals.
(See also~\cite{Orr:1989-1,Levine:1989-1,Igusa-Orr:2001-1}.)

Section~\ref{subsection:results-transfinite-grope} below states that the
$\bar\mu_\kappa$-invariant connects to a notion of transfinite gropes, with
details in Section~\ref{section:transfinite-stallings-dwyer}.

\subsection{A transfinite tower interpretation}
\label{subsection:results-transfinite-tower}

Corollary~\ref{corollary:main-determination-equiv-rel} and
Theorem~\ref{theorem:main-milnor-inv} may be viewed as classifications of towers
of transfinite lower central quotients of 3-manifold groups.  Briefly, we
address the following problem: classify extensions of length $\kappa+1$, by
3-manifold groups, of the length $\kappa$ tower of the transfinite lower central
quotients

\begin{equation}
  \label{equation:transfinite-lcsq-tower}
  \begin{tikzcd}[sep=18pt]
    \widehat\Gamma/\widehat\Gamma_{\kappa}
    \ar[r,dash pattern=on 3pt off 2pt on .5pt off 2pt on .5pt off 2pt
      on .5pt off 2pt on 50pt]
    &
    \widehat\Gamma/\widehat\Gamma_{\omega}
    \ar[r,dash pattern=
      on .5pt off 1pt on .5pt off 1.2pt on .5pt off 1.4pt on .5pt off 1.6pt
      on .5pt off 1.8pt on .5pt off 2pt on .5pt off 2.2pt on 50pt]
    &
    \widehat\Gamma/\widehat\Gamma_{2}
    \ar[r]
    &
    \widehat\Gamma/\widehat\Gamma_{1}
    \ar[r,equal]
    &[-10pt]
    \{1\}
    \\[-10pt]
    & & \Gamma/\Gamma_2 \ar[u,equal] & \Gamma/\Gamma_1 \ar[u,equal] &
  \end{tikzcd}
\end{equation}
of a given 3-manifold group $\Gamma=\pi_1(Y)$.

To be precise, we introduce some abstract terminology defined as follows:
\begin{enumerate}[(i)]
  \item A \emph{length $\kappa$ tower} in a category $\cC$ is a functor $A$ of
  the (opposite) category of ordinals $\{\lambda\mid \lambda\le\kappa\}$, with
  arrows $\lambda\to\lambda'$ for $\lambda'\le\lambda$ as morphisms, into~$\cC$.
  Denote it by~$\{A(\lambda)\}_{\lambda\le\kappa}$ or $\{A(\lambda)\}$.
  \item A \emph{$\kappa$-equivalence} between two towers $\{A(\lambda)\}$
  and~$\{A'(\lambda)\}$ is a natural equivalence $\phi = \{\phi_\lambda\colon
  A(\lambda) \isomto A'(\lambda)\}_{\lambda\le\kappa}$ between the two functors,
  that is, each $\phi_\lambda$ is an equivalence and $\phi_\lambda$ is a lift of
  $\phi_{\lambda'}$ for $\lambda' \le \lambda \le \kappa$. Say $\{A(\lambda)\}$
  and~$\{A'(\lambda)\}$ are \emph{$\kappa$-equivalent} if there is a
  $\kappa$-equivalence between them.
  \item A \emph{length $\kappa+1$ extension} of a length $\kappa$ tower
  $\{A(\lambda)\}_{\lambda\le \kappa}$ is a length $\kappa+1$ tower
  $\{B(\lambda)\}_{\lambda\le\kappa+1}$ equipped with a $\kappa$-equivalence
  between $\{A(\lambda)\}_{\lambda\le\kappa}$ and
  $\{B(\lambda)\}_{\lambda\le\kappa}$.
  \item Two length $\kappa+1$ extensions $\{B(\lambda)\}_{\lambda\le \kappa+1}$
  and $\{B'(\lambda)\}_{\lambda\le \kappa+1}$ of a tower of length $\kappa$,
  $\{A(\lambda)\}_{\lambda\le \kappa}$, are \emph{equivalent} if  the composition
  $B(\kappa)\isomto A(\kappa) \isomto B'(\kappa)$ lifts to an equivalence
  $B(\kappa+1)\isomto B'(\kappa+1)$.
\end{enumerate}

In this paper, towers and their extensions will always be transfinite lower
central quotient towers $\{\widehat\pi/\widehat\pi_\lambda\}_{\lambda\le\kappa}$
of 3-manifold groups~$\pi$.  (In this case,
$\{\widehat\pi/\widehat\pi_\lambda\}_{\lambda\le\kappa}$ and $\{\widehat
G/\widehat G_\lambda\}_{\lambda\le\kappa}$ are $\kappa$-equivalent if and only
if $\widehat\pi/\widehat\pi_\kappa$ and $\widehat G/\widehat G_\kappa$ are
isomorphic.)  We define a \emph{length $\kappa+1$ extension
of~\eqref{equation:transfinite-lcsq-tower} by a 3-manifold group} to be a length
$\kappa+1$ extension of the form $\{\widehat\pi/\widehat\pi_\lambda\}$ where
$\pi=\pi_1(M)$ for some closed 3-manifold~$M$.

For towers of 3-manifold groups, the following two problems are formulated
naturally:

\begin{enumerate}
  \item Classify length $\kappa+1$ extensions of a given fixed
  tower of length $\kappa$, modulo equivalence of extensions in the sense of~(iv).
  \item Classify length $\kappa+1$ towers whose length $\kappa$ subtowers are
  $\kappa$-equivalent to a given fixed tower of length $\kappa$, modulo
  $(\kappa+1)$-equivalence in the sense of~(ii).
\end{enumerate}

The following results are immediate consequences of
Corollary~\ref{corollary:main-determination-equiv-rel} and
Theorem~\ref{theorem:main-milnor-inv}\@.

\begin{corollaryalpha}
  \label{corollary:main-tower-classifications}
  For every ordinal $\kappa$, the following hold.

  \begin{enumerate}
    \item \label{item:tower-classification-lifts}
    The set of classes
    \[
      \left\{
        \begin{tabular}{@{}c@{}}
          length $\kappa+1$ extensions of~\eqref{equation:transfinite-lcsq-tower}\\
          by 3-manifold groups
        \end{tabular}
      \right\}
      \bigg/
      \begin{tabular}{@{}c@{}}
        equivalence of length \\
        $(\kappa+1)$-extensions of~\eqref{equation:transfinite-lcsq-tower}
      \end{tabular}
    \]
    is in one-to-one correspondence with $\cR_\kappa(\Gamma)/\mathord{\sim}$,
    via the invariant~$\theta_\kappa$.

    \item \label{item:tower-classification-nonlifts}
    The set of classes
    \[
      \left\{
        \begin{tabular}{@{}c@{}}
          length $\kappa+1$ towers of 3-manifold groups \\
          with length $\kappa$ subtower $\kappa$-equivalent
          to~\eqref{equation:transfinite-lcsq-tower}
        \end{tabular}
      \right\}      
      \bigg/
      \begin{tabular}{@{}c@{}}
        $(\kappa+1)$-equivalence
      \end{tabular}
    \]
    is in one-to-one correspondence with $\cR_\kappa(\Gamma)/\mathord{\approx}$,
    via the 3-manifold Milnor invariant~$\bar\mu_\kappa$.

  \end{enumerate}
\end{corollaryalpha}

\begin{remark}
  The two classifications in
  Corollary~\ref{corollary:main-tower-classifications}\ref{item:tower-classification-lifts}
  and~\ref{item:tower-classification-nonlifts} are indeed \emph{not} identical.
  More precisely, the natural surjection from the set of classes in
  Corollary~\ref{corollary:main-tower-classifications}\ref{item:tower-classification-lifts}
  onto that in~\ref{item:tower-classification-nonlifts}, or equivalently the
  surjection $\cR_\kappa(\Gamma)/{\sim} \to \cR_\kappa(\Gamma)/{\approx}$, is
  \emph{not} injective in general.  In fact, for the first transfinite ordinal
  $\omega$, Theorem~\ref{theorem:main-tb-omega} below presents an explicit
  3-manifold example for which $\cR_\omega(\Gamma)/{\sim}$ is an infinite set
  but $\cR_\omega(\Gamma)/{\approx}$ is a singleton.
\end{remark}

\subsection{Transfinite gropes and the invariants}
\label{subsection:results-transfinite-grope}

In this paper, we also introduce a previously unexplored notion of
\emph{transfinite gropes} (see Section~\ref{subsection:transfinite-grope}), and
relate them to the transfinite Milnor invariants.  Once again, this extends well
known results concerning classical Milnor invariant of links and the existence
of finite (asymmetric) gropes.  For instance, Freedman and
Teichner~\cite{Freedman-Teichner:1995-2} and Conant, Schneiderman and Teichner
as summarized in~\cite{Conant-Teichner-Schneiderman:2011-1}, as well as work of
the first author~\cite{Cha:2018-1}.

In~\cite{Freedman-Teichner:1995-2}, for finite $k$, a grope corresponding to the
$k$th term of the lower central series is called a grope of class~$k$.  Briefly,
we extend this to the case of an arbitrary transfinite ordinal $\kappa$, to
define a notion of a \emph{grope of \textup{(}transfinite\textup{)}
class~$\kappa$}.  We say that a 4-dimensional cobordism $W$ between two
3-manifolds $M$ and $N$ is a \emph{grope cobordism of class~$\kappa$} if
$H_1(M)\to H_1(W)$ and $H_1(N)\to H_1(W)$ are isomorphisms and the cokernels of
$H_2(M)\to H_2(W)$ and $H_2(N)\to H_2(W)$ are generated by homology classes
represented by gropes of class~$\kappa$. See
Definitions~\ref{definition:transfinite-grope},
\ref{definition:class-of-homology-class-in-H_2}
and~\ref{definition:cobordism-of-class-kappa} for precise descriptions.

Transfinite gropes give another characterization of the equivalent properties in
Theorems~\ref{theorem:main-determination-lcsq-nonlift}
and~\ref{theorem:main-milnor-inv}, as stated below.

\begin{theorem-named}[Addendum to
  Theorems~\ref{theorem:main-determination-lcsq-nonlift}
  and~\ref{theorem:main-milnor-inv}]
  
  Suppose $M$ is a closed 3-manifold such that
  $\widehat{\pi_1(M)}/\widehat{\pi_1(M)}_\kappa$ is isomorphic
  to~$\widehat\Gamma/\widehat\Gamma_\kappa$. Then the following is equivalent to
  the properties \textup{\ref{item:iso-kappa+1-nonlift}} and
  \textup{\ref{item:theta-kappa-vanishes-mod-aut}} in
  Theorem~\ref{theorem:main-determination-lcsq-nonlift}, and to the properties
  \ref{item:milnor-inv-vanish}--\ref{item:milnor-inv-next-defined} of
  Theorem~\ref{theorem:main-milnor-inv}.

  \begin{enumerate}[start=0]
    \item\label{item:kappa+1-cobordism} There is a grope cobordism of class
    $\kappa+1$ between $M$ and another closed 3-manifold $N$ satisfying
    $\widehat{\pi_1(N)} / \widehat{\pi_1(N)}_{\kappa+1} \cong \widehat\Gamma /
    \widehat\Gamma_{\kappa+1}$.
  \end{enumerate}
\end{theorem-named}

Its proof is given in Section~\ref{subsection:transfinite-grope}.

As a key ingredient of the proof, we develop and use a transfinite
generalization of a well-known theorem of Stallings and
Dwyer~\cite{Stallings:1965-1,Dwyer:1975-1}. Since we believe that it will be
useful for other applications in the future as well, we present the statement
here.

\begin{theorem-named}[Theorem~\ref{theorem:transfinite-stallings-dwyer}]
  Let $\kappa>1$ be an arbitrary ordinal.  Suppose $f\colon \pi\to G$ be a group
  homomorphism inducing an isomorphism $H_1(\pi) \isomto H_1(G)$.  In addition,
  if $\kappa$ is a transfinite ordinal, suppose $G$ is finitely generated.  Then
  $f$ induces an isomorphism $\widehat\pi/\widehat\pi_{\kappa} \isomto \widehat
  G/\widehat G_{\kappa}$ if and only if $f$ induces an epimorphism
  \begin{equation*}
      H_2(\widehat\pi)\to H_2(\widehat G)/\Ker\{H_2(\widehat G)\rightarrow
      H_2(\widehat G/\widehat G_{\lambda})\}
  \end{equation*}
  for all ordinals $\lambda<\kappa$.
\end{theorem-named}

See also Corollaries~\ref{corollary:transfinite-stallings-dwyer},
\ref{corollary:transfinite-stallings-dwyer-without-localization-in-H_2}
and~\ref{corollary:transfinite-lcsq-grope} in
Section~\ref{section:transfinite-stallings-dwyer}.

\subsection{Realization of the invariants}
\label{subsection:results-realization}

Our next result is an algebraic characterization of the classes in~$\cR_\kappa(\Gamma)$.
Denote by $tH_*(-)$ the torsion subgroup of~$H_*(-)$.

\begin{theoremalpha}
  \label{theorem:main-realization}
  Let $\kappa\ge 2$ be an arbitrary ordinal.  A class $\theta\in
  H_3(\widehat\Gamma/\widehat\Gamma_\kappa)$ lies in $\cR_\kappa(\Gamma)$ if and
  only if the following two conditions hold.
  \begin{enumerate}
    \item The cap product
    \[
      \cap\,\theta\colon tH^2(\widehat\Gamma/\widehat\Gamma_\kappa)
      \to tH_1(\widehat\Gamma/\widehat\Gamma_\kappa) \cong tH_1(\Gamma)
    \]
    is an isomorphism.
    \item The composition
    \[
      H^1(\widehat\Gamma/\widehat\Gamma_\kappa)
      \xrightarrow{\cap\,\theta} H_2(\widehat\Gamma/\widehat\Gamma_\kappa)
      \xrightarrow{\text{pr}}
      H_2(\widehat\Gamma/\widehat\Gamma_\kappa) / 
      \Ker\{H_2(\widehat\Gamma/\widehat\Gamma_\kappa)
      \rightarrow H_2(\widehat\Gamma/\widehat\Gamma_\lambda)\} 
    \]
    is surjective for all ordinals $\lambda<\kappa$.
  \end{enumerate}
\end{theoremalpha}

We remark that the definition of $\cR_\kappa(\Gamma)$ given in~\eqref{equation:realizable-classes} still makes sense even when $\Gamma$ is not a 3-manifold group.
(In this case $\cR_\kappa(\Gamma)$ may be empty.)
Theorem~\ref{theorem:main-realization} holds for any finitely presented group~$\Gamma$.

The conditions (1) and (2) in Theorem~\ref{theorem:main-realization} may be viewed as Poincar\'e duality imposed properties of the given class $\theta$ with respect to the cap product.
Also note that if $\kappa$ is a discrete ordinal, ``for all ordinals $\lambda<\kappa$'' in (2) can be replaced with ``for $\lambda=\kappa-1$.''

For a finite ordinal $\kappa$, Theorem~\ref{theorem:main-realization} is
essentially due to Turaev~\cite{Turaev:1984-1}.  Our new contribution in
Theorem~\ref{theorem:main-realization} is to extend his result transfinitely.

The proof of Theorem~\ref{theorem:main-realization} is given in
Section~\ref{section:realization-transfinite-inv}.  Among other ingredients, the
transfinite generalization of the Stallings-Dwyer
theorem~\cite{Stallings:1965-1,Dwyer:1975-1} stated above as
Theorem~\ref{theorem:transfinite-stallings-dwyer}  plays a key role in the proof
of Theorem~\ref{theorem:main-realization}.

\subsection{Universal $\theta$-invariant}
\label{subsection:results-universal-inv}

By generalizing the approach of Levine's work on links~\cite{Levine:1989-1}, we
define and study what we call the \emph{universal $\theta$-invariant} of a
3-manifold.

Once again, fix a closed 3-manifold $Y$ and let $\Gamma=\pi_1(Y)$.
Now suppose $M$ is a closed 3-manifold with $\pi=\pi_1(M)$ which admits an
isomorphism $f\colon \widehat{\pi} \isomto \widehat{\Gamma}$.

\begin{definition}
  \label{definition:theta-final}
  Define $\widehat\theta(M)\in H_3(\widehat\Gamma)$ to be the image of the
  fundamental class $[M]\in H_3(M)$ under the composition
  \[
    H_3(M)\to H_3(\pi) \to H_3(\widehat\pi) \xrightarrow[\cong]{f_*}
    H_3(\widehat\Gamma).
  \]
\end{definition}

Also, define the set of realizable classes in $H_3(\widehat\Gamma)$ by 
\[
  \widehat\cR(\Gamma) = \bigg\{\widehat\theta(M)\in H_3(\widehat\Gamma)
  \, \bigg| \, 
  \begin{tabular}{@{}c@{}}
    $M$ is a closed 3-manifold equipped with \\
    an isomorphism $\widehat{\pi_1(M)} \isomto \widehat\Gamma$
  \end{tabular} \bigg\}.
\]
Note that the value of $\widehat\theta(M)$ in the orbit space
$\widehat\cR(\Gamma)/\Aut(\widehat\Gamma)$ is determined by $M$, independent of
the choice of the isomorphism~$f$.

We remark that if $M$ is equipped with an isomorphism $f\colon \widehat\pi
\isomto \widehat\Gamma$ so that $\widehat\theta(M)$ is defined, then $f$ induces
an isomorphism $\widehat\pi/\widehat\pi_\kappa \isomto
\widehat\Gamma/\widehat\Gamma_\kappa$, and thus the invariant $\theta_\kappa(M)$
is defined for all ordinals~$\kappa$.  Moreover, $\theta_\kappa(M)$ is the image
of $\widehat\theta(M)$ under $\widehat\cR(\Gamma) \to \cR_\kappa(\Gamma)$
induced by the projection $\widehat\Gamma \to
\widehat\Gamma/\widehat\Gamma_\kappa$.  Since this factors through
$\cR_{\kappa+1}(\Gamma)$, it follows that $\theta_\kappa(M)$ vanishes in the
cokernel of $\cR_{\kappa+1}(\Gamma)\to \cR_\kappa(\Gamma)$, or equivalently
$\theta_\kappa(M) \sim \theta_\kappa(Y)$ in $\cR_\kappa(\Gamma)$, for every
ordinal~$\kappa$.  Consequently, $\bar\mu_\kappa(M)$ vanishes for all~$\kappa$
if $\widehat\theta(M)$ is defined.  It seems to be hard to prove or disprove the
converse.

Similarly to the $\theta_\kappa$-invariants (see Theorem~\ref{theorem:results-homology-cob}), $\widehat\theta(M) \in
\widehat\cR(\Gamma)/\Aut(\widehat\Gamma)$ is a homology cobordism invariant.  We prove this in
Theorem~\ref{theorem:homology-cob-final-inv}\@.  Also, we prove a realization theorem characterizing homology classes in
$\widehat\cR(\Gamma)$, which is analogous to Theorem~\ref{theorem:main-realization}\@.

\begin{theoremalpha}
  \label{theorem:main-realization-final-inv}
  A homology class $\theta\in H_3(\widehat\Gamma)$ lies in $\widehat\cR(\Gamma)$
  if and only if the following two conditions hold.
  \begin{enumerate}
    \item The cap product\/ $\cap\,\theta\colon tH^2(\widehat\Gamma) \to
    tH_1(\widehat\Gamma) \cong tH_1(\Gamma)$ is an isomorphism.
    \item The cap product\/ $\cap\,\theta\colon H^1(\widehat\Gamma) \to
    H_2(\widehat\Gamma)$ is surjective.
  \end{enumerate}
\end{theoremalpha}

We prove Theorem~\ref{theorem:main-realization-final-inv} in
Section~\ref{section:final-transfinite-invariant}.

We remark that Levine proved a realization theorem for his link invariant which
lives in $H_3(\widehat F)$~\cite{Levine:1989-1}: for all $\theta\in H_3(\widehat
F)$, there is a link $L$ for which his invariant is defined and equal
to~$\theta$. Theorem~\ref{theorem:main-realization-final-inv} says that in case
of general 3-manifolds, not all homology classes in $H_3$ are necessarily
realizable.  An example is given in Section~\ref{section:tb-computation-final}.

It is an open problem whether Levine's link invariant in~\cite{Levine:1989-1} is
nontrivial.  In Theorem~\ref{theorem:main-tb-final} below, for the 3-mainfold
case, we show that $\widehat\theta(M)$ is nontrivial.

\subsection{A torus bundle example}
\label{subsection:results-torus-bundle}

This section gives a complete and careful analysis of one example which illustrates the full collection of transfinite
invariants considered in this paper. For the underlying $3$-manifold, a torus bundle over a circle, the fundamental group of the fiber is a module over the group ring of covering translations, facilitating our computation of the group localization.

We thus compute and analyze the full array of invariants under consideration ---  $\theta$, $\overline\mu$ and $\widehat\theta$. Moreover, this example illustrates several fundamental features of the invariants, including: (i)~\emph{nontriviality} of $\theta_\omega$ for the first transfinite ordinal~$\omega$, (ii)~\emph{nontriviality} of $\widehat\theta$ even when all finite length $\theta$ and $\overline\mu$ vanish, and (iii)~\emph{torsion} values of finite length $\theta$ and~$\overline\mu$.

Let $Y$ be the torus bundle over $S^1$ with monodromy $\sbmatrix{-1 & 0 \\ 0 &
-1}$.  That is, viewing $S^1$ as the unit circle in the complex plane,
\[
  Y = S^1\times S^1\times [0,1] / (z^{-1},w^{-1},0) \sim (z,w,1).
\]
Let $\Gamma=\pi_1(Y)$ be the fundamental group of the torus bundle.  In our
earlier work~\cite{Cha-Orr:2011-1}, we computed the homology
localization~$\widehat\Gamma$.  Using this, it is not hard to compute its
transfinite lower central quotients and see that $\widehat\Gamma$ is
transfinitely nilpotent.  Indeed, $\widehat\Gamma_{\omega+1}$ is trivial.
Our computation starts from this.

\subsubsection*{The first transfinite invariant}

For the first transfinite ordinal $\omega$, we compute the homology
$H_3(\widehat\Gamma/\widehat\Gamma_\omega)$ and its subset of realizable classes
$\cR_\omega(\Gamma)$.  Moreover we completely determine the two equivalence
relations $\sim$ and $\approx$ on $\cR_\omega(\Gamma)$, which were defined in
Sections~\ref{subsection:results-determination}
and~\ref{subsection:results-milnor-inv}.  The computation especially tells us the
following.

\begin{theoremalpha}
  \label{theorem:main-tb-omega}
  For the torus bundle group $\Gamma$, the following hold. 
  \begin{enumerate}
    \item \label{item:main-tb-omega-theta}
    The set $\cR_\omega(\Gamma)/{\sim}$ of equivalence classes of realizable
    values of $\theta_\omega$ is infinite.  Consequently, by
    Corollary~\ref{corollary:main-tower-classifications}\ref{item:tower-classification-lifts},
    there are infinitely many distinct equivalence classes of length $\omega+1$
    extensions, by 3-manifolds, of the length $\omega$ tower
    $\{\widehat\Gamma/\widehat\Gamma_\lambda\}_{\lambda\le\omega}$ of the torus
    bundle~$Y$ (in the sense of
    Section~\ref{subsection:results-transfinite-tower}).

    \item \label{item:main-tb-omega-mu}
    The set $\cR_\omega(\Gamma)/{\approx}$ is a singleton. Consequently,
    $\bar\mu_\omega(M) \in \cR_\omega(\Gamma)/{\approx}$ vanishes whenever it is
    defined.  Also, for all closed 3-manifold groups $\pi$ such that the length
    $\omega$ tower $\{\widehat\pi/\widehat\pi_\lambda\}_{\lambda\le\omega}$ is
    $\omega$-equivalent to that of $\Gamma$ (in the sense of
    Section~\ref{subsection:results-transfinite-tower}), the length $\omega+1$
    tower $\{\widehat\pi/\widehat\pi_\lambda\}_{\lambda\le\omega+1}$ is
    automatically $(\omega+1)$-equivalent to that of~$\Gamma$, by
    Corollary~\ref{corollary:main-tower-classifications}\ref{item:tower-classification-nonlifts}.
    \end{enumerate}
\end{theoremalpha}

Theorem~\ref{theorem:main-tb-omega}\ref{item:main-tb-omega-theta} illustrates
that the transfinite $\theta$-invariant of length $\omega$ provides highly
nontrivial information, even when the transfinite Milnor invariant $\bar\mu$ of
the same length vanishes. Examples with nonvanishing transfinite Milnor
invariants will be given in
Section~\ref{subsection:results-nontrivial-milnor} below.  See
Theorem~\ref{theorem:main-tb-milnor}\@.

The tower interpretations in Theorem~\ref{theorem:main-tb-omega} particularly
tell us the following: there are 3-manifold groups $\pi$ such that there is an
isomorphism $\widehat\pi/\widehat\pi_\omega \isomto
\widehat\Gamma/\widehat\Gamma_\omega$ which does not lift to an isomorphism
between $\widehat\pi/\widehat\pi_{\omega+1}$ and
$\widehat\Gamma/\widehat\Gamma_{\omega+1}$ but
$\widehat\pi/\widehat\pi_{\omega+1}$ and
$\widehat\Gamma/\widehat\Gamma_{\omega+1}$ are isomorphic.

Theorem~\ref{theorem:main-tb-omega} is an immediate consequence of
Theorem~\ref{theorem:tb-transfinite-computation} and
Corollary~\ref{corollary:tb-omega+1-tower-isom}, which presents the outcome
of our computation for $\widehat\Gamma/\widehat\Gamma_\omega$.  See
Section~\ref{section:tb-transfinite-computation} for full details.

\subsubsection*{The universal invariant}

We also carry out computation of the invariant $\widehat\theta$ over the
homology localization of torus bundle group.  Among consequences of the
computation, we have the following.

\begin{theoremalpha}
  \label{theorem:main-tb-final}
  For the torus bundle group $\Gamma$, the set
  $\widehat\cR(\Gamma)/\Aut(\widehat\Gamma)$ of realizable values of
  $\widehat\theta$ modulo the automorphism action is infinite.  This detects the
  existence of infinitely many distinct homology cobordism classes of closed
  3-manifolds $M$ with $\pi=\pi_1(M)$, such that $\widehat\pi \cong
  \widehat\Gamma$ and thus $\theta_\kappa(M)$ is defined and vanishes in
  $\Coker\{\cR_{\kappa+1}(\Gamma) \to \cR_\kappa(\Gamma)\}$ for all
  ordinals~$\kappa$.  In particular, for every ordinal $\kappa$, the Milnor
  invariant $\bar\mu_\kappa(M)$ vanishes for these 3-manifolds~$M$.
\end{theoremalpha}

This illustrates that the invariant $\widehat\theta$ is highly nontrivial for
3-manifolds for which all (transfinite) Milnor type invariants vanish.

This may be compared with the case of Levine's link invariant $\theta(L)\in
H_3(\widehat F)$ where $F$ is a free group~\cite{Levine:1989-1}.  (For
$0$-surgery on a link, this invariant is equivalent to J. Y. Le Dimet's link
concordance invariant defined in~\cite{LeDimet:1988-1}.)  The fundamental
question of Levine in~\cite{Levine:1989-1}, which is still left open, asks
whether $\theta(L)$ can be nontrivial.  Due to Levine's realization result
in~\cite{Levine:1989-1}, this is equivalent to whether $H_3(\widehat F)$ is
nontrivial.  Our result shows that in the case of general 3-manifold groups, the
answer is affirmative, even modulo the automorphism action.

Theorem~\ref{theorem:main-tb-final} is a consequence of
Theorems~\ref{theorem:tb-final-realization} and~\ref{theorem:tb-final-aut}.
Indeed, in Section~\ref{section:tb-computation-final}, we provide a complete
computation of $\cR(\widehat\Gamma)$ and the action of $\Aut(\widehat\Gamma)$.

\subsubsection*{Finite length invariants}

The torus bundle example also reveals interesting aspects of finite length case
of the Milnor type invariant.  Our computation of the invariant $\theta_k$ for
finite $k$ proves the following result.

\begin{theoremalpha}
  \label{theorem:main-tb-finite}
  For every finite $k$, the set of realizable classes $\cR_k(\Gamma)$ is finite,
  and thus the set of equivalence classes $\cR_k(\Gamma)/\mathord{\sim}$ is
  finite.  Moreover,
  \[
    2\le \# (\cR_k(\Gamma)/\mathord{\sim}) \le 7\cdot 2^{4(k-2)}+1.
  \]
  Consequently, by
  Corollary~\ref{corollary:main-tower-classifications}\ref{item:tower-classification-lifts},
  the number of equivalence classes of length $k+1$ extensions, by 3-manifold
  groups, of the length $k$ tower $\Gamma/\Gamma_k \to \cdots \to
  \Gamma/\Gamma_1$ (in the sense of
  Section~\ref{subsection:results-transfinite-tower}) is between $2$ and $7\cdot
  2^{4(k-2)}+1$ inclusively.
\end{theoremalpha}

Theorem~\ref{theorem:main-tb-finite} is a consequence of
Theorem~\ref{theorem:tb-finite-computation} and
Corollary~\ref{corollary:tb-finite-tower-estimate} in
Section~\ref{section:tb-finite-computation}, which provide a more detailed
description of the structure of~$\cR_k(\Gamma)$ and related objects.

\begin{remark}
  Recall that, for $m$-component links with vanishing Milnor invariants of
  length${}<k$, the Milnor invariants of length $k$ are integer-valued, and
  consequently, they are either all trivial, or have infinitely many values.
  (Indeed the Milnor invariants of length $k$ span a free abelian group of known
  finite rank.  See~\cite{Orr:1989-1}.)  However, for the torus bundle case in
  Theorem~\ref{theorem:main-tb-finite}, it turns out that the finite length
  $\theta_k$ invariants live in torsion groups, in fact, finite 2-groups. This
  leads us to questions related to potential applications to link concordance, and
  to the higher order Arf invariant conjecture asked by Conant, Schneiderman and
  Teichner.  In the final section of this paper, we discuss these questions,
  together with other questions arising from our work.
\end{remark}

\subsection{Modified torus bundle examples}
\label{subsection:results-nontrivial-milnor}

We now consider a family of modified torus bundles $\{M_r \mid r$ is an odd
integer$\}$, to show that transfinite Milnor invariants of 3-manifolds are
nontrivial in general.

The modified torus bundles are obtained by changing just one entry in the
monodromy matrix of the previous example~$Y$: $M_r$ has monodromy $\sbmatrix{ -1
& r \\ 0 & -1}$.  That is,
\[
  M_r = S^1\times S^1\times [0,1] / (z^{-1}, z^r w^{-1}, 0) \sim (z,w,1).
\]
We remark that discussions with Sergei Ivanov and Roman Mikhailov led us to
consider this modification.  They studied the Bousfield-Kan completion of
$3$-manifold groups, with $\pi_1(M_r)$ as main
examples~\cite{Ivanov-Mikhailov:2019-1}.

Fix an odd integer $d$, and choose $Y=M_d$ as the basepoint manifold, to which
other manifolds $M_r$ are to be compared.  Let $\Gamma=\pi_1(M_d)$.  We prove
the following.

\begin{theoremalpha}
  \label{theorem:main-tb-milnor}
  For any odd integer $r$, $\bar\mu_k(M_r)\in \cR_k(\Gamma)/{\approx}$ is
  defined and vanishes for every finite~$k$.  Moreover,
  $\widehat{\pi_1(M_r)}/\widehat{\pi_1(M_r)}_\omega \cong
  \widehat\Gamma/\widehat\Gamma_\omega$ and thus $\bar\mu_{\omega}(M_r)$ is
  defined.  But $\bar\mu_{\omega}(M_r) = \bar\mu_{\omega}(M_s)$ in
  $\cR_\omega(\Gamma)/{\approx}$ if and only if the rational number~$|r/s|$ is a
  square.
  
  In particular, the set of realizable values $\cR_\omega(\Gamma)/{\approx}$ of
  the 3-manifold Milnor invariant is infinite, and there are infinitely many
  homology cobordism classes of $3$-manifolds with the same finite length Milnor
  invariants but distinct Milnor invariants of length~$\omega$.
\end{theoremalpha}

Indeed, we show that $\cR_\omega(\Gamma)/{\approx}$ is equal to
$\Z_{(2)}^\times=\{a/b\in \Q \mid a, b\in 2\Z+1\}$ modulo the (multiplicative)
subgroup $\pm(\Z_{(2)}^\times)^2 =\{\pm \alpha^2 \mid \alpha\in
\Z_{(2)}^\times\}$. So, $\cR_\omega(\Gamma)/{\approx}$ can be naturally
identified with the set of odd positive integers $r$ with no repeated primes in
the factorization. Such an $r$ corresponds to the value of the length $\omega$
Milnor invariant $\bar\mu_\omega(M_{rd})$ of the 3-manifold~$M_{rd}$. See
Theorem~\ref{theorem:modified-tb-comp}.  So, the modified torus bundles
explicitly realize nontrivial values of the transfinite Milnor invariant
$\bar\mu_\omega$ over the group~$\Gamma=\pi_1(M_d)$.

The following is a consequence of Theorem~\ref{theorem:main-tb-milnor} combined
with
Corollary~\ref{corollary:main-tower-classifications}\ref{item:tower-classification-nonlifts}:
there are infinitely many 3-manifold groups $\pi$, such that the lower central
series quotient towers $\{\widehat\pi/\widehat\pi_\kappa\}_{\kappa\le\omega}$ of
length $\omega$ are mutually $\omega$-equivalent (in the sense of
Section~\ref{subsection:results-transfinite-tower}), but the length
$\omega+1$ towers $\{\widehat\pi/\widehat\pi_\kappa\}_{\kappa\le\omega+1}$ are
not pairwise $(\omega+1)$-equivalent.

\section{Homology localization of groups}
\label{section:localization}

In this section we review basic facts on the homology localization of groups,
and prove some results which will be useful in later sections.  All results in
this section were known to J. P. Levine.  We include these results here for
completeness, since group localization and the results herein plays an essential
role in this paper.

\subsection{Preliminaries}
\label{subsection:localization-prelim}

We begin with the definition of the homology localization which we use.
Recall that a group homomorphism $\pi\to G$ is \emph{2-connected} if it induces
an isomorphism on $H_1(-;\Z)$ and an epimorphism on $H_2(-;\Z)$.  Let $\Omega$
be the collection of 2-connected homomorphisms between finitely presented
groups. A group $\Gamma$ is \emph{local with respect to $\Omega$}, or simply
\emph{local} if, for every $\alpha\colon A \to B$ in $\Omega$ and every
homomorphism $f\colon A\to \Gamma$, there is a unique homomorphism $g\colon B\to
\Gamma$ satisfying $g\circ\alpha=f$.
\[
  \begin{tikzcd}[sep=large]
    A \ar[r,"\alpha"] \ar[d,"f" left] & B \ar[ld,dashed,"g" below right]
    \\
    \Gamma
  \end{tikzcd}
\]
A \emph{localization with respect to $\Omega$} is defined to be a pair
$(E,\iota)$ of a functor $E$ from the category of groups to the full subcategory
of local groups and a natural transformation $\iota=\{\iota_G\colon G \to
E(G)\}$ satisfying the following: for each homomorphism $f\colon G\to \Gamma$
with $\Gamma$ local, there is a unique homomorphism $g\colon E(G) \to \Gamma$
such that $g\circ\iota_G = f$.
\[
  \begin{tikzcd}[sep=large]
    G \ar[r,"\iota_G"] \ar[d,"f" left] & E(G) \ar[ld,dashed,"g" below right]
    \\
    \Gamma
  \end{tikzcd}
\]
In this paper, we denote $E(G)$ by~$\widehat G$.

It is a straightforward exercise that a localization is unique if it exists. The
existence of a localization with respect to $\Omega$ is due to Vogel and Levine.
Indeed, in his unpublished manuscript~\cite{Vogel:1978-1}, Vogel developed a
general theory of localization of spaces with respect to homology, and its group
analogue is the localization we discuss.  In~\cite{Levine:1989-2}, Levine
developed an alternative approach using certain systems of equations over a
group to define a notion of ``algebraic closure.''  He showed that it exists and
is equal to the localization with respect to the subset of our $\Omega$
consisting of $\alpha\colon A\to B$ in $\Omega$ such that $\alpha(A)$ normally
generates~$B$. Although the modified closure with respect to our $\Omega$ (that
is, omitting the normal closure condition) was known to Levine, this theory
first appeared with proof in a paper~\cite{Cha:2004-1}.  As a useful overview on
homology localization for geometric topologists, the readers are referred
to~\cite[Section~2]{Cha-Orr:2009-1}.

The following properties of the homology localization $\widehat\pi$ are essential
for our purpose. 

\begin{theorem}[\cite{Levine:1989-1,Cha:2004-1}]
  \phantomsection\label{theorem:localization-basic-facts}
  \leavevmode\Nopagebreak
  \begin{enumerate}
    \item\label{item:localization-2-connected-map}
    If $\pi \to G$ is a 2-connected homomorphism between finitely presented
    groups, then it induces an isomorphism $\widehat\pi \isomto \widehat G$.
    \item\label{item:localization-colimit}
    For a finitely presented group $G$, there is a sequence
    \[
      G=P(1)\rightarrow P(2) \rightarrow \cdots \rightarrow P(k) \rightarrow \cdots
    \]
    of 2-connected homomorphisms of finitely presented groups $P(k)$ such that
    the localization $G\to \widehat G$ is equal to the colimit homomorphism $G
    \to \colim_k P(k)$. Consequently, $G\to \widehat G$ is 2-connected.
    \end{enumerate}
\end{theorem}

Theorem~\ref{theorem:localization-basic-facts}\ref{item:localization-2-connected-map}
is obtained by a routine standard argument using the universal
properties given in the definitions.  We omit the details.  For instance,
see~\cite[Proposition~6.4]{Cha:2004-1}, \cite[Proposition~5]{Levine:1989-1}. The
proof of
Theorem~\ref{theorem:localization-basic-facts}\ref{item:localization-colimit}
is not straightforward and uses the actual construction of the localization.
See~\cite[Proposition~6.6]{Cha:2004-1}, \cite[Proposition~6]{Levine:1989-1}.

\begin{corollary}
  \phantomsection\label{corollary:localization-basic-consequences}
  \leavevmode\Nopagebreak
  \begin{enumerate}
    \item\label{item:local-lcsq-ordinary-lcsq}
    For a finitely presented group $G$, the homomorphism $G\to \widehat G$
    induces an isomorphism $G/G_k \to \widehat G/\widehat G_k$ for
    each~$k<\infty$.
    \item\label{item:homology-equiv-lcsq}
    A homology equivalence $X \to Y$ between finite CW-complexes $X$ and $Y$
    with $\pi=\pi_1(X)$ and $G=\pi_1(Y)$ gives rise to isomorphisms $\widehat\pi
    \isomto \widehat G$ and $\widehat \pi/\widehat \pi_\kappa \isomto \widehat
    G/\widehat G_\kappa$ for each ordinal~$\kappa$.
  \end{enumerate}
\end{corollary}

\begin{proof}
  (1)~From
  Theorem~\ref{theorem:localization-basic-facts}\ref{item:localization-colimit},
  it follows that $G\to \widehat G$ is 2-connected.   By Stallings'
  Theorem~\cite{Stallings:1965-1}, $G\to \widehat G$ induces an isomorphism
  $G/G_k \to \widehat G/\widehat G_k$. 
  
  (2) The induced homomorphism $\pi \to G$ is 2-connected, since $K(\pi,1)$ and
  $K(G,1)$ are obtained by attaching cells of dimension${}\ge 3$ to $X$ and~$Y$.
  Since $X$ and $Y$ are finite, it follows that $\widehat\pi \cong \widehat G$
  by
  Theorem~\ref{theorem:localization-basic-facts}\ref{item:localization-2-connected-map}.
  Therefore $\widehat\pi/\widehat\pi_\kappa \cong \widehat G/\widehat G_\kappa$
  for every~$\kappa$.
\end{proof}

\subsection{Acyclic equations and induced epimorphisms on localizations}
\label{subsection:acyclic-equations-surjectivity}

J. P. Levine first proved the following result in~\cite{Levine:1994-1}, in much
greater generality than stated here.  His proof involved group localization
determined by closure with respect to contractible equations, not acyclic
equations. For this reason, we include a brief proof here.  However, this Lemma
was certainly known to Levine.

\begin{lemma}
  \label{lemma:H1-surj-localization-surj}
  If a group homomorphism $\pi\to G$ induces an epimorphism $H_1(\pi)\to H_1(G)$
  and if $G$ is finitely generated, then it induces an epimorphism $\widehat\pi
  \to \widehat G$.
\end{lemma}

This proof of Lemma~\ref{lemma:H1-surj-localization-surj} depends on an
equation-based approach to the localization. In what follows, we give a quick
review of definitions and results we need.  Fix a group~$G$. Following the idea
of Levine~\cite{Levine:1989-1} (see also
Farjoun-Orr-Shelah~\cite{Farjoun-Orr-Shelah:1989}) consider a system
$S=\{x_i=w_i\}$ of equations of the form
\[
  x_i = w_i(x_1,\ldots,x_n), \quad i=1,\ldots,n
\]
where each $x_i$ is a formal variable and $w_i=w_i(x_1,\ldots,x_n)$ is an
element of the free product $G*F$ of $G$ and the free group $F=F\langle
x_1,\ldots x_n\rangle$ on $x_1,\ldots,x_n$.  A \emph{solution} $\{g_i\}_{i=1}^n$
to the system $S$ is defined to be an ordered tuple of $n$ elements $g_i\in G$
such that $g_i = w_i(g_1,\ldots,g_n)$ for $i=1,\ldots,n$. A group homomorphism
$\phi\colon G\to \Gamma$ induces $\phi*\id \colon G*F \to \Gamma*F$, which sends
a system of equations $S$ over $G$ to a system
$\phi(S):=\{x_i=(\phi*\id)(w_i)\}$ over~$\Gamma$. If $\{g_i\}$ is a
solution to $S$, then $\{\phi(g_i)\}$ is a solution to~$\phi(S)$.

Following~\cite[Definition~4.1]{Cha:2004-1}, we say an equation
$x_i=w_i(x_1,\ldots,x_n)$ is \emph{null-homologous}, or \emph{acyclic}, if $w_i$
lies in the kernel of the projection 
\[
  G*F \to F \to H_1(F)=F/[F,F].
\]
A group $G$ is \emph{$\Z$-closed} if every system of acyclic equations over $G$
has a unique solution in~$G$.  We remark that these definitions are variations
of Levine's notion of contractible equations and algebraic closure
in~\cite{Levine:1989-1}.

\begin{theorem}[\cite{Levine:1989-1,Cha:2004-1}] 
  \phantomsection\label{theorem:closure-vs-localization}
  \leavevmode\Nopagebreak
  \begin{enumerate}
    \item\label{item:local-iff-closed}
    A group $G$ is local if and only if $G$ is $\Z$-closed.  In particular,
    every system of acyclic equations over $\widehat G$ has a unique solution
    in~$\widehat G$.
    \item\label{item:localization-consists-of-solutions}
    Every element in $\widehat G$ is a solution of a system of acyclic equations
    over~$G$.  More precisely, for each $g\in \widehat G$, there is a system of
    acyclic equations $S=\{x_i = w_i\}_{i=1}^n$ over $G$ such that the system
    $\iota_G(S)$ over $\widehat G$ has a solution $\{g_i\}_{i=1}^n$ with
    $g_1=g$.
  \end{enumerate}    
\end{theorem}

For the proof of
Theorem~\ref{theorem:closure-vs-localization}\ref{item:local-iff-closed},
see~\cite[Theorems~5.2 and~6.1, Corollary~6.3]{Cha:2004-1}.  For the proof of
Theorem~\ref{theorem:closure-vs-localization}\ref{item:localization-consists-of-solutions},
see~\cite[Theorem~6.1, Proposition~6.6]{Cha:2004-1}.  We remark that these
proofs follow Levine's approach in~\cite[Propositions~3 and~6]{Levine:1989-1}.

\begin{proof}[Proof of Lemma~\ref{lemma:H1-surj-localization-surj}]
  
  Suppose $f\colon \pi\to G$ induces an epimorphism $f_*\colon H_1(\pi)\to
  H_1(G)$.  Fix a finite set $\{a_1,\ldots,a_k\}$ which generates~$G$.  We begin
  by writing equations over $G$ which have $\{a_j\}$ as a solution.  Let $F=F\langle
  y_1,\ldots,y_k\rangle$.  For each $a_j$, since $f_*$ is surjective,
  $a_j=f(b_j)\cdot c_j$ for some $b_j\in \pi$ and $c_j\in [G,G]$.  Write $c_j$
  as a product of commutators in the generators~$a_i^{\pm1}$, to choose a word
  $u_j = u_j(y_1,\ldots,y_k)$ in $[F,F]$ such that $c_j = u_j(a_1,\ldots,a_k)$.
  Let $S_0$ be the system of the acyclic equations $\{y_j = b_j \cdot
  u_j\}_{j=1}^k$ over~$\pi$.  Then $\{a_j\}$ is a solution to the system $f(S_0)
  = \{y_j = f(b_j)\cdot u_j\}$ over~$G$. Applying $\iota_G\colon G\to \widehat
  G$, it follows that $\{\iota_G(a_j)\}$ is a solution to the system
  $\iota_Gf(S_0)$ over~$\widehat G$.
  
  Now, to show that $\widehat f\colon \widehat\pi \to \widehat G$ is surjective,
  fix $g\in \widehat G$.  By
  Theorem~\ref{theorem:closure-vs-localization}\ref{item:localization-consists-of-solutions},
  there is a system $S=\{x_i=w_i\}_{i=1}^n$ of acyclic equations over $G$, with
  $w_i\in G*F\langle x_1,\ldots,x_n\rangle$, such that the system $\iota_G(S)$
  over $\widehat G$ has a solution $\{g_i\}$ with $g_i\in \widehat G$, $g_1=g$.
  Substitute each occurrence of the generator $a_j$ in the word $w_i$ with
  $b_j\cdot u_j$, to obtain a new word $v_i=v_i(x_n,\ldots,x_n,y_1,\ldots,y_k)$.
  Now consider the system of $n+k$ equations $S'=\{x_i=v_i\}_{i=1}^n \cup
  \{y_j=b_j u_j\}_{j=1}^k$, over the group~$\pi$.  Apply the homomorphism
  $\iota_\pi\colon \pi \to \widehat\pi$ to obtain the system $\iota_\pi(S')$
  over~$\widehat\pi$.  By
  Theorem~\ref{theorem:closure-vs-localization}\ref{item:localization-consists-of-solutions},
  $\iota_\pi(S')$ has a solution in $\widehat\pi$, say $\{r_i\}_{i=1}^n \cup
  \{s_j\}_{j=1}^k$.  That is, $r_i=\iota_\pi v_i(r_1,\ldots,r_n,s_1,\ldots,s_k)$
  and $s_j=\iota_\pi b_j \cdot u_j(s_1,\ldots,s_k)$.  Now, apply
  $\smash{\widehat f\colon \widehat\pi \to \widehat G}$ to the
  system~$\iota_\pi(S')$. By the functoriality of the localization, we have
  $\widehat f \iota_\pi(S') = \iota_G f(S')$, and it has $\{\widehat f(r_i)\}
  \cup \{\widehat f(s_j)\}$ as a solution in~$\widehat G$.   The last $k$
  equations of $\iota_Gf(S')$ form the system $\iota_Gf(S_0)$.  By the
  uniqueness of a solution for $\iota_G f(S_0)$, we have $\widehat
  f(s_j)=\iota_G(a_j)$. By the uniqueness of a solution for $f(S')$, it follows
  that $\widehat f(r_i)=g_i$.  In particular, $\widehat f(r_1)=g_1=g$.  This
  proves that $\widehat f\colon \widehat\pi \to \widehat G$ is surjective.
\end{proof}

\subsection{Transfinite lower central quotients of local groups are local}
\label{subsection:lcsq-of-local-group-is-local}

It is well known that a nilpotent group is local, or equivalently the lower
central quotient $G/G_k$ of an arbitrary group $G$ is local for all finite~$k$.
(See, for instance,~\cite[p.~573]{Levine:1989-1}.)  But it is no longer true for
the ordinary transfinite lower central quotients~$G/G_\kappa$.  For instance,
for a free group $F$ of rank${}>1$, $F/F_\omega \cong F$ and this is not local.
However, for local groups the following is true.

\begin{lemma}
  \label{lemma:lcsq-of-local-group-is-local}
  If $G$ is a local group, then $G/G_\kappa$ is local for every
  ordinal~$\kappa\ge 1$.  In particular, $\widehat G/\widehat G_\kappa$ is
  local for every group~$G$.
\end{lemma}

We will use the equation-based approach to prove this.

\begin{proof}
  Suppose $S=\{x_i = w_i(x_1,\ldots,x_n)\}_{i=1}^n$ is a system of acyclic
  equations over~$G/G_\kappa$.  It suffices to show that $S$ has a unique
  solution in~$G/G_\kappa$.  For the existence, lift $S$ to a system over $G$,
  by replacing each element of $G/G_\kappa$ which appears in the words $w_i$
  with a pre-image in~$G$.  Since $G$ is local, there is a solution for the
  lift, and the image of the solution under the projection $G\to G/G_\kappa$ is
  a solution for~$S$.

  To prove the uniqueness, we proceed by transfinite induction.  First, for
  $\kappa=1$, $G/G_\kappa=\{e\}$ and thus everything is unique.  Suppose
  $\kappa\ge 2$ and suppose the solution of a system of acyclic equations
  over~$G/G_\lambda$ is unique for all~$\lambda<\kappa$.    Suppose $\{x_i =
  g_i\}$ and $\{x_i=g_i'\}$ are two solutions in $G/G_\kappa$ for a given
  system~$S$ of acyclic equations.

  Suppose that $\kappa$ is a discrete ordinal. Let $p\colon G/G_\kappa \to
  G/G_{\kappa-1}$ be the projection.  Since $\{x_i=p(g_i)\}$ and
  $\{x_i=p(g_i')\}$ are solutions of $p(S)$ over $G/G_{\kappa-1}$,
  $p(g_i)=p(g_i')$ by the uniqueness over~$G/G_{\kappa-1}$. So $g_i'=g_i c_i$
  for some $c_i \in G_{\kappa-1}/G_\kappa$.  Since $G_{\kappa-1}/G_\kappa$ is
  central in $G/G_\kappa$ and the image of $w_i$ under $(G/G_\kappa)*F \to F$
  lies in $[F,F]$, it follows that
  \[
    g_i' = w_i(g_1',\ldots,g_n') = w_i(g_1 c_1,\ldots,g_n c_n)
    = w_i(g_1,\ldots,g_n)= g_i.
  \]
  
  Now, suppose $\kappa$ is a limit ordinal.  For $\lambda<\kappa$, let
  $p_\lambda\colon G/G_\kappa \to G/G_\lambda$ be the projection.  Since
  $\{x_i=p_\lambda(g_i)\}$ and $\{x_i=p_\lambda(g_i')\}$ are solutions of
  $p_\lambda(S)$, it follows that $p_\lambda(g_i) = p_\lambda(g_i')$ by the
  uniqueness of a solution over~$G/G_\lambda$. That is, $g_i^{-1}g_i' \in \Ker p
  = G_\lambda/G_\kappa$ for each $\lambda<\kappa$. Since $G_\kappa =
  \bigcap_{\lambda<\kappa} G_\lambda$, it follows that $g_i=g_i'$
  in~$G/G_\kappa$.
\end{proof}

We remark that when $\kappa$ is a discrete ordinal in the above proof, the
existence of a solution can also be shown under an induction hypothesis that
$G/G_{\kappa-1}$ is local, without assuming that $G$ is local.  Indeed, if
$\{x_i=h_i\}$ is a solution for $p(S)$ over $G/G_{\kappa-1}$, then for any
choice of $h_i' \in p^{-1}(h_i) \subset G/G_\kappa$, it turns out that the
elements $g_i = w_i(h_1',\ldots,h_n')$ form a solution $\{x_i=g_i\}$ for the
given $S$, by a similar argument to the uniqueness proof.
See~\cite[Proposition~1(c)]{Levine:1989-1}.  On the other hand, when $\kappa$ is
a limit ordinal, the assumption that $G$ is local is essential for the existence
(and necessary --- recall the example of $F \cong F/F_\omega$).

\subsection{Closure in the completion}
\label{subsection:closure-in-completion}

For a group $G$, let $\widetilde G = \smash{\varprojlim}_{k<\infty} G/G_k$ be the nilpotent completion.
It is well known that $\widetilde G$ is a local group, essentially by Stallings' theorem.
Therefore, there is a unique homomorphism $\widehat G \to \widetilde G$ making the following diagram commutative:
\[
  \begin{tikzcd}
    G \ar[r]\ar[d] & \widehat G \ar[ld]
    \\
    \widetilde G
  \end{tikzcd}
\]
Following Levine's approach in~\cite{Levine:1989-2}, define $\overline G = \Im\{\widehat G\to \overline G\}$.
We call $\overline G$ the \emph{closure in the completion}.

It is straightforward to verify that $\Ker\{\widehat G \to \overline G\}=\widehat G_\omega$, that is, $\overline G \cong \widehat G/\widehat G_\omega$, using Stallings' theorem.

For later use, we will discuss a special case of a metabelian extension.
Let $G$ be an abelian group and $A$ be a $\Z G$-module.
Denote the semi-direct product by~$A\rtimes G$. Let $\epsilon\colon \Z G \to \Z$ be the augmentation map, and $I:=\Ker \epsilon$ be the augmentation ideal.
Then the lower central subgroup $(A\rtimes G)_{k+1}$ is equal to~$I^k A$, so $(A\rtimes G)/(A\rtimes G)_k = (A/I^kA)\rtimes G$.
It follows that $\widetilde{A\rtimes G} = \widetilde A\rtimes G$, where $\widetilde A := \smash{\varprojlim}_{k<\infty} A/I^kA$ is the $I$-adic completion.
Also, $A\rtimes G$ is residually nilpotent if and only if $\bigcap_{k<\infty} I^k A = 0$.

Let $S:=\{r\in \Z G \mid \epsilon(r)=1\}$, and let $S^{-1}A=\{a/s \mid a\in A,\, s\in S\}$ be the classical localization with respect to~$S$.
By mutiplying $a$ and $s$ by~$-1$, one sees that $S^{-1}A$ is equal to the localization with respect to a larger subset $\{r\in \Z G \mid \epsilon(r)=\pm1\}$.

\begin{theorem}
  \label{theorem:closure-in-completion-metabelian}
  Suppose $\bigcap_{k<\infty} I^k A = 0$.
  Then $\overline{A\rtimes G} = S^{-1}A \rtimes G$.
\end{theorem}

Theorem~\ref{theorem:closure-in-completion-metabelian} is due to Levine~\cite[Proposition~3.2]{Levine:1994-1}.
Indeed, he gave a proof (of a more general statement) for the localization defined in~\cite{Levine:1989-1,Levine:1989-2}, but just by modifying it slightly, his argument applies to the case of the localization we use (which is defined in~\cite{Cha:2004-1}) as well.

For concreteness and for the reader's convenience, we provide a quick proof.

\begin{proof}[Proof of Theorem~\ref{theorem:closure-in-completion-metabelian}]
  By~\cite[Theorem~A.2]{Cha-Orr:2011-1}, $S^{-1}A \rtimes G$ is a local group, since $S^{-1}A$ is the cohn localization of the $\Z G$-module $A$ and $G$ is abelian and thus local.
  It follows that there is a unique homomorphism $\widehat{A\rtimes G} \to S^{-1}A\rtimes G$ making the following diagram commutative:
  \[
    \begin{tikzcd}
      A\rtimes G \ar[r]\ar[d] & \widehat{A\rtimes G} \ar[ld]
      \\
      S^{-1}A \rtimes G
    \end{tikzcd}
  \]

  We claim that $\widehat{A\rtimes G} \to S^{-1}A\rtimes G$ is surjective.
  To show this, it suffices to verify that every $a/s\in S^{-1}A$ lies in the image of $\widehat{A\rtimes G}$.
  Observe that $x=a/s$ a solution of the equation $x = w(x)$, where $w(x) = a+(1-s)x$.
  Write $1-s=\sum_i n_i g_i$, $n_i\in \Z$, $g_i\in G$.
  Then, in multiplicative notation, $w(x) = a\prod_i g_i x^{n_i} g_i^{-1}$, a word in $(A\rtimes G)*F\langle x\rangle$.
  Since $\epsilon(s)=1$, we have $\sum_i n_i = 0$.
  That is, the equation $x=w(x)$ over $A\rtimes G$ is acyclic.
  Therefore, there is a solution $z\in \widehat{A\rtimes G}$ for $x=w(x)$, and $z$ must be sent to $a/s\in S^{-1}A$, since $a/s$ is a solution for $x=w(x)$ in the local group $S^{-1}A\rtimes G$.
  This proves the claim.

  We claim that $A \to \widetilde A = \smash{\varprojlim}_{k<\infty} A/I^kA$ factors through~$S^{-1}A$.  To show this, it suffices to prove that every $s\in S$ is invertible in $\widetilde{\Z G} = \smash{\varprojlim}_{k<\infty} \Z G/I^k$.
  Indeed this is a known fact verified by an elementary argument as follows.
  Since $\epsilon(s)=1$, $1-s \in I$.
  So, writing $(1-s)^k = 1-r_k\cdot s$ with $r_k\in \Z G$, we have $r_k\cdot s \equiv 1 \bmod I^k$.
  Also, $r_{k+1} \equiv r_k \bmod I^k$, so $(r_k)\in \smash{\varprojlim}_{k<\infty} \Z G/I^k$ is a multiplicative inverse of the given~$s$.
  
  By the second claim, there is a natural homomorphism $S^{-1}A \rtimes G \to \widetilde A \rtimes G = \widetilde{A\rtimes G}$.
  Since $\bigcap I^kA =0$, the map $A\to \widetilde A$ is injective, and thus $S^{-1}A \to \widetilde A$ is injective.
  It follows that $S^{-1}A \rtimes G \to \widetilde{A\rtimes G}$ is injective.

  Now, consider
  \[
    \widehat{A\rtimes G} \to S^{-1}A \rtimes G \to \widetilde{A\rtimes G}.
  \]
  The first arrow is surjective by the first claim, and the second arrow is injective, so $S^{-1}A \rtimes G$ is the image of $\widehat{A\rtimes G}$ in $\widetilde{A\rtimes G}$.
  That is, $S^{-1}A \rtimes G = \overline{A\rtimes G}$.
\end{proof}

\section{Invariance under homology cobordism}
\label{section:homology-cobordism-invariance}

In this section we give a proof of
Theorem~\ref{theorem:results-homology-cob}, which says that
$\theta_\kappa$ is invariant under homology cobordism.  Indeed, it is a
straightforward consequence of the definition and the key property of the
homology localization.  We provide details for concreteness.

\begin{definition}
  Two closed 3-manifolds $M$ and $N$ are \emph{homology cobordant} if there is a
  4-manifold $W$ such that $\partial W = M\sqcup -N$ and the inclusions induce
  isomorphisms $H_*(M) \cong H_*(W) \cong H_*(N)$.  Such a 4-manifold $W$ is
  called a \emph{homology cobordism}.
\end{definition}

Fix a group $\Gamma$ and an ordinal~$\kappa$.  Recall that for a closed
3-manifold $M$ with $\pi=\pi_1(M)$ which is equipped with an isomorphism
$f\colon \widehat\pi/\widehat\pi_\kappa \isomto
\widehat\Gamma/\widehat\Gamma_\kappa$, the invariant $\theta_\kappa(M)$ is
defined to be the image of the fundamental class of $M$ under
\[
  H_3(M)\to H_3(\pi)\to H_3(\widehat\pi) \to H_3(\widehat\pi/\widehat\pi_\kappa) \xrightarrow{f_*} H_3(\widehat\Gamma/\widehat\Gamma_\kappa).
\]

\begin{proof}[Proof of Theorem~\ref{theorem:results-homology-cob}] Suppose  $M$
  and $N$ are homology cobordant closed 3-manifolds with $\pi=\pi_1(M)$,
  $G=\pi_1(N)$.
  Theorem~\ref{theorem:results-homology-cob}\ref{item:homology-cob-inv-lcsq}
  asserts that there is an isomorphism $\phi\colon \widehat G/\widehat G_\kappa
  \isomto \widehat\pi/\widehat\pi_\kappa$.  Let $W$ be a homology cobordism
  between $M$ and~$N$.  Then, by
  Corollary~\ref{corollary:localization-basic-consequences}\ref{item:homology-equiv-lcsq},
  the inclusions of $M$ and $N$ into $W$ induce isomorphisms
  $\widehat\pi/\widehat\pi_\kappa \cong
  \widehat{\pi_1(W)}/\widehat{\pi_1(W)}_\kappa \cong \widehat G/\widehat
  G_\kappa$.  Let $\phi\colon \widehat G/\widehat G_\kappa \isomto
  \widehat\pi/\widehat\pi_\kappa$ be the composition.  This is the promised isomorphism.

  Suppose $f\colon \widehat\pi/\widehat\pi_\kappa \isomto
  \widehat\Gamma/\widehat\Gamma_\kappa$ is an isomorphism.  Let
  $\theta_\kappa(M)$ and $\theta_\kappa(N)$ be the invariants defined using the
  isomorphisms $f$ and $f\circ \phi$.
  Theorem~\ref{theorem:results-homology-cob}\ref{item:homology-cob-inv-theta}
  asserts that $\theta_\kappa(M)=\theta_\kappa(N)$ in
  $H_3(\widehat\Gamma/\widehat\Gamma_\kappa)$.  To show this, consider the
  following commutative diagram.
  \[
    \begin{tikzcd}[row sep=large]
      H_3(M) \ar[r] \ar[d,"i_*" left] &
      H_3(\widehat\pi/\widehat\pi_\kappa) \ar[d,"\cong" right]
      \ar[rd,"\cong" below left,"f_*" above right]
      \\
      H_3(W) \ar[r] &
      H_3(\widehat{\pi_1(W)}/\widehat{\pi_1(W)}_\kappa) \ar[r,"\cong"]
      & H_3(\widehat\Gamma/\widehat\Gamma_\kappa)
      \\
      H_3(N) \ar[r] \ar[u,"j_*" left] &
      H_3(\widehat G/\widehat G_\kappa) \ar[u,"\cong" right] 
      \ar[ru,"\cong" above left,"(f\circ \phi)_*" below right]
    \end{tikzcd}
  \]
  Since the fundamental classes satisfy $i_*[M]-j_*[N]=\partial[W]=0$ in
  $H_3(W)$, $\theta_\kappa(M)-\theta_\kappa(N)=0$ in
  $H_3(\widehat\Gamma/\widehat\Gamma_\kappa)$.
  
  From this, it also follows that $\theta_\kappa(M)=\theta_\kappa(N)$ in
  $H_3(\widehat\Gamma/\widehat\Gamma_\kappa) /
  \Aut(\widehat\Gamma/\widehat\Gamma_\kappa)$ even when $\theta_\kappa(M)$ and
  $\theta_\kappa(N)$ are defined using arbitrarily given isomorphisms
  $\widehat\pi/\widehat\pi_\kappa \isomto \widehat\Gamma/\widehat\Gamma_\kappa$
  and $\widehat G/\widehat G_\kappa \isomto
  \widehat\Gamma/\widehat\Gamma_\kappa$ (not necessarily the above $f$ and
  $f\circ\phi$), since the orbit of $\theta_\kappa(-)$ under the action of
  $\Aut(\widehat\Gamma/\widehat\Gamma_\kappa)$ is independent of the choice of
  the isomorphism.  This shows
  Theorem~\ref{theorem:results-homology-cob}\ref{item:homology-cob-inv-theta-mod-aut}.
\end{proof}

\section{Bordism and transfinite lower central quotients}
\label{section:bordism-and-lcsq}

The goal of this section is to prove 
Theorem~\ref{theorem:main-determination-lcsq-lift} and Corollary~\ref{corollary:main-determination-equiv-rel}.

\subsection{Proof of Theorem~\ref{theorem:main-determination-lcsq-lift}}

Recall that Theorem~\ref{theorem:main-determination-lcsq-lift} says that if $M$
is a closed 3-manifold with $\pi=\pi_1(M)$ endowed with an isomorphism $f\colon
\widehat\pi/\widehat\pi_\kappa \isomto \widehat\Gamma/\widehat\Gamma_\kappa$,
the following are equivalent:
\begin{enumerate}
  \item There exists a lift $\widehat\pi/\widehat\pi_{\kappa+1}
  \isomto \widehat\Gamma/\widehat\Gamma_{\kappa+1}$ of $f$ which is an
  isomorphism.
  \item The invariant $\theta_\kappa(M)$ vanishes in
  $\Coker\{\cR_{\kappa+1}(\Gamma) \to \cR_\kappa(\Gamma)\}$.
\end{enumerate}

In our proof, it is essential to use the fact that $H_3(-)$ is isomorphic to the
oriented bordism group~$\Omega^{SO}_3(-)$, to obtain a 4-dimensional bordism
from condition~(2).  More specifically, for another closed 3-manifold $N$ with
$G=\pi_1(N)$ equipped with $g\colon \widehat G/\widehat G_\kappa \isomto
\widehat\Gamma/\widehat\Gamma_\kappa$, we have
$\theta_\kappa(N)=\theta_\kappa(M)$ in
$H_3(\widehat\Gamma/\widehat\Gamma_\kappa)$ if and only if there is a
4-dimensional bordism $W$ between $(M,f)$ and $(N,g)$ over the
group~$\widehat\Gamma/\widehat\Gamma_\kappa$.  The core of the proof of
Theorem~\ref{theorem:main-determination-lcsq-lift} consists of careful analysis
of the relationship of such a bordism $W$ and the involved fundamental groups.

We begin with a general lemma, for which 4-dimensional duality plays a
crucial role.

\begin{lemma}
  \label{lemma:bordism-H^2-ker-inclusion}
  Suppose $W$ is a 4-dimensional cobordism between two closed 3-manifolds $M$
  and~$N$.  That is, $\partial W=N\sqcup -M$.  Suppose $A$ is an arbitrary
  abelian group.  Let $\partial\colon H_2(W,\partial W;A) \to H_1(\partial W;A)$
  be the boundary homomorphism of the long exact sequence of $(W,\partial W)$.
  If the composition
  \[
    \begin{tikzcd}[sep=4ex]
      \Im \partial \ar[r,hook] &
      H_1(\partial W;A) = H_1(M;A)\oplus H_1(N;A) \ar[r,"p",two heads] &
      H_1(M;A)
    \end{tikzcd}
  \]
  of the inclusion and the projection $p$ is injective, then 
  \[
    \Ker\{H^2(W;A)\rightarrow H^2(M;A)\} \subset
    \Ker\{H^2(W;A)\rightarrow H^2(N;A)\}.
  \]
\end{lemma}

\begin{proof}
  Consider the following diagram.
  \[
    \begin{tikzcd}[row sep=large]
      H^2(W;A) \ar[r,"k^*"] &
      H^2(\partial W;A) \ar[r,equal] &
      H^2(M;A)\oplus H^2(N;A) \ar[r,"i^*"] &
      H^2(M;A)
      \\
      H_2(W,\partial W;A) \ar[r,"\partial" below]
        \ar[u,"PD_W" left,"\cong" right] &
      H_1(\partial W;A) \ar[r,equal]
        \ar[u,"PD_{\partial W}" left,"\cong" right] &
      H_1(M;A)\oplus H_1(N;A) \ar[r,"p" below]
        \ar[u,"PD_{M}\oplus PD_{N}" left,"\cong" right] &
      H_1(M;A) \ar[u,"PD_M" left,"\cong" right]
    \end{tikzcd}
  \]
  Here $i^*$ and $k^*$ are inclusion-induced, and $PD_\bullet$ denotes the
  Poincare duality isomorphism, that is, $PD_\bullet^{-1}(c) = c\cap [\bullet]$
  where $[\bullet]$ is the fundamental class.  The left and middle squares
  commute since $\partial[W] = [\partial W] = [M]\oplus[N]$.   The right square
  commutes since $i^*$ is equal to the projection onto the first factor.
  
  We have
  \begin{equation}
    \label{equation:H^2-ker-identity-for-bordism}
    \begin{aligned}
      \Ker\{H^2(W;A)\xrightarrow{i^* k^*}H^2(M;A)\} &= PD_W(\Ker p\circ\partial)
        &&\text{by the diagram,}
      \\
      & = PD_W(\Ker \partial) &&\text{since $p|_{\Im\partial}$ is injective.}
    \end{aligned}
  \end{equation}
  Apply the same argument to $N$ in place of $M$ to obtain
  \begin{equation}
    \label{equation:H^2-ker-inclusion-for-bordism}
    \Ker\{H^2(W;A)\rightarrow H^2(N;A)\} = PD_W(\Ker q\circ\partial)
    \supset PD_W(\Ker \partial)
  \end{equation}
  where
  \[
    q\colon H_1(\partial W;A)= H_1(M;A)\oplus H_2(N;A) \to H_1(N;A)
  \]
  is the projection onto the second factor.
  From~\eqref{equation:H^2-ker-identity-for-bordism}
  and~\eqref{equation:H^2-ker-inclusion-for-bordism}, the conclusion follows
  immediately.
\end{proof}

Theorem~\ref{theorem:main-determination-lcsq-lift} will be proven as a
consequence of the following result.

\begin{theorem}
  \label{theorem:determination-next-lcsq}
  Suppose $\kappa\ge 2$ and $M$ and $N$ are closed 3-manifolds with
  $\pi=\pi_1(M)$ and $G=\pi_1(N)$ which are endowed with isomorphisms $f\colon
  \widehat\pi/\widehat\pi_\kappa \isomto \widehat\Gamma/\widehat\Gamma_\kappa$
  and $g\colon \widehat G/\widehat G_\kappa \isomto
  \widehat\Gamma/\widehat\Gamma_\kappa$.  Define $\theta_\kappa(M)$ and
  $\theta_\kappa(N)$ using $f$ and~$g$.  If $\theta_\kappa(M)=\theta_\kappa(N)$
  in $H_3(\widehat\Gamma/\widehat\Gamma_\kappa)$, then the isomorphism
  \[
    f^{-1}g\colon \widehat G/\widehat G_\kappa \xrightarrow{\cong}
    \widehat\pi/\widehat\pi_\kappa
  \]
  lifts to an isomorphism
  \[
    \widehat G/\widehat G_{\kappa+1} \xrightarrow{\cong}
    \widehat\pi/\widehat\pi_{\kappa+1}.
  \]
\end{theorem}

\begin{proof}
  Since $H_3(\widehat\Gamma/\widehat\Gamma_\kappa)$ is equal to the oriented
  bordism group $\Omega^{SO}_3(\widehat\Gamma/\widehat\Gamma_\kappa)$, there
  exists a 4-dimensional bordism $W$, over
  $\widehat\Gamma/\widehat\Gamma_\kappa$, between $M$ and~$N$.  We begin with
  some claims.

  \begin{claim-named}[Claim~1] For any abelian group $A$, the inclusions
    $i\colon M\hookrightarrow W$ and $j\colon N\hookrightarrow W$ induce
    injections $i_*\colon H_1(M;A)\to H_1(W;A)$ and $j_*\colon H_1(N;A)\to
    H_1(W;A)$.
  \end{claim-named}

  \begin{subproof}
    To show this, consider the following commutative diagram.
    \[
      \hphantom{H_1(W;A)={}}
      \begin{tikzcd}[column sep=scriptsize,row sep=large]
        \llap{$H_1(M;A)={}$} H_1(\pi;A) \ar[r,"\cong" below] \ar[d,"i_*" left] &
        H_1(\widehat\pi;A) \ar[r,"\cong" below] & 
        H_1(\widehat\pi/\widehat\pi_\kappa;A) \ar[r,"f_*","\cong" below]&
        H_1(\widehat\Gamma/\widehat\Gamma_\kappa;A)
        \\
        \llap{$H_1(W;A)={}$} H_1(\pi_1(W);A) \ar[rrru,bend right=12]
      \end{tikzcd}
    \]
    When $A=\Z$, the leftmost horizontal arrow is an isomorphism by
    Theorem~\ref{theorem:localization-basic-facts}, and the middle horizontal
    arrow is an isomorphism too since $\kappa\ge 2$.  The rightmost horizontal
    arrow, $f_*$, is an isomorphism since so is~$f$.  Therefore, the composition
    $H_1(M;A)\to H_1(\widehat\Gamma/\widehat\Gamma_\kappa;A)$ is an isomorphism
    for $A=\Z$, and consequently it is an isomorphism for an arbitrary $A$ by
    the universal coefficient theorem. From this and the above diagram, it
    follows that $i_*$ is injective.  The injectivity of $j_*$ is shown in the
    same way, using $N$ in place of~$M$. This proves Claim~1.
  \end{subproof}

  \begin{claim-named}[Claim~2]
    For any abelian group $A$,
    \[
      \Ker\{H^2(W;A)\xrightarrow{i^*} H^2(M;A)\} =
      \Ker\{H^2(W;A)\xrightarrow{j^*} H^2(N;A)\}.
    \]
  \end{claim-named}

  \begin{subproof}
    To show this, use notations of Lemma~\ref{lemma:bordism-H^2-ker-inclusion}.
    Let $\partial \colon H_2(W,\partial W;A)\to H_1(W;A)$ be the boundary map,
    and let $p$ and $q$ be the projections of $H_1(W;A) = H_1(M;A)\oplus
    H_1(N;A)$ onto the first and second factor respectively.  By
    Lemma~\ref{lemma:bordism-H^2-ker-inclusion}, it suffices to show that the
    restrictions $p|_{\Im \partial}$ and $q|_{\Im\partial}$ are injective.  In
    our case, 
    \[
      \begin{aligned}
        \Im \partial &= \Ker\{H_1(\partial W;A) \rightarrow H_1(W;A)\}
        \\
        & = \{(x,y)\in H_1(M;A)\oplus H_1(N;A) \mid i_*(x) + j_*(y)=0\}
      \end{aligned}
    \]
    where $i_*\colon H_1(M;A)\to H_1(W;A)$ and $j_*\colon H_1(N;A)\to H_1(W;A)$.
    So, for $(x,y)\in \Im\partial$, if $0=p(x,y)=x$, then $j_*(y)=-i_*(x)=0$,
    and thus $y=0$ since $j_*$ is injective by Claim~1.  This shows that
    $p|_{\Im\partial}$ is injective.  The same argument shows that
    $q|_{\Im\partial}$ is injective.  This completes the proof of Claim~2.
  \end{subproof}

  Let $A=\widehat\pi_\kappa/\widehat\pi_{\kappa+1}$, and realize the short exact
  sequence
  \[
    0\to A \to \widehat\pi/\widehat\pi_{\kappa+1} \to
    \widehat\pi/\widehat\pi_{\kappa} \to 1
  \]
  as a fibration $B(\widehat\pi/\widehat\pi_{\kappa+1}) \to
  B(\widehat\pi/\widehat\pi_{\kappa})$ with fiber $B(A)$.  We will use the
  following basic facts from obstruction theory.  A map $f\colon X\to
  B(\widehat\pi/\widehat\pi_{\kappa})$ of a CW-complex $X$ gives an obstruction
  class $o_X\in H^2(X;A)$ which vanishes if and only if there is a lift $X\to
  B(\widehat\pi/\widehat\pi_{\kappa+1})$.  In our case, the coefficient system
  $\{A\}$ is untwisted on $B(\widehat\pi/\widehat\pi_{\kappa})$ since the
  abelian subgroup $A=\widehat\pi_\kappa/\widehat\pi_{\kappa+1}$ is central
  in~$\widehat\pi/\widehat\pi_{\kappa+1}$.  So, $o_X$ determines a homotopy
  class of a map $\phi_X\colon X\to K(A,2)$, which is null-homotopic if and only
  if $f$ lifts.  Conversely, $\phi_X$ determines~$o_X$.  Namely, $o_X$ is the
  image of $\id_A$ under
  \[
    \Hom(A, A) = \Hom(H_2(K(A,2)),A) = H^2(K(A,2);A)
    \xrightarrow{(\phi_X)^*} H^2(X,A).
  \]
  By the naturality of the
  obstruction class $o_X$, $\phi_X$ is the composition
  \[
    X\xrightarrow{f} B(\widehat\pi/\widehat\pi_{\kappa}) \xrightarrow{\phi} K(A,2)
  \]
  where $\phi=\phi_{B(\widehat\pi/\widehat\pi_{\kappa})}$ is the map associated
  to the identity of $B(\widehat\pi/\widehat\pi_{\kappa})$.

  Consider the following specific lifting problem, which is for $X=N$:
  \[
    \begin{tikzcd}[row sep=large]
      & &
      B(\widehat\pi/\widehat\pi_{\kappa+1})\ar[d]
      \\
      N \ar[r] \ar[d,hook,"j"'] \ar[rru,dashed,bend left=15] &
      B(\widehat G/\widehat G_\kappa)
        \ar[r,"\simeq"',"f^{-1}g"] \ar[d,"g"',"\simeq"] &
      B(\widehat\pi/\widehat\pi_\kappa) \ar[r,"\phi"] &
      K(A,2)
      \\
      W \ar[r] &
      B(\widehat\Gamma/\widehat\Gamma_\kappa)
        \ar[ru,"f^{-1}"' yshift=1ex,"\simeq"]
    \end{tikzcd}
  \]
  Here the bottom row is obtained from that $W$ is a bordism
  over~$\widehat\Gamma/\widehat\Gamma_\kappa$.
  
  \begin{claim-named}[Claim~3]
    There exists a lift $N\to B(\widehat\pi/\widehat\pi_{\kappa+1})$.
  \end{claim-named}

  \begin{subproof}
    To prove this, note that the obstruction $o_N$ is the image of $\id_A$ under
    the composition
    \[
      \begin{tikzcd}[row sep=large,column sep=small]
        \Hom(A,A) \ar[r,equal] &
        H^2(K(A,2);A) \ar[r] &
        H^2(B(\widehat\pi/\widehat\pi_{\kappa});A) \ar[r] \ar[d,"(f^{-1})^*"'] &
        H^2(N;A)
        \\
        & & H^2(B\widehat\Gamma/\widehat\Gamma_{\kappa});A) \ar[r] &
        H^2(W;A) \ar[u,"j^*"']
      \end{tikzcd}.
    \]
    Thus, $o_N$ vanishes if and only the image of $\id_A$ in $H^2(W;A)$ lies in
    the kernel of the map $H^2(W;A) \xrightarrow{j^*} H^2(N;A)$.  To show that
    it is the case, consider the following lifting problem for $M$ in place
    of~$N$:
    \[
      \begin{tikzcd}[row sep=large]
        &
        B(\widehat\pi/\widehat\pi_{\kappa+1}) \ar[d]
        \\
        M \ar[r]\ar[d,hook,"i"'] \ar[ru,dashed,bend left=20] &
        B(\widehat\pi/\widehat\pi_\kappa) \ar[r,"\phi"] &
        K(A,2)
        \\
        W \ar[r] &
        B(\widehat\Gamma/\widehat\Gamma_\kappa)
          \ar[u,"f^{-1}"',"\simeq"]
      \end{tikzcd}
    \]
    Since $M\to B(\pi) \to B(\widehat\pi) \to
    B(\widehat\pi/\widehat\pi_{\kappa+1})$ is obviously a lift, the obstruction
    $o_M$ vanishes. On the other hand, by the above argument applied to this
    case, $o_M$ vanishes if and only if the image of $\id_A$ in $H^2(W;A)$ lies
    in the kernel of $H^2(W;A) \xrightarrow{i^*} H^2(M;A)$.  By Claim~2, it
    follows that the image of $\id_A$ is contained in the kernel of $H^2(W;A)
    \to H^2(N;A)$ as well.  That is, the obstruction $o_N$ vanishes too.  This
    proves Claim~3.
  \end{subproof}

  \begin{claim-named}[Claim 4]
    There is a lift $\alpha\colon \widehat G/\widehat G_{\kappa+1}
    \to \widehat\pi/\widehat\pi_{\kappa+1}$ of~$f^{-1}g$.
    \[
      \begin{tikzcd}[sep=large]
        \widehat G/\widehat G_{\kappa+1} \ar[r,dashed,"\alpha"] \ar[d] &
        \widehat\pi/\widehat\pi_{\kappa+1} \ar[d]
        \\
        \widehat G/\widehat G_{\kappa}\ar[r,"f^{-1}g" below,"\cong"] &
        \widehat\pi/\widehat\pi_{\kappa}
      \end{tikzcd}
    \]
  \end{claim-named}

  \begin{subproof}
    To show this, first take the homomorphism $G\to
    \widehat\pi/\widehat\pi_{\kappa+1}$ induced by the lift $N\to
    B(\widehat\pi/\widehat\pi_{\kappa+1})$ in Claim~3.  It is a lift of
    $f^{-1}g\colon \widehat G/\widehat G_\kappa \isomto
    \widehat\pi/\widehat\pi_\kappa$.  Since $\widehat\pi/\widehat\pi_{\kappa+1}$
    is local by Lemma~\ref{lemma:lcsq-of-local-group-is-local}, $G\to
    \widehat\pi/\widehat\pi_{\kappa+1}$ induces a homomorphism $\widehat G \to
    \widehat\pi/\widehat\pi_{\kappa+1}$.  It induces a desired homomorphism
    $\alpha\colon \widehat G/\widehat G_{\kappa+1} \to
    \widehat\pi/\widehat\pi_{\kappa+1}$, since $\widehat G_{\kappa+1}\subset
    \widehat G$ is sent into $(\widehat\pi/\widehat\pi_{\kappa+1})_{\kappa+1} =
    \widehat\pi_{\kappa+1}/\widehat\pi_{\kappa+1}=\{e\}$.
  \end{subproof}

  Our goal is to show that the lift $\alpha$ in Claim~4 is an isomorphism.  For
  this purpose,  exchange the roles of $M$ and $N$ and apply the same argument,
  to obtain a lift of~$g^{-1}f\colon \widehat\pi/\widehat\pi_{\kappa} \to
  \widehat G/\widehat G_{\kappa}$, and call it $\beta\colon
  \widehat\pi/\widehat\pi_{\kappa+1} \to \widehat G/\widehat G_{\kappa+1}$ .
  
  \begin{claim-named}[Claim~5]
    The composition $\alpha\beta \colon
    \widehat\pi/\widehat\pi_{\kappa+1} \to \widehat\pi/\widehat\pi_{\kappa+1}$
    is an isomorphism.
  \end{claim-named}

  \begin{subproof}
    To prove this, consider the following diagram.
    \[
      \begin{tikzcd}[row sep=large]
        1 \ar[r] & \widehat\pi_\kappa/\widehat\pi_{\kappa+1}
          \ar[r]\ar[d,"\alpha\beta|_{\widehat\pi_\kappa/\widehat\pi_{\kappa+1}}"'] &
        \widehat\pi/\widehat\pi_{\kappa+1} \ar[r]\ar[d,"\alpha\beta"] &
        \widehat\pi/\widehat\pi_\kappa \ar[r]\ar[d,"\id"] & 1
        \\
        1 \ar[r] & \widehat\pi_\kappa/\widehat\pi_{\kappa+1} \ar[r] &
        \widehat\pi/\widehat\pi_{\kappa+1} \ar[r] &
        \widehat\pi/\widehat\pi_\kappa \ar[r] & 1
      \end{tikzcd}
    \]
    Here, the right square commutes since $\alpha$ and $\beta$ are lifts of
    $f^{-1}g$ and $g^{-1}f$ and thus $\alpha\beta$ is a lift of the identity. By
    the five lemma, if the left vertical arrow
    $\widehat\pi_\kappa/\widehat\pi_{\kappa+1} \to
    \widehat\pi_\kappa/\widehat\pi_{\kappa+1}$ is an isomorphism, then
    $\alpha\beta$ is an isomorphism too.  We will indeed show that $\alpha\beta$
    restricts to the identity on the (larger)
    subgroup~$\widehat\pi_2/\widehat\pi_{\kappa+1}$.  Suppose $g\in
    \widehat\pi_2/\widehat\pi_{\kappa+1}$.  Write $g$ as a product $g=\prod_i
    [a_i,b_i]$ of commutators, where $a_i$, $b_i\in
    \widehat\pi/\widehat\pi_{\kappa+1}$.  Since $\alpha\beta$ is a lift of the
    identity of~$\widehat\pi/\widehat\pi_{\kappa}$, we have
    $\alpha\beta(a_i)=a_i u_i$ and $\alpha\beta(b_i)=b_i v_i$ for some $u_i$,
    $v_i\in \widehat\pi_\kappa/\widehat\pi_{\kappa+1}$.  Since
    $\widehat\pi_\kappa/\widehat\pi_{\kappa+1}$ is central in
    $\widehat\pi/\widehat\pi_{\kappa+1}$, $[a_iu_i,b_iv_i]=[a_i,b_i]$.  It
    follows that
    \[
      \alpha\beta(g)=\prod_i [\alpha\beta(a_i),\alpha\beta(b_i)]
      = \prod_i [a_i u_i, b_i v_i] = \prod_i [a_i,b_i] = g.
    \]
    This completes the proof of Claim~5.
  \end{subproof}

  Now, by Claim~5, $\alpha$ is injective and $\beta$ is surjective. Exchange the
  roles of $\alpha$ and $\beta$ and apply the same argument, to conclude that
  $\alpha$ is surjective and $\beta$ is injective.   Therefore $\alpha$ and
  $\beta$ are isomorphisms.  This completes the proof of
  Theorem~\ref{theorem:determination-next-lcsq}.
\end{proof}

\begin{proof}[Proof of Theorem~\ref{theorem:main-determination-lcsq-lift}]
  Suppose $M$ is a closed 3-manifold with $\pi=\pi_1(M)$, which is endowed with
  an isomorphism $f\colon \widehat\pi/\widehat\pi_\kappa \isomto
  \widehat\Gamma/\widehat\Gamma_\kappa$.  Suppose $f$ lifts to an isomorphism
  $\tilde f\colon \widehat\pi/\widehat\pi_{\kappa+1} \isomto
  \widehat\Gamma/\widehat\Gamma_{\kappa+1}$.  Then $\theta_{\kappa+1}(M,\tilde f)$ is sent to
  $\theta_\kappa(M,f)$ under $\cR_{\kappa+1}(\Gamma)\to \cR_\kappa(\Gamma)$.
  Therefore $\theta_{\kappa}(M)=\theta_{\kappa+1}(M,\tilde f)$ vanishes in the cokernel of
  $\cR_{\kappa+1}(\Gamma)\to \cR_\kappa(\Gamma)$.

  For the converse, suppose $\theta_{\kappa}(M)=\theta_{\kappa}(M,f)$ vanishes in the cokernel of
  $\cR_{\kappa+1}(\Gamma)\to \cR_\kappa(\Gamma)$. This means that there is a
  closed 3-manifold $N$ with $G=\pi_1(N)$ which is endowed with an isomorphism
  $\tilde g\colon \widehat G/\widehat G_{\kappa+1} \isomto
  \widehat\Gamma/\widehat\Gamma_{\kappa+1}$, such that $\theta_{\kappa+1}(N,\tilde g)$ is
  sent to $\theta_\kappa(M,f)$ under $\cR_{\kappa+1}(\Gamma)\to
  \cR_\kappa(\Gamma)$.  Let $g\colon \widehat
  G/\widehat G_{\kappa} \isomto \widehat\Gamma/\widehat\Gamma_{\kappa}$ be induced by~$\tilde g$.  Then $\theta_\kappa(N,g)=\theta_\kappa(M,f)$, and thus it follows
  that $g^{-1}f$ lifts to an isomorphism $\widehat\pi/\widehat\pi_{\kappa+1}
  \isomto \widehat G/\widehat G_{\kappa+1}$,  by
  Theorem~\ref{theorem:determination-next-lcsq}.  Compose this lift
  with $\tilde g\colon \widehat G/\widehat G_{\kappa+1} \isomto
  \widehat\Gamma/\widehat\Gamma_{\kappa+1}$, to obtain an isomorphism
  $\widehat\pi/\widehat\pi_{\kappa+1} \isomto
  \widehat\Gamma/\widehat\Gamma_{\kappa+1}$ which is a lift of~$f$.
\end{proof}

\subsection{Proof of Corollary~\ref{corollary:main-determination-equiv-rel}}
\label{subsection:proof-main-determination-rel}

Recall from Definition~\ref{definition:equiv-rel-realzable-classes} that the
equivalence relation $\sim$ on $\cR_\kappa(\Gamma)$ is defined as follows.  For
$\theta \in \cR_\kappa(\Gamma)$, there is a closed 3-manifold $M$ with
$\pi=\pi_1(M)$, which is equipped with an isomorphism $f\colon
\widehat\pi/\widehat\pi_{\kappa} \isomto
\widehat\Gamma/\widehat\Gamma_{\kappa}$, such that $\theta_\kappa(M)=\theta$.
Let $I_\theta$ be the image of
\begin{equation}
  \label{equation:projection-from-pi-to-Gamma}
  \cR_{\kappa+1}(\pi) \to \cR_\kappa(\pi) \xrightarrow[f_*]{\approx}
  \cR_\kappa(\Gamma).
\end{equation}

\begin{lemma}
  \label{lemma:partition-on-realizable-classes}
  The set $I_\theta$ is well-defined, and $I_\phi=I_\theta$ whenever $\phi\in
  I_\theta$.
\end{lemma}

From Lemma~\ref{lemma:partition-on-realizable-classes}, it follows that the sets
$I_\theta$ form a partition of~$\cR_\kappa(\Gamma)$. On~$\cR_\kappa(\Gamma)$, we
write $\theta\sim\phi$ if $I_\theta=I_\phi$.

\begin{proof}[Proof of Lemma~\ref{lemma:partition-on-realizable-classes}]
  Let $\theta\in \cR_\kappa(\Gamma)$ and $(M,f)$ be as above, and suppose $N$ is a closed 3-manifold with $G=\pi_1(N)$ equipped with an
  isomorphism $g\colon \widehat G/\widehat G_{\kappa} \isomto
  \widehat\Gamma/\widehat\Gamma_{\kappa}$ such that $\theta_\kappa(N,g)$ lies in
  the image of the map~\eqref{equation:projection-from-pi-to-Gamma}.  Then there
  is an isomorphism lift $\widetilde{f^{-1}g}\colon \widehat G/\widehat G_{\kappa+1} \isomto
  \widehat\pi/\widehat\pi_{\kappa+1}$ of $f^{-1}g\colon \widehat G/\widehat
  G_\kappa \isomto \widehat\pi/\widehat\pi_\kappa$, by
  Theorem~\ref{theorem:main-determination-lcsq-lift} applied to~$(N,f^{-1}g)$.
  The induced functions on $\cR_*(-)$ form the following commutative diagram,
  where all horizontal arrows are bijective.
  \[
    \begin{tikzcd}[row sep=large,column sep=scriptsize]
      \cR_{\kappa+1}(G) \ar[rr,"\approx","\widetilde{(f^{-1}g)}_*"'] \ar[d] & &
      \cR_{\kappa+1}(\pi) \ar[d]
      \\
      \cR_{\kappa}(G) \ar[r,"\approx","g_*"'] &
      \cR_{\kappa}(\Gamma) &
      \cR_{\kappa}(\pi) \ar[l,"\approx"',"f_*"] &      
    \end{tikzcd}
  \]
  Indeed, if $P$ is a closed 3-manifold equipped with $\widetilde h\colon \widehat{\pi_1(P)} / \widehat{\pi_1(P)}_{\kappa+1} \isomto \widehat G / \widehat G_{\kappa+1}$ which induces $h\colon \widehat{\pi_1(P)} / \widehat{\pi_1(P)}_{\kappa} \isomto \widehat G / \widehat G_{\kappa}$, then the images of $\theta_{\kappa+1}(P,\widetilde h) \in \cR_{\kappa+1}(G)$ under the arrows in the above diagram are given by several $\theta$-invariants of the same $P$, as shown below:
  \[
    \begin{tikzcd}[row sep=large,column sep=scriptsize]
      \theta_{\kappa+1}(P,\widetilde h)
      \ar[rr,|->,"\widetilde{(f^{-1}g)}_*"] \ar[d,|->] & &
      \theta_{\kappa+1}(P,\widetilde{(f^{-1}g)}\circ\widetilde h) \ar[d,|->]
      \\
      \theta_\kappa(P,h) \ar[r,|->,"g_*"] &
      \theta_{\kappa}(P,gh) &
      \theta_{\kappa}(P,f^{-1}gh) \ar[l,|->,"f_*"'] &      
    \end{tikzcd}
  \]
  As an immediate consequence of the commutativity, we have
  \begin{equation}
    \label{equation:image-equality-for-realizable-classes}
    \Im\{\cR_{\kappa+1}(G) \rightarrow \cR_\kappa(G) \rightarrow
    \cR_\kappa(\Gamma) \} = 
    \Im\{\cR_{\kappa+1}(\pi) \rightarrow \cR_\kappa(\pi) \rightarrow
    \cR_\kappa(\Gamma) \}.
  \end{equation}

  Now, writing $\phi=\theta_\kappa(N,g)$, the left and right hand sides of~\eqref{equation:image-equality-for-realizable-classes} are $I_\phi$ and~$I_\theta$, respectively.
  This shows the assertion $I_\phi=I_\theta$.
  Also, when $\theta_\kappa(N,g)=\theta$, \eqref{equation:image-equality-for-realizable-classes} shows that $I_\theta$ is well-defined, independent of the choice of~$(M,f)$.
\end{proof}

Once we formulate the above setup, it is rather straightforward to obtain
Corollary~\ref{corollary:main-determination-equiv-rel}, which asserts the
following: suppose $M$ and $N$ are closed 3-manifolds with fundamental groups
$\pi=\pi_1(M)$ and $G=\pi_1(N)$, which are equipped with isomorphisms $f\colon
\widehat\pi/\widehat\pi_\kappa \isomto \widehat\Gamma/\widehat\Gamma_\kappa$ and
$g\colon \widehat G/\widehat G_\kappa \isomto
\widehat\Gamma/\widehat\Gamma_\kappa$. Then, $f^{-1}g$ lifts to an isomorphism
$G/\widehat G_{\kappa+1} \isomto \widehat\pi/\widehat\pi_{\kappa+1}$  if and
only if $\theta_\kappa(M) \sim \theta_\kappa(N)$ in~$\cR_\kappa(\Gamma)$.

\begin{proof}[Proof of Corollary~\ref{corollary:main-determination-equiv-rel}]
  
  Let $f_*\colon \cR_\kappa(\pi) \to \cR_\kappa(\Gamma)$ be the induced
  bijection. By definition, $\theta_\kappa(N)$ lies in the subset
  $I_{\theta_\kappa(M)}$ of $\cR_\kappa(\Gamma)$ if and only if $f_*^{-1}
  \theta_\kappa(N) \in \cR_\kappa(\pi)$ lies in the image of
  $\cR_{\kappa+1}(\pi) \to \cR_\kappa(\pi)$; in other words, $f_*^{-1}
  \theta_\kappa(N) = 0$ in the cokernel of $\cR_{\kappa+1}(\pi) \to
  \cR_\kappa(\pi)$.  It is the case if and only if $f^{-1}g$ lifts to an
  isomorphism $\widehat G/\widehat G_{\kappa+1} \isomto
  \widehat\pi/\widehat\pi_{\kappa+1}$, by applying
  Theorem~\ref{theorem:main-determination-lcsq-lift} to~$(N,f^{-1}g)$.
\end{proof}

\section{Transfinite Stallings-Dwyer theorem and transfinite gropes}
\label{section:transfinite-stallings-dwyer}

The goal of this section is to provide \emph{transfinite} generalizations of a
well known result of Stallings \cite{Stallings:1965-1} and
Dwyer~\cite{Dwyer:1975-1}, and relate it with a notion of transfinite gropes
which we define in this section too.  In
Section~\ref{subsection:transfinite-grope}, we prove
the {\em Addendum to Theorems~\ref{theorem:main-determination-lcsq-nonlift} and~\ref{theorem:main-milnor-inv}}, as stated in Section~\ref{subsection:results-transfinite-grope}.  The transfinite generalizations of
the Stallings-Dwyer theorem will also be used in the proof of realization
theorems in Section~\ref{section:realization-transfinite-inv}.

\subsection{Algebraic statements}
\label{subsection:homology-and-transfinite-lcsq}

\begin{theorem}[Transfinite Stallings-Dwyer]
  \label{theorem:transfinite-stallings-dwyer}
  Let $\kappa>1$ be an arbitrary ordinal.  Suppose $f\colon \pi\to G$ is a group
  homomorphism inducing an isomorphism $H_1(\pi) \isomto H_1(G)$.  If $\kappa$
  is an infinite ordinal, suppose $G$ is finitely generated.  Then $f$ induces
  an isomorphism $\widehat\pi/\widehat\pi_{\kappa} \isomto \widehat G/\widehat
  G_{\kappa}$ if and only if $f$ induces an epimorphism
  \begin{equation}
      \label{equation:transfinite-stallings-dwyer-H2-map}
      H_2(\widehat\pi)\to H_2(\widehat G)/\Ker\{H_2(\widehat G)\rightarrow
      H_2(\widehat G/\widehat G_{\lambda})\}
  \end{equation}
  for all ordinals $\lambda<\kappa$.
\end{theorem}

Note that if $\kappa$ is a discrete ordinal, then the homomorphism
\eqref{equation:transfinite-stallings-dwyer-H2-map} is surjective for all
$\lambda<\kappa$ if and only if it is surjective for $\lambda=\kappa-1$.
Especially, if $\kappa$ is finite, then by
Corollary~\ref{corollary:localization-basic-consequences}\ref{item:local-lcsq-ordinary-lcsq},
Theorem~\ref{theorem:transfinite-stallings-dwyer} specializes to the
Stallings-Dwyer theorem~\cite{Stallings:1965-1, Dwyer:1975-1}: \emph{for a
homomorphism $f\colon \pi\to G$ which induces an isomorphism $H_1(\pi) \isomto
H_1(G)$, $f$ induces an isomorphism $\pi/\pi_{k}
\cong G/G_{k}$ if and only if $f$ induces an epimorphism}
\[
    H_2(\pi)\to H_2(G)/\Ker\{H_2(G)\rightarrow H_2(G/G_{k-1})\}.
\]

Before proving Theorem~\ref{theorem:transfinite-stallings-dwyer} in Section~\ref{subsection:proof-transfinite-stallings-dwyer}, we record some consequences.  We will use the following notation, which is a transfinite generalization of the notation used in Dwyer~\cite[p.~178]{Dwyer:1975-1}.

\begin{definition}[Transfinite Dwyer kernel]
  
  Suppose $G$ is a group, and $\kappa>1$ is an ordinal.  The \emph{transfinite
  Dwyer kernel} is defined by
  \[
    \psi_\kappa(G) = 
    \begin{cases}
      \Ker\{H_2(G)\to H_2(G/G_{\kappa-1})\}
      &\text{if $\kappa$ is a discrete ordinal,}
      \\
      \bigcap_{\lambda<\kappa} \psi_\lambda(G)
      &\text{if $\kappa$ is a limit ordinal.}
    \end{cases}
  \]
  More generally, for a space $X$ with $\pi=\pi_1(X)$, define $\psi_\kappa(X)$
  by
  \[
    \psi_\kappa(X) = 
    \begin{cases}
      \Ker\{H_2(X)\to H_2(\pi)\to H_2(\pi/\pi_{\kappa-1})\}
      &\text{if $\kappa$ is a discrete ordinal,}
      \\
      \bigcap_{\lambda<\kappa} \psi_\lambda(X)
      &\text{if $\kappa$ is a limit ordinal.}
    \end{cases}
  \]
  That is, $\psi_\kappa(BG)=\psi_\kappa(G)$.  
\end{definition}

\begin{corollary}
  \label{corollary:transfinite-stallings-dwyer}
  Suppose $f\colon \pi \to G$ induces an isomorphism $H_1(\pi)\to H_1(G)$,
  $\kappa>1$, and suppose $G$ is finitely presented if $\kappa$ is transfinite.
  If
  \begin{equation}
    \label{equation:transfinite-dwyer-kernel-H2-map}
    H_2(\widehat\pi) \xrightarrow{\hat f_*} H_2(\widehat G) 
    \to H_2(\widehat G)/\psi_\kappa(\widehat G)
  \end{equation}
  is surjective, then $f$ induces an isomorphism $\widehat\pi/\widehat\pi_\kappa
  \isomto \widehat G/\widehat G_\kappa$ 
\end{corollary}

Note that Corollary~\ref{corollary:transfinite-stallings-dwyer} assumes the
surjectivity of a single
homomorphism~\eqref{equation:transfinite-dwyer-kernel-H2-map}, instead of the
surjectivity of infinitely many
homomorphisms~\eqref{equation:transfinite-stallings-dwyer-H2-map} in
Theorem~\ref{theorem:transfinite-stallings-dwyer}, for the limit ordinal case.

\begin{proof}
  If $\kappa$ is a discrete ordinal, the codomain
  of~\eqref{equation:transfinite-dwyer-kernel-H2-map} is equal to that
  of~\eqref{equation:transfinite-stallings-dwyer-H2-map}, and thus the corollary
  follows from Theorem~\ref{theorem:transfinite-stallings-dwyer}.  If $\kappa$
  is a limit ordinal, $H_2(\widehat G)/\psi_\kappa(\widehat G)$ surjects onto
  $H_2(\widehat G)/\Ker\{H_2(\widehat G)\to H_2(\widehat G/\widehat
  G_\lambda)\}$ for all $\lambda<\kappa$.  From this and
  Theorem~\ref{theorem:transfinite-stallings-dwyer}, the corollary follows.
\end{proof}

In practice, it may be difficult to verify the hypothesis that
\eqref{equation:transfinite-stallings-dwyer-H2-map}
or~\eqref{equation:transfinite-dwyer-kernel-H2-map} is surjective, since
localizations are involved.  The following variation does not involve
localizations in the hypothesis.

\begin{corollary}
  \label{corollary:transfinite-stallings-dwyer-without-localization-in-H_2}
  Suppose $f\colon \pi \to G$ induces an isomorphism $H_1(\pi)\to H_1(G)$.
  Suppose $\kappa>1$ and $G$ is finitely presented.  If
  \begin{equation*}
    H_2(\pi)\xrightarrow{f_*} H_2(G) \to H_2(G)/\psi_\kappa(G)
  \end{equation*}
  is surjective, then $f$ induces an isomorphism $\widehat\pi/\widehat\pi_\kappa
  \isomto \widehat G/\widehat G_\kappa$.
\end{corollary}

\begin{proof}
  Consider the following commutative diagram.
  \[
    \begin{tikzcd}[row sep=large]
      H_2(\pi) \ar[r] \ar[d] &
      H_2(G) \ar[r] \ar[d,two heads] &
      H_2(G)/\psi_\kappa(G) \ar[d,two heads]
      \\
      H_2(\widehat\pi) \ar[r] &
      H_2(\widehat G) \ar[r] &
      H_2(\widehat G)/\psi_\kappa(\widehat G)
    \end{tikzcd}
  \]
  Since $G$ is finitely presented, the middle vertical arrow is surjective by
  Theorem~\ref{theorem:localization-basic-facts}\ref{item:localization-colimit},
  and consequently the right vertical arrow is surjective.  It follows that
  bottom horizontal composition is surjective if the top horizontal composition
  is surjective.  So Corollary~\ref{corollary:transfinite-stallings-dwyer}
  implies
  Corollary~\ref{corollary:transfinite-stallings-dwyer-without-localization-in-H_2}.
\end{proof}

\subsection{Transfinite gropes}
\label{subsection:transfinite-grope}

In this subsection we relate transfinite lower central quotients to a
transfinite version of gropes, using the results in
Section~\ref{subsection:homology-and-transfinite-lcsq}.  The main statement is
Corollary~ \ref{corollary:transfinite-lcsq-grope}.  This is a
transfinite generalization of the finite case approach of Freedman and
Teichner~\cite[Section~2]{Freedman-Teichner:1995-2}.

We begin with new definitions.  In what follows, a \emph{symplectic basis} on a
surface of genus $g$ designates a collection of simple closed curves $a_i$,
$b_i$ ($i=1,\ldots,g$) such that $a_i$ and $b_i$ are transverse and intersect
exactly once for all $i$ and $(a_i\cup b_i)\cap(a_j\cup b_j)=\emptyset$ for
$i\ne j$.

\begin{definition}[Transfinite gropes]
  \phantomsection\label{definition:transfinite-grope}
  \leavevmode\Nopagebreak
  \begin{enumerate}
    \item Suppose $\Sigma \to X$ is a map of a connected surface $\Sigma$ with
    connected or empty boundary to a space~$X$.  For a discrete ordinal
    $\kappa>1$, we say that the map $\Sigma \to X$ \emph{supports a grope of
    class~$\kappa$}, or shortly \emph{supports a $\kappa$-grope}, if there is a
    symplectic basis $\{a_i,b_i\}$ on $\Sigma$ such that $a_i$ bounds a
    $(\kappa-1)$-grope in $X$, in the sense defined below, for each~$i$.
    \item A loop $\gamma$ in $X$ \emph{bounds a grope of class $\lambda$}, that
    is, \emph{bounds a $\lambda$-grope}, if either
    \begin{enumerate}
      \item $\lambda=1$,
      \item $\lambda>1$ is a discrete ordinal and there is a
      map of a surface to $X$ which is bounded by $\gamma$ and supports a
      $\lambda$-grope, or
      \item $\lambda$ is a limit ordinal and $\gamma$ bounds a $\mu$-grope for
      each $\mu<\lambda$.
    \end{enumerate}
  \end{enumerate}
\end{definition}

\begin{definition}[The grope class of a second homology class]
  \label{definition:class-of-homology-class-in-H_2}
  Let $\kappa>1$.  A homology class $\sigma\in H_2(X)$ is \emph{represented by
  a $\kappa$-grope}, or \emph{is of class~$\kappa$}, if either
  \begin{enumerate}
    \item $\kappa$ is a discrete ordinal and $\sigma$ is represented by a map
    of a closed surface supporting a $\kappa$-grope, or
    \item $\kappa$ is a limit ordinal and $\sigma$
    is represented by a $\lambda$-grope for every $\lambda<\kappa$.
  \end{enumerate}
\end{definition}

We remark that for finite $k$, if $\Sigma\to X$ supports a $k$-grope in our
sense, then a map of a $k$-grope in the sense
of~\cite[Section~2]{Freedman-Teichner:1995-2} is obtained by stacking the
inductively given surfaces along basis curves, and vice versa.

\begin{proposition}
  \phantomsection\label{proposition:transfinite-grope-lcs-dwyer-kernel}
  \leavevmode\Nopagebreak
  \begin{enumerate}
    \item\label{item:transfinite-grope-lcs}
    For $\kappa\ge 1$, a loop $\gamma$ in $X$ bounds a $\kappa$-grope if and
    only if $[\gamma]\in \pi_1(X)_\kappa$.
    \item\label{item:transfinite-grope-dwyer-kernel}
    For $\kappa>1$, a class $\sigma\in H_2(X)$ lies in the transfinite Dwyer
    kernel $\psi_\kappa(X)$ if and only if $\sigma$ is represented by a
    $\kappa$-grope.    
  \end{enumerate}
\end{proposition}

For finite $\kappa$,
Proposition~\ref{proposition:transfinite-grope-lcs-dwyer-kernel}\ref{item:transfinite-grope-dwyer-kernel}
appeared earlier in~\cite[Lemma~2.3]{Freedman-Teichner:1995-2}.

The following is an immediate consequence of Corollary~\ref{corollary:transfinite-stallings-dwyer-without-localization-in-H_2} and Proposition~\ref{proposition:transfinite-grope-lcs-dwyer-kernel}.

\begin{corollary}
  \label{corollary:transfinite-lcsq-grope}
  Suppose $\kappa>1$ is an arbitrary ordinal and $f\colon X\to Y$ is a map of a
  space $X$ to a finite CW complex $Y$ which induces an isomorphism $H_1(X)
  \isomto H_1(Y)$.  If $\Coker\{H_2(X)\to H_2(Y)\}$ is generated by classes
  represented by $\kappa$-gropes in $Y$, then $f$ induces an isomorphism
  \[
    \widehat{\pi_1(X)}/\widehat{\pi_1(X)}_{\kappa} \isomto
    \widehat{\pi_1(Y)}/\widehat{\pi_1(Y)}_{\kappa}.
  \]
\end{corollary}

\begin{proof}[Proof of
  Proposition~\ref{proposition:transfinite-grope-lcs-dwyer-kernel}]
  
  From the definitions, \ref{item:transfinite-grope-lcs} follows
  straightforwardly by transfinite induction. To prove (2), we proceed by
  transfinite induction too.  Since $\pi_1(X)_{1}=\pi_1(X)$, every $\sigma\in
  H_2(X)$ lies in~$\psi_2(X)$, and is represented by a $2$-grope. So, (2)~holds
  for $\kappa=2$.
  
  Suppose $\kappa>2$ and (2) holds for all ordinals less than~$\kappa$.  If
  $\kappa$ is a limit ordinal, then by definition, $\sigma\in H_2(X)$ is in
  $\psi_\kappa(X)$ if and only $\sigma\in \psi_\lambda(X)$ for all
  $\lambda<\kappa$.  By the induction hypothesis, it holds if and only if
  $\sigma$ is represented by a $\lambda$-grope for all $\lambda<\kappa$.  By the
  definition, it holds if and only if $\sigma$ is represented by a
  $\kappa$-grope.  This shows that (2) holds for~$\kappa$.

  If $\kappa>2$ is a discrete ordinal, the finite case argument given
  in~\cite[Proof of Lemma~2.3]{Freedman-Teichner:1995-2} can be carried out.  We
  provide details for the reader's convenience.  Let $\pi=\pi_1(X)$.  Suppose
  $\sigma\in H_2(X)$ is represented by a $\kappa$-grope, that is, $\sigma$ is
  the class of a map $\Sigma\to X$ of a surface admitting geometrically
  symplectic basis $\{a_i,b_i\}$ such that each $a_i$ bounds a
  $(\kappa-1)$-grope in~$X$. By~(1), $[a_i]\in \pi_1(X)_{\kappa-1}$, and so
  $a_i$ is null-homotopic in $B(\pi/\pi_{\kappa-1})$.  By surgery on $\Sigma$
  along the~$a_i$, it follows that the image of $\sigma$ in
  $H_2(B(\pi/\pi_{\kappa-1}))$ is a spherical class, and thus trivial. This
  shows that $\sigma$ lies in $\psi_\kappa(X)$.  For the converse, suppose a
  class represented by a map $\Sigma\to X$ of a surface $\Sigma$ lies
  in~$\psi_\kappa(X)$.  Attach 2-cells to $X$ along generators of
  $\pi_{\kappa-1}$, and attach more cells of dimension${}\ge3$, to
  construct~$B(\pi/\pi_{\kappa-1})$. Since $\Sigma$ is null-homologous in
  $B(\pi/\pi_{\kappa-1})$ (and since $H_3=\Omega_3^{SO}$), $\Sigma \to X
  \hookrightarrow B(\pi/\pi_{\kappa-1})$ extends to a compact 3-manifold $R$
  bounded by~$\Sigma$.  We may assume that the center of each cell which we
  attached to $X$ is a regular value of $R\to B(\pi/\pi_{\kappa-1})$.  Remove,
  from $R$, tubular neighborhoods of the inverse images of the centers.  This
  gives a bordism over $X$ between $\Sigma \to X$ and a map of a union of tori
  and spheres.  Spheres support a $\kappa$-grope by definition.  Since the
  meridian of each torus bounds a disk in $B(\pi/\pi_{\kappa-1})$, the meridian
  bounds a $(\kappa-1)$-grope in $X$ by~(1). By definition, it follows that the
  tori support a $\kappa$-grope. This completes the proof.
\end{proof}

As an application, we give a proof of the addendum to
Theorems~\ref{theorem:main-determination-lcsq-nonlift}
and~\ref{theorem:main-milnor-inv} stated in
Section~\ref{subsection:results-transfinite-grope}.  We first define a
terminology used in the statement.  Recall that a cobordism $W$ between $M$ and
$N$ is an \emph{$H_1$-cobordism} if inclusions induce isomorphisms $H_1(M)\cong
H_1(W) \cong H_1(N)$. 

\begin{definition}
  \label{definition:cobordism-of-class-kappa}
  Let $\kappa$ be an ordinal.  An $H_1$-cobordism $W$ between $M$ and $N$ is a
  \emph{grope cobordism of class $\kappa$} if each of $\Coker\{H_2(M)\to
  H_2(W)\}$ and $\Coker\{H_2(N)\to H_2(W)\}$ is generated by homology classes in
  $H_2(W)$ represented by $\kappa$-gropes.
\end{definition}

Now, the addendum to Theorems~\ref{theorem:main-determination-lcsq-nonlift}
and~\ref{theorem:main-milnor-inv} says the following: let $\Gamma$ be a group
and $\kappa$ be an arbitrarily given ordinal. Suppose $M$ is a closed 3-manifold
with $\pi=\pi_1(M)$ which is equipped with an isomorphism
$\widehat\pi/\widehat\pi_\kappa \isomto \widehat\Gamma/\widehat\Gamma_\kappa$.
Then the following are equivalent.
\begin{enumerate}[start=0]
  \item There is a grope cobordism of class $\kappa+1$ between $M$ and another
  closed 3-manifold $N$ satisfying $\widehat{\pi_1(N)} /
  \widehat{\pi_1(N)}_{\kappa+1} \cong \widehat\Gamma /
  \widehat\Gamma_{\kappa+1}$.
  \item $\widehat\pi/\widehat\pi_{\kappa+1}$ is isomorphic to
  $\widehat\Gamma/\widehat\Gamma_{\kappa+1}$.
  \item The invariant $\theta_\kappa(M)$ vanishes in
  $\Coker\{\cR_{\kappa+1}(\Gamma) \to \cR_{\kappa}(\Gamma) /
  \Aut(\widehat\Gamma/\widehat\Gamma_\kappa)\}$.
\end{enumerate}

\begin{proof}
  We have already shown that (1) and~(2) are equivalent in
  Section~\ref{subsection:results-determination}.  Suppose (1) holds.  Then
  $M\times [0,1]$ is a grope cobordism of class $\kappa+1$, and thus (0) holds.
  For the converse, suppose $W$ is a grope cobordism of class $\kappa+1$ given
  in~(0). Since $\Coker\{H_2(M)\to H_2(W)\}$ and $\Coker\{H_2(N)\to H_2(W)\}$
  are generated by $(\kappa+1)$-gropes and $H_1(M)\cong H_1(W)\cong H_1(N)$, we
  have
  \[
    \widehat\pi/\widehat\pi_{\kappa+1} \cong
    \widehat{\pi_1(W)}/\widehat{\pi_1(W)}_{\kappa+1} \cong 
    \widehat{\pi_1(N)}/\widehat{\pi_1(N)}_{\kappa+1}
  \]
  by Corollary~\ref{corollary:transfinite-lcsq-grope}.  
  It follows that (1) holds.
\end{proof}

\subsection{Proof of the algebraic statement}
\label{subsection:proof-transfinite-stallings-dwyer}

Now, we prove the main algebraic statement of this section.

\begin{proof}[Proof of Theorem~\ref{theorem:transfinite-stallings-dwyer}]
 
  First, we assert that the surjectivity of $H_1(\pi) \to H_1(G)$ implies that
  $\widehat\pi/\widehat\pi_\kappa \to \widehat G/\widehat G_\kappa$ is
  surjective.  Indeed, if $\kappa$ is finite, then the assertion is a well known
  fact obtained from a standard commutator identity.  For the reader's
  convenience, we describe an outline of the argument.  If $a_i\equiv b_i
  \bmod G_2$, then we have
  \[
    [a_1,[a_2,\ldots,[a_{k-1},a_k]\ldots{}]] \equiv
    [b_1,[b_2,\ldots,[b_{k-1},b_k]\ldots{}]] \mod G_{k+1}.
  \]
  From this it follows that $\pi_k/\pi_{k+1}\to G_k/G_{k+1}$ is surjective for
  all finite~$k$.  The surjectivity of $\pi/\pi_{k}\to G/G_{k}$ is obtained by
  applying the five lemma, inductively, to the following diagram.
  \[
    \begin{tikzcd}[row sep=large]
      1 \ar[r] & \pi_{k-1}/\pi_k \ar[r]\ar[d] & \pi/\pi_k \ar[r]\ar[d] &
      \pi/\pi_{k-1} \ar[r]\ar[d] & 1
      \\
      1 \ar[r] & G_{k-1}/G_k \ar[r] & G/G_k \ar[r] &
      G/G_{k-1} \ar[r] & 1
    \end{tikzcd}
  \]

  When $\kappa$ is an infinite ordinal, since $G$ is assumed to be finitely
  generated, $\widehat\pi \to \widehat G$ is surjective if $H_1(\pi) \to H_1(G)$
  is surjective, by Lemma~\ref{lemma:H1-surj-localization-surj}. It
  follows that $\widehat\pi/\widehat\pi_\kappa \to \widehat G/\widehat G_\kappa$
  is surjective.

  Therefore, under the assumption that $\widehat\pi/\widehat\pi_\kappa \to
  \widehat G/\widehat G_\kappa$ is surjective, it suffices to prove that the
  following two conditions are equivalent:
  \begin{enumerate}
    \item[(i)${}_\kappa$] $\widehat\pi/\widehat\pi_\kappa \to \widehat
    G/\widehat G_\kappa$ is injective, or equivalently is an isomorphism.
    \item[(ii)${}_\kappa$] $H_2(\widehat\pi) \to H_2(\widehat G) /
    K_{\lambda}(\widehat G)$ is surjective for all $\lambda<\kappa$,  where
    \[
      K_\lambda(\widehat G):=\Ker\{H_2(\widehat G)\rightarrow H_2(\widehat
      G/\widehat G_\lambda)\}.
    \]
  \end{enumerate}

  We proceed by transfinite induction on the ordinal~$\kappa$. For $\kappa=2$,
  (i)${}_\kappa$ holds since $\widehat\pi/\widehat\pi_2=H_1(\pi)\cong H_1(G) =
  \widehat G/\widehat G_2$, and (ii)${}_\kappa$ holds too, since $H_2(\widehat
  G)/K_1(\widehat G)$ is trivial.

  Fix an ordinal $\kappa\ge 3$, and let $f\colon \widehat\pi \to \widehat G$ be
  a homomorphism which satisfies the hypothesis of
  Theorem~\ref{theorem:transfinite-stallings-dwyer}.  Suppose that
  (i)${}_{\kappa'}$ and (ii)${}_{\kappa'}$ are equivalent for all
  $\kappa'<\kappa$.

  If $\kappa$ is a discrete ordinal, then we proceed similarly to the original
  argument of Stallings and Dwyer~\cite{Stallings:1965-1,Dwyer:1975-1}, as
  described below. First, note that (ii)${}_\lambda$ holds for all
  $\lambda<\kappa$ if and only if (ii)${}_{\kappa-1}$ holds, when $\kappa$ is
  discrete.  Recall, for a normal subgroup $N$ of a group $\Gamma$, the
  Lyndon-Hochschild-Serre spectral sequence for the short exact sequence $1 \to
  N \to \Gamma \to \Gamma/N \to 1$ gives rise to an exact sequence
  \[
    H_2(\Gamma)\to H_2(\Gamma/N) \to H_0(\Gamma/N;H_1(N))
    \to H_1(\Gamma) \to H_1(\Gamma/N)
  \]
  which is called Stallings' exact sequence~\cite{Stallings:1965-1}.  Apply this
  to $(\Gamma,N)=(\widehat\pi,\widehat\pi_{\kappa-1})$ and $(\widehat G,\widehat
  G_{\kappa-1})$, to obtain the following diagram with exact rows.
  \[
    \begin{tikzcd}[column sep=scriptsize,row sep=large]
      0 \ar[r] &
      H_2(\widehat\pi) / K_{\kappa-1}(\widehat\pi) \ar[r] \ar[d] &
      H_2(\widehat\pi/\widehat\pi_{\kappa-1}) \ar[r] \ar[d] &
      \widehat\pi_{\kappa-1}/\widehat\pi_{\kappa} \ar[r] \ar[d] & 0
      \\
      0 \ar[r] &
      H_2(\widehat G) / K_{\kappa-1}(\widehat G) \ar[r] &
      H_2(\widehat G/\widehat G_{\kappa-1}) \ar[r] &
      \widehat G_{\kappa-1}/\widehat G_{\kappa} \ar[r] & 0
    \end{tikzcd}
  \]
  If (i)${}_\kappa$ holds, then (i)${}_{\kappa-1}$ holds too.  If
  (ii)${}_\kappa$ holds, then (ii)${}_{\kappa-1}$ holds too and consequently
  (i)${}_{\kappa-1}$ holds by the induction hypothesis.  So, in either case, we
  may assume that (i)${}_{\kappa-1}$ holds.  Then the middle vertical arrow of
  the diagram is an isomorphism.  By the snake lemma, it follows that
  \begin{equation}
    \label{equation:proof-stallings-dwyer-ker-coker}
    \Ker\{\widehat\pi_{\kappa-1}/\widehat\pi_\kappa \rightarrow
    \widehat G_{\kappa-1}/\widehat G_{\kappa}\}
    \cong \Coker\{H_2(\widehat\pi)/K_{\kappa-1}(\widehat\pi) \rightarrow
    H_2(\widehat G)/K_{\kappa-1}(\widehat G)\}.
  \end{equation}
  Since $\widehat\pi/\widehat\pi_{\kappa-1} \cong \widehat G/\widehat
  G_{\kappa-1}$ by~(i)${}_{\kappa-1}$, (i)${}_\kappa$ holds if and only if the
  left hand side of~\eqref{equation:proof-stallings-dwyer-ker-coker} is trivial.
  Also, (ii)${}_\kappa$ holds if and only if the right hand side
  of~\eqref{equation:proof-stallings-dwyer-ker-coker} is trivial. It follows
  that (i)${}_\kappa$ and (ii)${}_\kappa$ are equivalent.

  Now, suppose that $\kappa$ is a limit ordinal.  Suppose (i)${}_\kappa$ holds.
  For each $\lambda < \kappa$, since $\kappa$ is a limit ordinal,
  $\lambda+1<\kappa$.  So (i)${}_\kappa$ implies (i)${}_{\lambda+1}$. By the
  induction hypothesis, it follows that (ii)${}_{\lambda+1}$ holds.  In
  particular, $H_2(\widehat\pi) \to H_2(\widehat G) / K_{\lambda}(\widehat G)$
  is surjective.  This shows that (ii)${}_\kappa$ holds.

  For the converse, suppose (ii)${}_\kappa$ holds.  For each $\lambda<\kappa$,
  (ii)${}_\kappa$ implies (ii)${}_\lambda$, and thus (i)${}_\lambda$ holds by
  the induction hypothesis.  That is, $f$ induces an isomorphism
  $\widehat\pi/\widehat\pi_\lambda \isomto \widehat G/\widehat G_\lambda$.
  Therefore, if $g\in \Ker\{\widehat\pi/\widehat\pi_\kappa \to \widehat
  G/\widehat G_\kappa\}$, then $g \in \Ker\{\widehat\pi/\widehat\pi_\kappa \to
  \widehat\pi/\widehat\pi_\lambda\}$ for all $\lambda<\kappa$. Since
  $\widehat\pi_\kappa = \bigcap_{\lambda<\kappa} \widehat\pi_\lambda$, it
  follows that $g$ is trivial.  This proves that $\widehat\pi/\widehat\pi_\kappa
  \to \widehat G/\widehat G_\kappa$ is injective, and thus (i)${}_{\kappa}$
  holds.

  This completes the proof of Theorem~\ref{theorem:transfinite-stallings-dwyer}.
\end{proof}

We remark that the above proof of the equivalence of (i)${}_\kappa$ and
(ii)${}_\kappa$ indeed shows the following statement (just by replacing
$\widehat\pi$ and $\widehat G$ with $P$ and $Z$ below), which we record as a
lemma for later use in this paper.

\begin{lemma}
  \label{lemma:transfinite-stallings-dwyer-main-argument}
  Suppose $\kappa>1$ and $f\colon P\to Z$ is a group homomorphism which induces
  an epimorphism $P/P_\kappa \to Z/Z_\kappa$ and an isomorphism $H_1(P) \isomto
  H_1(Z)$. Then the following are equivalent:
  \begin{enumerate}
    \item[\textup{(i)}] $f$ induces an isomorphism $P/P_\kappa
    \to Z/Z_\kappa$.
    \item[\textup{(ii)}] $f$ induces an epimorphism $H_2(P) \to
    H_2(Z) / K_{\lambda}(Z)$ for all $\lambda<\kappa$,  where
    \[
      K_\lambda(Z):=\Ker\{H_2(Z)\rightarrow H_2(Z/Z_\lambda)\}.
    \]
  \end{enumerate}
\end{lemma}

\section{Realization of transfinite invariants}
\label{section:realization-transfinite-inv}

In this section, we prove Theorem~\ref{theorem:main-realization} stated in
Section~\ref{subsection:results-realization}, which characterizes the realizable
classes $\theta$ in~$H_3(\widehat\Gamma/\widehat\Gamma_\kappa)$.  In the proof
of Theorem~\ref{theorem:main-realization}, we will use the following lemmas. The
first lemma provides a finitely generated approximation of the transfinite lower
central quotients of the localization, along the lines of
Theorem~\ref{theorem:localization-basic-facts}\ref{item:localization-colimit}.

\begin{lemma}
  \label{lemma:approximation-local-lcsq}
  Suppose $G$ is a finitely presented group, $\kappa>1$ is an ordinal, and $H$
  is a finitely generated subgroup in~$\widehat G/\widehat G_\kappa$.  Then $H$
  is contained in a finitely generated subgroup $Q$ in $\widehat G/\widehat
  G_\kappa$ such that the inclusion induces an isomorphism $H_1(Q)\to
  H_1(\widehat G/\widehat G_\kappa)$.
\end{lemma}

\begin{proof}
  Since $H$ is finitely generated, there is a 2-connected homomorphism $P\to
  \widehat G$ of a finitely presented group $P$ such that the image of $P\to
  \widehat G \to \widehat G/\widehat G_\kappa$ contains $H$, by
  Theorem~\ref{theorem:localization-basic-facts}\ref{item:localization-colimit}.
  Let $Q$ be the image of $P\to \widehat G/\widehat G_\kappa$.  Since  $P\to Q$
  is surjective,  $H_1(P)\to H_1(Q)$ is surjective.  Since the composition $P\to
  Q \to \widehat G/\widehat G_\kappa$ induces an isomorphism on $H_1$,
  $H_1(P)\to H_1(Q)$ is injective, and consequently $H_1(P)\cong H_1(Q) \cong
  H_1(\widehat G/\widehat G_\kappa)$ under the induced homomorphisms.
\end{proof}

\begin{lemma}
  \label{lemma:H_1-surj-imply-local-lcsq-surj}
  Consider any ordinal~$\kappa$.
  Suppose $\pi$ is finitely generated, $G$ is finitely presented, and $f\colon
  \pi \to \widehat G/\widehat G_\kappa$ is a group homomorphism which induces an
  epimorphism $H_1(\pi)\to H_1(\widehat G/\widehat G_\kappa) = H_1(G)$.  Then
  $f$ induces an epimorphism $\widehat\pi\to \widehat G/\widehat G_\kappa$.
\end{lemma}

\begin{proof}
  Since $\widehat G/\widehat G_\kappa$ is trivial for $\kappa=1$, we may assume
  that $\kappa\ge 2$.  Recall from
  Lemma~\ref{lemma:lcsq-of-local-group-is-local} that the transfinite lower
  central quotient of a local group is local.  So, in our case, $\widehat
  G/\widehat G_\kappa$ is local, and thus there is an induced homomorphism
  $\widehat\pi\to \widehat G/\widehat G_\kappa$ by the universal property
  of~$\widehat\pi$.
  \[
    \begin{tikzcd}[sep=large]
      \pi \ar[r] \ar[d,"f"'] & \widehat\pi \ar[ld,dashed]
      \\
      \widehat G/\widehat G_\kappa
    \end{tikzcd}
  \]
  To show that $\widehat\pi \to \widehat G/\widehat G_\kappa$ is surjective, it
  suffices to prove that every finitely generated subgroup $H$ in $\widehat
  G/\widehat G_\kappa$ is contained in the image of $\widehat\pi \to \widehat
  G/\widehat G_\kappa$.  Since $H$ and $\pi$ are finitely generated, there is a
  finitely generated subgroup $Q$ in $\widehat G/\widehat G_\kappa$ such that
  the inclusion induces an isomorphism $H_1(Q) \isomto H_1(\widehat G/\widehat
  G_\kappa)$ and both $H$ and $f(\pi)$ are contained in~$Q$, by
  Lemma~\ref{lemma:approximation-local-lcsq}.  Consider the following
  commutative diagram.
  \[
    \begin{tikzcd}[column sep=large]
      \pi \ar[r] \ar[d] \ar[rr,bend left=40,"f"] &
      Q \ar[r,hook] \ar[d] &
      \widehat G/\widehat G_\kappa
      \\
      \widehat\pi \ar[r] & \widehat Q \ar[ru]
    \end{tikzcd}
  \]
  Since $H_1(\pi)\to H_1(\widehat G/\widehat G_\kappa)$ is surjective, it
  follows that $H_1(\pi)\to H_1(Q)$ is surjective.  Therefore, by
  Lemma~\ref{lemma:H1-surj-localization-surj}, $\widehat\pi \to
  \widehat Q$ is surjective.  Since the given subgroup $H\subset \widehat
  G/\widehat G_\kappa$ is contained in $Q$, it follows that $H$ is
  contained in the image of~$\widehat\pi$.  This completes the proof.
\end{proof}

Another key ingredient of our proof of Theorem~\ref{theorem:main-realization} is
the following ``homology surgery'' result for 3-manifolds over a \emph{finitely
generated} fundamental group, which is due to Turaev~\cite{Turaev:1984-1}.
Aforementioned in Section~\ref{section:main-results}, we denote the torsion
subgroup of $H_*(-)$ by~$tH_*(-)$.

\begin{lemma}[{Turaev~\cite[Lemma~2.2]{Turaev:1984-1}}]
  \label{lemma:turaev-homology-surgery}
  Suppose $g\colon N \to X$ is a map of a closed 3-manifold $N$ to a CW-complex
  $X$ with finitely generated~$\pi_1(X)$ such that the cap product
  \[
    \cap\, g_*[N]\colon tH^2(X)\to tH_1(X)  
  \]
  is an isomorphism.  Then $(N,g)$ is bordant, over $X$, to a pair $(M,f)$ of a
  closed 3-manifold $M$ and a map $f\colon M\to X$ which induces an isomorphism
  $f_*\colon H_1(M) \isomto H_1(X)$.
\end{lemma}

Now we are ready to start the proof of Theorem~\ref{theorem:main-realization}\@.
Recall from Section~\ref{subsection:results-determination} that the set
$\cR_\kappa(\Gamma)$ of realizable classes is defined to be the collection of
$\theta\in H_3(\widehat\Gamma/\widehat\Gamma_\kappa)$ such that
$\theta=\theta_\kappa(M)$ for some closed 3-manifold $M$ with $\pi=\pi_1(M)$
equipped with an isomorphism $\widehat\pi/\widehat\pi_\kappa \isomto
\widehat\Gamma/\widehat\Gamma_\kappa$.
Here, $\Gamma$ is a fixed finitely presented group.
Let $\kappa\ge 2$.
Theorem~\ref{theorem:main-realization} says that $\theta\in \cR_\kappa(\Gamma)$
if and only if the following two conditions hold.
\begin{enumerate}
  \item The cap product
  \[
    \cap\,\theta\colon tH^2(\widehat\Gamma/\widehat\Gamma_\kappa)
    \to tH_1(\widehat\Gamma/\widehat\Gamma_\kappa) \cong tH_1(\Gamma)
  \]
  is an isomorphism.
  \item The composition
  \[
    H^1(\widehat\Gamma/\widehat\Gamma_\kappa)
    \xrightarrow{\cap\,\theta} H_2(\widehat\Gamma/\widehat\Gamma_\kappa)
    \xrightarrow{\text{pr}}
    H_2(\widehat\Gamma/\widehat\Gamma_\kappa) /
    K_{\lambda}(\widehat\Gamma/\widehat\Gamma_\kappa)
  \]
  is surjective for all $\lambda<\kappa$, where
  $K_{\lambda}(\widehat\Gamma/\widehat\Gamma_\kappa)=\Ker\{H_2(\widehat\Gamma/\widehat\Gamma_\kappa)
  \rightarrow H_2(\widehat\Gamma/\widehat\Gamma_\lambda)\}$.
\end{enumerate}

\begin{proof}[Proof of Theorem~\ref{theorem:main-realization}\@]

  For the only if direction, suppose $\theta\in \cR_\kappa(\Gamma)$.  Choose a
  closed 3-manifold $M$ with $\pi=\pi_1(M)$ and an isomorphism $f\colon
  \widehat\pi/\widehat\pi_\kappa \isomto \widehat\Gamma/\widehat\Gamma_\kappa$
  such that $\theta_\kappa(M)=\theta$.  That is, $\theta=\phi_*[M]$ where
  $\phi_*\colon H_3(M)\to H_3(\widehat\Gamma/\widehat\Gamma_\kappa)$ is induced by
  the composition
  \[
    \phi\colon M\to B\pi \to B\widehat\pi \to B(\widehat\pi/\widehat\pi_\kappa)
    \xrightarrow[\simeq]{f} B(\widehat\Gamma/\widehat\Gamma_\kappa).
  \]
  Then, the following diagram is commutative.
  \[
    \begin{tikzcd}[sep=large]
      tH^2(M) \ar[d,"{\cap\,[M]}"'] &
      tH^2(\widehat\Gamma/\widehat\Gamma_\kappa)
        \ar[l,"\phi^*"']\ar[d,"\cap\,\theta"]
      \\
      tH_1(M) \ar[r,"\phi_*"'] &
      tH_1(\widehat\Gamma/\widehat\Gamma_\kappa)
    \end{tikzcd}
  \]
  The cap product $\cap\,[M]$ is an isomorphism by Poincar\'e duality. The
  bottom arrow $\phi_*$ is an isomorphism since $H_1(M) =
  H_1(\pi)=H_1(\widehat\pi/\widehat\pi_\kappa)$ and $f\colon
  \widehat\pi/\widehat\pi_\kappa \to \widehat\Gamma/\widehat\Gamma_\kappa$ is an
  isomorphism.  From this it also follows that the top arrow $\phi^*$ is an
  isomorphism, since $tH^2(-) = \Ext(H_1(-),\Z)$.  Therefore,
  $\cap\, \theta$ is an isomorphism.  This shows that (1) holds.

  To show that (2) holds, suppose $\lambda<\kappa$ and consider the following
  commutative diagram.
  \begin{equation}
    \label{equation:cap-product-by-fund-class-on-H^1}
    \begin{tikzcd}[row sep=large,column sep=small]
      H^1(M) \ar[d,"{\cap\,[M]}"'] & &
      H^1(\widehat\Gamma/\widehat\Gamma_\kappa)
        \ar[d,"\cap\,\theta"]  \ar[rrd,"\text{pr}\circ(\cap\,\theta)"]
        \ar[ll,"\phi^*"']
      \\
      H_2(M) \ar[rr,"\phi_*"'] \ar[d] & &
      H_2(\widehat\Gamma/\widehat\Gamma_\kappa) \ar[rr,"\text{pr}"'] & &
      H_2(\widehat\Gamma/\widehat\Gamma_\kappa) / 
        K_{\lambda}(\widehat\Gamma/\widehat\Gamma_\kappa)
      \\
      H_2(\pi) \ar[r] & 
      H_2(\widehat\pi) \ar[r] &
      H_2(\widehat\pi/\widehat\pi_\kappa) \ar[u] \ar[rru]
    \end{tikzcd}
  \end{equation}
  By Poincar\'e duality, $\cap\,[M]$ is an isomorphism.  Since
  $H^1(-)=\Hom(H_1(-),\Z)$ and $H_1(M)=H_1(\widehat\pi/\widehat\pi_\kappa) \to
  H_1(\widehat\Gamma/\widehat\Gamma_\kappa)$ is an isomorphism, the top arrow
  $\phi^*$ is an isomorphism. Also, the assumption that $f\colon
  \widehat\pi/\widehat\pi_\kappa \to \widehat\Gamma/\widehat\Gamma_\kappa$ is an
  isomorphism implies that the composition $H_2(\widehat\pi) \to
  H_2(\widehat\pi/\widehat\pi_\kappa) \to
  H_2(\widehat\Gamma/\widehat\Gamma_\kappa) /
  K_\lambda(\widehat\Gamma/\widehat\Gamma_\kappa)$ is surjective, by applying
  Lemma~\ref{lemma:transfinite-stallings-dwyer-main-argument} to the composition
  $\widehat\pi \to \widehat\pi/\widehat\pi_\kappa \isomto
  \widehat\Gamma/\widehat\Gamma_\kappa$.  Since $H_2(M)\to H_2(\pi)$ and
  $H_2(\pi) \to H_2(\widehat\pi)$ are surjective (see
  Theorem~\ref{theorem:localization-basic-facts}\ref{item:localization-colimit}
  for the latter), it follows that the composition
  $\text{pr}\circ(\cap\,\theta)$
  in~\eqref{equation:cap-product-by-fund-class-on-H^1} is surjective. This
  proves that (2) holds.

  It remains to show the if direction.  Suppose (1) and (2) hold for a given
  class $\theta\in H_3(\widehat\Gamma/\widehat\Gamma_\kappa)$.  Since
  $H_3=\Omega^{SO}_3$, there is a map $\psi\colon N\to
  B(\widehat\Gamma/\widehat\Gamma_\kappa)$ of a closed 3-manifold $N$ such that
  $\psi_*[N] = \theta$.
  
  We will invoke Turaev's homology surgery for 3-manifolds
  (Lemma~\ref{lemma:turaev-homology-surgery}) to alter~$(N,\psi)$.
  Note that $\widehat\Gamma/\widehat\Gamma_\kappa$ is not finitely generated in
  general, and thus Lemma~\ref{lemma:turaev-homology-surgery} does
  not apply directly over $B(\widehat\Gamma/\widehat\Gamma_\kappa)$.  So we
  proceed as follows, using a finitely generated approximation.  Apply
  Lemma~\ref{lemma:approximation-local-lcsq} to choose a finitely
  generated subgroup $Q$ in $\widehat\Gamma/\widehat\Gamma_\kappa$ such that the
  inclusion induces an isomorphism $H_1(Q) \isomto
  H_1(\widehat\Gamma/\widehat\Gamma_\kappa)$ and $\pi_1(N)\to
  \widehat\Gamma/\widehat\Gamma_\kappa$ factors through~$Q$.  Let $\psi'\colon N
  \to B(\pi_1(N)) \to B(Q)$ be the composition, and consider the following
  commutative diagram.
  \[
    \begin{tikzcd}[sep=large]
      tH^2(Q) \ar[d,"{\cap\, \psi'_*[N]}"'] \ar[r,"\cong"] &
      tH^2(\widehat\Gamma/\widehat\Gamma_\kappa)
        \ar[d,"\cong"',"{\cap\, \psi_*[N]} = {\cap\, \theta}"] &
      \\
      tH_1(Q) & 
      tH_1(\widehat\Gamma/\widehat\Gamma_\kappa) \ar[l,"\cong"]
    \end{tikzcd}
  \]
  The two horizontal arrows and the right vertical arrow $\cap\,\theta$ are
  isomorphisms, by our choice of $Q$, by the fact $tH^2(-) =
  \Ext(H_1(-),\Z)$ and by the hypothesis~(1).  So $\cap\,\psi'_*[N]$ is an
  isomorphism too.  Now apply Lemma~\ref{lemma:turaev-homology-surgery}
  to $(N,\psi')$ to produce a closed 3-manifold $M$ endowed with a map $M \to
  B(Q)$ which induces an isomorphism on~$H_1$.  Let $\phi\colon M \to B(Q) \to
  B(\widehat\Gamma/\widehat\Gamma_\kappa)$ be the composition.  It induces an
  isomorphism $H_1(M) \isomto H_1(Q) \cong
  H_1(\widehat\Gamma/\widehat\Gamma_\kappa)$.  Also, since $(M,\phi)$ is bordant
  to $(N,\psi)$, we have $\phi_*[M] = \psi_*[N] = \theta$.

  Let $\pi=\pi_1(M)$, and consider $\pi\to \widehat\Gamma/\widehat\Gamma_\kappa$
  induced by~$\phi$.  It gives rise to a homomorphism $\widehat\pi \to
  \widehat\Gamma/\widehat\Gamma_\kappa$ since
  $\widehat\Gamma/\widehat\Gamma_\kappa$ is local by
  Lemma~\ref{lemma:lcsq-of-local-group-is-local}.  Consider the
  diagram~\eqref{equation:cap-product-by-fund-class-on-H^1} again.  Now, we have
  that the composition $\text{pr}\circ(\cap\,\theta)$ is surjective by the
  hypothesis~(2).  Note that this surjection is equal to the composition of the
  six arrows along the counterclockwise outmost path from
  $H^1(\widehat\Gamma/\widehat\Gamma_\kappa)$ to
  $H_2(\widehat\Gamma/\widehat\Gamma_\kappa) /
  K_\lambda(\widehat\Gamma/\widehat\Gamma_\kappa)$
  in~\eqref{equation:cap-product-by-fund-class-on-H^1}.  So, the map
  $H_2(\widehat\pi) \to H_2(\widehat\Gamma/\widehat\Gamma_\kappa) /
  K_\lambda(\widehat\Gamma/\widehat\Gamma_\kappa)$, which is the last one
  applied in the composition, is surjective.  By applying
  Lemma~\ref{lemma:transfinite-stallings-dwyer-main-argument} to $\widehat\pi\to
  \widehat\Gamma/\widehat\Gamma_\kappa$, it follows that $\phi$ induces an
  isomorphism $\widehat\pi/\widehat\pi_\kappa \isomto
  \widehat\Gamma/\widehat\Gamma_\kappa$.  Therefore $\theta=\phi_*[M]$ lies in
  $\cR_\kappa(\Gamma)$.  This completes the proof of
  Theorem~\ref{theorem:main-realization}\@.
\end{proof}

\section{Universal \texorpdfstring{$\theta$}{theta}-invariant}
\label{section:final-transfinite-invariant}

We begin by recalling the definition of the universal $\theta$-invariant from
Definition~\ref{definition:theta-final}. As before, let $\Gamma$ be a
finitely presented group.  Suppose $M$ is a closed 3-manifold with
$\pi=\pi_1(M)$ equipped with an isomorphism $f\colon \widehat\pi \to
\widehat\Gamma$.  Motivated from Levine's link invariant
in~\cite{Levine:1989-1}, define $\widehat\theta(M)\in H_3(\widehat\Gamma)$ to be
the image of $[M]\in H_3(M)$ under
\[
  H_3(M)\to H_3(\pi) \to H_3(\widehat\pi)
  \xrightarrow[\cong]{f_*} H_3(\widehat\Gamma).
\]
The value of $\widehat\theta(M)$ depends on the choice of $f$, while its image
in $H_3(\widehat\Gamma)/\Aut(\widehat\Gamma)$ is independent of the choice
of~$f$.

The following is analogous to Theorem~\ref{theorem:results-homology-cob}.  We
omit the proof, since the argument is exactly the same as that of
Theorem~\ref{theorem:results-homology-cob}.

\begin{theorem}
  \label{theorem:homology-cob-final-inv}
  The invariant $\widehat\theta(M)$ is invariant under homology cobordism in the
  following sense: 
  \begin{enumerate}
    \item If $M$ and $N$ are homology cobordant 3-manifolds with $\pi=\pi_1(M)$
    and $G=\pi_1(N)$, then there is an isomorphism $\phi\colon \widehat G
    \isomto \widehat\pi$, and consequently $\widehat\theta(M)$ is defined if and
    only if $\widehat\theta(N)$ is defined.
    \item When $\widehat\theta(M)$ and $\widehat\theta(N)$ are defined using an
    isomorphism $f\colon \widehat{\pi} \isomto
    \widehat\Gamma$ and the composition $f\circ\phi$, we
    have $\widehat\theta(M) = \widehat\theta(N)$ in
    $H_3(\widehat\Gamma)$.
    \item When $\widehat\theta(M)$ and $\widehat\theta(N)$ are defined using
    arbitrary isomorphisms $\widehat\pi \isomto \widehat\Gamma$ and $\widehat G
    \isomto \widehat\Gamma$, we have $\widehat\theta(M)=\widehat\theta(N)$ in
    $H_3(\widehat\Gamma) / \Aut(\widehat\Gamma)$.
  \end{enumerate}
\end{theorem}

Let $\widehat\cR(\Gamma)$ be the collection of classes $\theta\in
H_3(\widehat\Gamma)$ such that there exists a closed 3-manifold $M$ with $\pi =
\pi_1(M)$ endowed with an isomorphism $\widehat\pi \isomto \widehat\Gamma$ for
which $\widehat\theta(M) =\theta$.  We will give a proof of
Theorem~\ref{theorem:main-realization-final-inv} stated in
Section~\ref{subsection:results-universal-inv}.  For the reader's convenience, we
recall the statement:  a homology class $\theta\in H_3(\widehat\Gamma)$ lies in
$\widehat\cR(\Gamma)$ if and only if the following two conditions hold.
\begin{enumerate}
  \item The cap product\/ $\cap\,\theta\colon tH^2(\widehat\Gamma) \to
  tH_1(\widehat\Gamma) \cong tH_1(\Gamma)$ is an isomorphism.
  \item The cap product\/ $\cap\,\theta\colon H^1(\widehat\Gamma) \to
  H_2(\widehat\Gamma)$ is surjective.
\end{enumerate}

\begin{proof}[Proof of Theorem~\ref{theorem:main-realization-final-inv}]

  We will first prove the only if part, using an argument almost identical to
  the proof of Theorem~\ref{theorem:main-realization}\@.  Suppose $M$ is a
  closed 3-manifold with $\pi=\pi_1(M)$ and $f\colon \widehat\pi \isomto
  \widehat\Gamma$ is an isomorphism.  Let $\theta=\widehat\theta(M)\in
  H_3(\widehat\Gamma)$.  That is, $\theta$ is the image of $[M]$ under the map
  induced by the composition $\phi\colon M \to B\pi \to B\widehat\pi
  \xrightarrow{f} B\widehat\Gamma$.  Consider the following commutative diagram:
  \[
    \begin{tikzcd}[sep=large]
      tH^2(M) \ar[d,"{\cap\,[M]}"'] &
      tH^2(\widehat\Gamma)
        \ar[l,"\phi^*"']\ar[d,"\cap\,\theta"]
      \\
      tH_1(M) \ar[r,"\phi_*"'] &
      tH_1(\widehat\Gamma)
    \end{tikzcd}
  \]
  By Poincar\'e duality, $\cap\,[M]$ is an isomorphism. The arrow $\phi_*$ is an
  isomorphism since $f$ is an isomorphism.  Using $tH^2(-) =
  \Ext(H_1(-),\Z)$, it follows that $\phi^*$ is an isomorphism. So, by the
  commutativity, $\cap\, \theta$ is an isomorphism. This shows that (1) holds.
  To show that (2) holds, consider the following commutative diagram:
  \[
    \begin{tikzcd}[sep=large]
      H^1(M) \ar[d,"{\cap\,[M]}"'] &
      H^1(\widehat\Gamma) \ar[d,"\cap\,\theta"] \ar[l,"\phi^*"']
      \\
      H_2(M) \ar[r,"\phi_*"'] &
      H_2(\widehat\Gamma)
    \end{tikzcd}
  \]
  The arrows $\phi^*$ is an isomorphism since $f$ is an isomorphism, and
  $\cap\,[M]$ is an isomorphisms by Poincar\'e duality.  Since $H_2(M) \to
  H_2(\pi)$ and $H_2(\pi) \to H_2(\widehat\pi)$ are surjective, $\phi_*$ is
  surjective.  So $\cap\,\theta$ is surjective.  That is, (2) holds.

  Now, we will prove the if part.  Our argument will be different from the proof
  of Theorem~\ref{theorem:main-realization}\@.  Suppose $\theta\in
  H_3(\widehat\Gamma)$ is a homology class satisfying the conditions~(1)
  and~(2).  Choose a sequence of 2-connected homomorphisms of finitely presented
  groups
  \[
    \Gamma=P(1) \to P(2) \to \cdots \to P(\ell) \to \cdots
  \]
  such that $\widehat\Gamma=\colim_\ell P(\ell)$, by using
  Theorem~\ref{theorem:localization-basic-facts}\ref{item:localization-colimit}.
  Since $H_3(\widehat\Gamma)$ is the colimit of $H_3(P(\ell))$, the class
  $\theta$ lies in the image of $H_3(P(\ell_0))$ for some~$\ell_0$.  Let
  $P=P(\ell_0)$ for brevity.  Denote $P \to \widehat\Gamma$ by $\iota$, and write
  $\theta=\iota_*(\sigma)$, where $\sigma\in H_3(P)$.

  We claim that we may assume that $\iota_*\colon H_2(P) \to
  H_2(\widehat\Gamma)$ is an isomorphism.  To prove this, first recall that
  $H_2(P) \to H_2(\widehat\Gamma)$ is surjective by the choice of the
  sequence~$\{P(\ell)\}$.  Let $N$ be the kernel of $H_2(P) \to
  H_2(\widehat\Gamma)$.  Since $P$ is finitely presented, $H_2(P)$ is a finitely
  generated abelian group, and thus $N$ is finitely generated.  Since
  $H_2(\widehat\Gamma)$ is the colimit of $H_2(P(\ell))$, it follows that the
  image of $N$ under $H_2(P) \to H_2(P(\ell_1))$ is trivial for some
  $\ell_1\ge\ell_0$.  Since $H_2(P) \to H_2(P(\ell_1))$ is surjective, we have
  $H_2(P(\ell_1)) \cong H_2(P)/N \cong H_2(\widehat\Gamma)$.  Replacing $P$ by
  $P(\ell_1)$, the claim is obtained.

  We will use Turaev's homology surgery, over the finitely presented group~$P$.
  Choose a map $\psi\colon N \to BP$ of a closed 3-manifold $N$ such that
  $\psi_*[N] = \sigma$, using that $\Omega^{SO}_3(P) = H_3(P)$.  Consider the
  following commutative diagram:
  \[
    \begin{tikzcd}[sep=large]
      tH^2(P) \ar[d,"{\cap\, \sigma}"']  &
      tH^2(\widehat\Gamma) \ar[l,"\iota^*"'] \ar[d,"{\cap\, \theta}"]
      \\
      tH_1(P) \ar[r,"\iota_*"'] & 
      tH_1(\widehat\Gamma) 
    \end{tikzcd}
  \]
  By condition (1), $\cap\,\theta$ is an isomorphism.  The arrows $\iota_*$ and
  $\iota^*$ are isomorphisms since $H_1(P)\to H_1(\widehat\Gamma)$ is an
  isomorphism by the choice of~$\{P(\ell)\}$ and $tH^2(-) = \Ext(H_1(-),\Z)$.
  So, $\cap\,\sigma$ is an isomorphism.  Apply Turaev's
  Lemma~\ref{lemma:turaev-homology-surgery}, to obtain a map
  $\phi\colon M \to B(P)$ of a closed 3-manifold $M$ with $\pi=\pi_1(M)$ such that
  $(M,\phi)$ is bordant to $(N,\psi)$ over $P$ and $\phi_*\colon H_1(M) \to
  H_1(P)$ is an isomorphism.  We have $\phi_*[M] = \psi_*[N] = \sigma$
  in~$H_3(P)$. Consider the following diagram:
  \[
    \begin{tikzcd}[sep=large]
      H^1(M) \ar[d,"{\cap\,[M]}"'] &
      H^1(P) \ar[d,"{\cap\,\sigma}"] \ar[l,"\phi^*"'] &
      H^1(\widehat\Gamma) \ar[d,"{\cap\,\theta}"] \ar[l,"\iota^*"']
      \\
      H_2(M) \ar[r,"\phi_*"'] &
      H_2(P) \ar[r,"\iota_*"'] &
      H_2(\widehat\Gamma)
    \end{tikzcd}
  \]
  By condition (2), $\cap\,\theta$ is surjective.  The arrows $\iota^*$ and
  $\iota_*$ are isomorphisms by the choice of $\{P(\ell)\}$ and by the claim.
  The arrow $\phi^*$ is an isomorphism since $\phi$ induces an isomorphism
  on~$H_1$.  By Poincar\'e duality, $\cap\,[M]$ is an isomorphism. From these
  facts, it follows that $\phi_*\colon H_2(M) \to H_2(P)$ is surjective. So, by
  Theorem~\ref{theorem:localization-basic-facts}\ref{item:localization-2-connected-map},
  $\phi_*\colon \pi=\pi_1(M) \to P$ induces an isomorphism $\widehat\pi \isomto
  \widehat P$.  Since $\iota$ induces $\widehat P \isomto \widehat \Gamma$, it
  follows that $\iota\phi\colon M \to \widehat\Gamma$ induces an isomorphism
  $\widehat\pi \isomto \widehat\Gamma$.  Since $\phi_*[M] = \sigma$, we have
  $\widehat\theta(M) = \iota_* \phi_*[M] = \iota_*\sigma = \theta$.  This shows
  that $\theta\in \widehat\cR(\Gamma)$.
\end{proof}

\section{The free group case and Milnor's link invariant}
  \label{section:classical-milnor-invariant}

In this section we discuss the case when $\Gamma$ is a free group, and show that
our invariants of finite length applied to the zero framed surgery manifold of a
link are equivalent to Milnor's link invariants and Orr's homotopy theoretic
reformulation of the Milnor invariant.  Most of the results from this section
appear in~\cite{Orr:1989-1, Igusa-Orr:2001-1, Levine:1989-1, Levine:1989-2}.
However, relating prior work to the results herein seems non-trivial.  This
section will highlight and clarify new perspectives on Milnor's link invariants.

We proceed as follows.  Fix a positive integer $m$, and as the ``basepoint''
manifold, let $Y$ be the connected sum of $m$ copies of $S^1\times S^2$.  Then
$\pi_1(Y)=F$, the free group on $m$ generators. In this case, we have the
following useful property. 

\begin{lemma}
  \label{lemma:realization-over-free-group}
  For finite $k\ge 2$,  $\cR_k(F) = H_3(F/F_k)$.
\end{lemma}
\begin{proof}
  Recall that $H_2(F/F_k)=F_k/F_{k+1}$ by Hopf's theorem.  Thus the projection
  induces a zero homomorphism $H_2(F/F_{k}) \to H_2(F/F_{k-1})$.  From this
  and the fact that $H_1(F)=\Z^m$ is torsion free, it follows that $\cR_k(F) =
  H_3(F/F_k)$, by Theorem~\ref{theorem:main-realization}.
\end{proof}

So, $\cR_k(F)$ is an abelian group, and consequently $\Coker\{\cR_{k+1}(F) \to
\cR_k(F)\}$ is an abelian group too.  We remark that the structure of
this cokernel was computed in~\cite{Orr:1989-1,Igusa-Orr:2001-1}.  The cokernel, 
$\Coker\{\cR_{k+1}(F) \to \cR_k(F)\} = \Coker\{H_3(F/F_{k+1}) \to H_3(F/F_k)\}$
is a free abelian group of rank $m\cR(m,k)-\cR(m,k+1)$ where
\[
  \cR(m,n) := \frac1n \sum_{d|n}\phi(d)\cdot m^{n/d}
\]
and $\phi(d)$ is the M\"obius function.

The following is another useful feature of the case of the free group~$F$.

\begin{lemma}
  \label{lemma:free-nilpotent-quotient-lift}
  Suppose $\pi$ is a group.  Then, for finite $k\ge 2$, every isomorphism $f\colon
  \pi/\pi_{k} \isomto F/F_{k}$ lifts to an isomorphism $\pi/\pi_{k+1} \isomto
  F/F_{k+1}$ if and only if there exists an isomorphism $\pi/\pi_{k+1} \isomto
  F/F_{k+1}$.
\end{lemma}

\begin{proof}
  The only if part is trivial.  For the if part, observe that a homomorphism
  $F\to F$ induces an isomorphism $F/F_k \to F/F_k$ if and only if it induces
  an isomorphism $H_1(F)\to H_1(F)$, by Stallings'
  Theorem~\cite{Stallings:1965-1}, since $H_2(F)=0$.  It follows that every
  automorphism of $F/F_k$ lifts to an automorphism of $F/F_{k+1}$ for $k\ge
  2$.  The conclusion is a straightforward consequence of this: if $\tilde
  g\colon \pi/\pi_{k+1} \isomto F/F_{k+1}$ is an isomorphism, then choose an
  automorphism lift $\tilde\phi\colon F/F_{k+1} \to F/F_{k+1}$ of the
  automorphism $\phi=fg^{-1}$, where  $g\colon \pi/\pi_{k} \isomto F/F_{k}$ is
  the induced isomorphism.  Then the composition $\tilde\phi \circ \tilde g$
  is an isomorphism which is a lift of $fg^{-1}\circ g = f$.
\end{proof}

Using the results stated in Section~\ref{section:main-results} and the above
lemmas on the free group, we compare the lower central quotients $\pi/\pi_k$ of
a 3-manifold group $\pi=\pi_1(M)$ with the free nilpotent quotient~$F/F_k$.  For
the initial case $k=2$, $\pi/\pi_k$ is isomorphic to $F/F_k$ if and only if
$H_1(\pi)\cong \Z^m$.  The following theorem deals with the induction step.

\begin{theorem}
  \label{theorem:finite-inv-over-free-group}
  Suppose $M$ is a closed 3-manifold with $\pi=\pi_1(M)$, equpped with an
  isomorphism $f\colon \pi/\pi_k \isomto F/F_k$, $k\ge 2$.  Then the following
  are equivalent.
  \begin{enumerate}
    \item The given $f$ lifts to an isomorphism $\pi/\pi_{k+1} \isomto
    F/F_{k+1}$.
    \item There is an isomorphism $\pi/\pi_{k+1} \cong F/F_{k+1}$ (which is not
    necessarily a lift).
    \item The invariant $\theta_k(M,f)$ vanishes in
    $\Coker\{\cR_{k+1}(F) \to \cR_k(F)\}$.
    \item The invariant $\theta_k(M)$ vanishes in $\Coker\{\cR_{k+1}(F) \to
    \cR_k(F)/\Aut(F/F_k)\}$.
    \item The invariant $\theta_k(M,g)$ vanishes in $\Coker\{\cR_{k+1}(F) \to \cR_k(F)\}$ for any isomorphism $g\colon\pi/\pi_k \isomto F/F_k$.
  \end{enumerate}
\end{theorem}

\begin{proof}
  (1)~and (2) are equivalent by Lemma~\ref{lemma:free-nilpotent-quotient-lift}.
  (1)~and (3) are equivalent by
  Theorem~\ref{theorem:main-determination-lcsq-lift}\@. (2)~and (4) are
  equivalent by Theorem~\ref{theorem:main-determination-lcsq-nonlift}\@. It
  follows that (2) implies (3) for any isomorphism~$f$.  In other words, (2)
  implies (5).  Finally, (5) implies (3) obviously.
\end{proof}

Now, we apply the above to links.  For an $m$-component link $L$ in $S^3$, let
$M_L$ be the zero framed surgery manifold of~$L$.  Note that if $L$ is the
trivial link, then $M_L$ is equal to the 3-manifold~$Y$ that we use in this
section.

In~\cite{Milnor:1957-1}, Milnor defined his concordance invariants, which we now call \emph{Milnor's numerical invariants}.
These invariants arise as coefficients of the Magnus expansion evaluated on homotopy classes of longitudes of a link.
More precisely, for an $m$-component link $L$,
the Magnus expansion is defined by sending the $\iota$th meridian to $1+t_\iota$ and extending it multiplicatively, and for a sequence $\iota_1,\ldots,\iota_k$ of integers $\iota_j\in \{1,\ldots,m\}$, Milnor's numerical invariant of length $k$, $\bar{\mu}_L(\iota_1,\ldots, \iota_k)$, is the coefficient of $t_{\iota_1}\cdots t_{\iota_{k-1}}$ in the Magnus expansion of the $\iota_k$th longitude of~$L$.
Milnor's numerical invariants of length $k$ are well-defined as integers for $L$ if all Milnor's numerical invariants of length${}<k$ vanish for~$L$.
One can find details in~\cite{Milnor:1957-1}. 

\begin{theorem}
  \label{theorem:milnor-and-theta_k-for-links}
  Suppose $L$ is a link with $m$ components.  For any finite $k \ge 2$, the following are equivalent.
  \begin{enumerate}
    \item[\textup{(1)$_k$}] There is an isomorphism $\pi_1(S^3\sm
    L)/\pi_1(S^3\sm L)_{k+1} \cong F/F_{k+1}$.
    \item[\textup{(2)$_k$}] The zero linking longitudes of $L$ lie in
    $\pi_1(S^3\sm L)_{k}$.
    \item[\textup{(3)$_k$}] There is an isomorphism $\pi_1(M_L)/\pi_1(M_L)_k
    \cong F/F_k$.
    \item[\textup{(4)$_k$}] Milnor's numerical invariants of length $k+1$ are well-defined for $L$ as integers.
  \end{enumerate}
  If the above \textup{(1)$_k$}--\textup{(4)$_k$} hold, then
  \textup{(1)$_{k+1}$}--\textup{(4)$_{k+1}$} and the following
  \textup{(5)}$_{k+1}$--\textup{(8)}$_{k+1}$ are equivalent.
  \begin{enumerate}
    \item[\textup{(5)$_{k+1}$}] Milnor's numerical invariants of length $k+1$ vanish for~$L$.
    \item[\textup{(6)$_{k+1}$}] For some $f\colon \pi_1(M_L)/\pi_1(M_L)_k
    \isomto F/F_k$, $\theta_k(M_L,f)$ vanishes in $\Coker\{\cR_{k+1}(F) \to \cR_k(F)\}$.
    \item[\textup{(7)$_{k+1}$}] For all $f\colon \pi_1(M_L)/\pi_1(M_L)_k \isomto
    F/F_k$, $\theta_k(M_L,f)$ vanishes in $\Coker\{\cR_{k+1}(F) \to \cR_k(F)\}$.
    \item[\textup{(8)$_{k+1}$}] The invariant $\theta_k(M_L)$ vanishes in $\Coker\{\cR_{k+1}(F) \to \cR_k(F)/\Aut(F/F_k)\}$.
  \end{enumerate}
\end{theorem}

From Theorem~\ref{theorem:milnor-and-theta_k-for-links}, it follows that all
Milnor invariants of length $k+1$ are defined without ambiguity if and only if
$\theta_k(M_L)$ is defined, and all Milnor invariants of length $k+1$ vanish if
and only if $\theta_k(M_L)$ vanishes in $\Coker\{\cR_{k+1}(F) \to \cR_k(F)\}$.

\begin{proof}
  
  The equivalence of (1)$_k$--(4)$_k$ is a folklore consequence of
  Milnor's theorem~\cite{Milnor:1957-1}:
  \[
    \pi_1(S^3\sm L)/\pi_1(S^3\sm L)_{k+1} \cong
    \langle F \mid F_{k+1},\, [w_1,x_1], \ldots,  [w_m,x_m]\rangle
  \]
  where $x_i$ and $w_i$ correspond to a meridian and zero linking longitude of the $i$th component of~$L$, respectively.
  Indeed, since $F/F_{k+1}$ is Hopfian, the right hand side, which is a
  quotient of $F/F_{k+1}$, is isomorphic to $F/F_{k+1}$ if and only if
  $[w_i,x_i]\in F_{k+1}$ for all~$i$. A standard application of the Magnus
  expansion, or Hall basis theorem, shows that $[w_i,x_i]\in F_{k+1}$ if and
  only if $w_i \in F_k$.  Also, since $\pi_1(M_L)$ is the quotient of
  $\pi_1(S^3\sm L)$ by the normal subgroup generated by the longitudes, we have
  \[
    \pi_1(M_L)/\pi_1(M_L)_k \cong \langle F \mid F_{k},\, w_1,\ldots,w_m\rangle
  \]
  by Milnor's theorem.  Thus $\pi_1(M_L)/\pi_1(M_L)_k \cong F/F_{k}$ if and only
  if $w_i\in F_k$, and it is the case if and only if $\pi_1(S^3\sm
  L)/\pi_1(S^3\sm L)_{k+1} \cong F/F_{k+1}$ by the above. Also, it is known that
  Milnor's invariants of length $k+1$ are well-defined integers without
  ambiguity if and only if $w_i \in F_{k}$~\cite{Milnor:1957-1}.  This shows
  that (1)$_k$--(4)$_k$ are equivalent.  Milnor also showed that his invariants
  of length $k+1$ vanish if and only if $w_i\in F_{k+1}$~\cite{Milnor:1957-1}.
  It follows that (5)$_{k+1}$ is equivalent to (1)$_{k+1}$--(4)$_{k+1}$.
  
  By Theorem~\ref{theorem:finite-inv-over-free-group}, each of
  (6)$_{k+1}$--(8)$_{k+1}$ is equivalent to (3)$_{k+1}$.  This completes the
  proof.
\end{proof}

In what follows we discuss the relationship of our invariants and the link invariant defined in~\cite{Orr:1989-1}.

Let $L$ be a link for which Milnor's invariants of length${}\le k$ vanish. Let
$E_L$ be the exterior of $L$, and $G=\pi_1(E_L) = \pi_1(S^3\sm L)$.  Let $K_k$
be the mapping cone of the inclusion $\bigvee^m S^1 = B(F) \to B(F/F_k)$, and
let $j\colon B(F/F_k) \to K_k$ be the inclusion.  By Milnor's
result~\cite{Milnor:1957-1} (or by
Theorem~\ref{theorem:milnor-and-theta_k-for-links}), there is an isomorphism
$F/F_\ell \isomto G/G_\ell$ which takes generators of $F$ to meridians, for
$\ell\le k+1$.  When $\ell=k$, this gives rise to a map
\[
  E_L \to B(G) \to B(G/G_k) \xrightarrow{\simeq} B(F/F_k) \xrightarrow{j} K_k
\]
which sends meridians to null-homotopic loops.  So this extends to a map
$\psi\colon S^3\to K_k$.  Denote the homotopy class of this extension
by~$\theta_k(L) = [\psi] \in \pi_3(K_k)$.  This is the invariant defined and
studied in~\cite{Orr:1989-1}.

Recall from the proof of Theorem~\ref{theorem:milnor-and-theta_k-for-links} that $G/G_k\isomto F/F_k$ induces $f\colon \pi_1(M_L)/\pi_1(M_L)_k\isomto F/F_k$.
Consider $\theta_k(M_L)=\theta_k(M_L,f)$.

To compare $\theta_k(L)$ with $\theta_k(M_L)$, we will use arguments which are
already known to experts.

Let $h\colon \pi_3(K_k) \to H_3(K_k)$ be the Hurewicz homomorphism.  Note that
the inclusion $j$ induces an isomorphism $j_*\colon \cR_k(F)=H_3(F/F_k) \to
H_3(K_k)$, since $K_k$ is obtained from $B(F/F_k)$ by attaching $2$-cells. We
claim that our $\theta_k(M_L)$ and Orr's $\theta_k(L)$ are identical in
$H_3(K_k)$. That is, $\theta_k(M_L) = j_*^{-1} h(\theta_k(L))$.

The claim is verified as follows.
Attach $m$ 2-handles to $S^3\times[0,1]$ along the zero-framing of the link $L\subset S^3=S^3\times1$, to obtain a 4-dimensional cobordism $W$ between $S^3$ and~$M_L$.
Let $\phi\colon M_L \to B(F/F_k)$ be the map induced by the above $f\colon \pi_1(M_L)/\pi_1(M_L)_k\isomto F/F_k$.
Since $\phi$ restricts to $\psi\colon E_L\to B(F/F_k)$ and since $W$ is obtained by attaching $m$ dual 2-handles to $M_L\times [0,1]$ along meridians, $M_L \xrightarrow{\phi} B(F/F_k) \xrightarrow{j} K_k$ extends to a map $W\to K_k$ which restricts to $\psi\colon S^3\to K_k$. 
This gives us the following commutative diagram.
\[
  \begin{tikzcd}[row sep=large]
    M_L \ar[r,hook] \ar[d,"\phi"']
    & W \ar[rd]
    & S^3 \ar[l,hook'] \ar[d,"\psi"]
    \\
    |[overlay]| B(F/F_k) \ar[rr,hook,"j"'] & & K_k
  \end{tikzcd}
\]
From the diagram, the assertion $j_*(\theta_k(M_L)) =
h(\theta_k(L))$ follows.

In addition, $\theta_k(M_L)=0$ in the cokernel of $H_3(F/F_k) \to H_3(F/F_k)$ if
and only if $\theta_k(L)=0$ in the cokernel of $\pi_3(K_{k+1}) \to \pi_3(K_k)$.
It follows immediately from the above and from the known fact that the
composition $j_*^{-1}h\colon \pi_3(K_k) \to H_3(F/F_k)$ induces an isomorphism
between the cokernels~\cite{Orr:1989-1,Igusa-Orr:2001-1}.

Consequently, the equivalence of (5), (6) and (7) in
Theorem~\ref{theorem:milnor-and-theta_k-for-links}
subsumes the following result of Orr~\cite{Orr:1989-1}: for a link $L$, the
Milnor invariants of length $k+1$ vanish if and only if $\theta_k(L)=0$ in
$\Coker\{\pi_3(K_{k+1}) \to \pi_3(K_k)\}$.

We remark that the same argument shows that Levine's link invariant $\theta(L)
\in H_3(\widehat F)$  defined in~\cite{Levine:1989-1} can be identified with our
final invariant~$\widehat\theta(M_L)$ of the zero-framed surgery manifold~$M_L$.

\begin{remark}
  Results of this section for general closed 3-manifolds and zero surgery manifolds of links holds for transfinite ordinals $k$, if one uses $\widehat G/\widehat G_k$ instead of $G/G_k$ for $G=F$ and $G=\pi_1(M_L)$, as we always do in this paper.
  More precise, we have the following.

  \begin{enumerate}
    \item Lemma~\ref{lemma:free-nilpotent-quotient-lift} is true for transfinite~$k$.
    To prove this, one uses our Theorem~\ref{theorem:transfinite-stallings-dwyer} instead of Stallings' Theorem in the above proof of Lemma~\ref{lemma:free-nilpotent-quotient-lift} (and use that $H_2(\widehat F)=0$.)
    \item Theorem~\ref{theorem:finite-inv-over-free-group} is true for transfinite~$k$.
    To prove this, one one uses the transfinite version of Lemma~\ref{lemma:free-nilpotent-quotient-lift} instead of Lemma~\ref{lemma:free-nilpotent-quotient-lift} in the above proof of Theorem~\ref{theorem:finite-inv-over-free-group}.
    \item Theorem~\ref{theorem:milnor-and-theta_k-for-links} is true for transfinite $k$, if one removes the conditions~(1)$_k$, (2)$_k$, (4)$_k$ and~(5)$_{k+1}$.
    The proof is the same as the finite case given above.
  \end{enumerate}

  On the other hand, for links, we do not know whether the transfinite case of the full version of Theorem~\ref{theorem:milnor-and-theta_k-for-links} is true.  
  In particular, the following question seems interesting: are our invariants of the zero surgery manifold $M_L$ determined by the homotopy class of the longitudes of $L$, relative to the transfinite lower central series of the group localization? (See the condition (2)$_k$ in Theorem~\ref{theorem:milnor-and-theta_k-for-links}.) 

  We also note that in~\cite{Igusa-Orr:2001-1}, Milnor's link invariants are interpreted as a spanning set for the set of cocycles in $H^3(F/F_k)$, allowing one to compute Milnor's numerical invariants from the Milnor invariants defined and studied in this paper.
  Explicit formulae for these cocycles are derived and evaluated on an {\em Igusa Picture} representing the homology class $\theta_k(M_L)$.  So, we may also ask: can one read the homotopy class of the longitudes using $3$-dimensional cocycles in $H^3(\widehat{F}/\widehat{F}_\kappa)$ for any ordinal~$\kappa$, thus establishing a numerical formulation for transfinite Milnor's invariants of links?
  
  These problems remain open for transfinite ordinals, and possibly hinge on obtaining a deeper computational understanding of the transfinite lower central series of local groups, and especially free local groups.
\end{remark}

\section{Torus bundle example: invariants of finite length}
\label{section:tb-finite-computation}

Let $Y$ be the torus bundle with monodromy $h\colon T^2\to T^2$ given by
$\sbmatrix{-1 & 0 \\ 0 & -1}$.  That is,
\begin{equation}
  \label{equation:torus-bundle}
  Y = T^2\times [0,1] / (h(x),0) \sim (x,0).
\end{equation}
Let $\Gamma=\pi_1(Y)$ be the fundamental group.  The group $\Gamma$ is an HNN
extension $\Z^2 \rtimes \Z$ of $\pi_1(T^2)=\Z^2$ by $\Z=\langle t\rangle$, which
acts on $\Z^2$ by $t(a,b)t^{-1}=(-a,-b)$.

The goal of this section is to study the invariant $\theta_k$ of finite length
over the torus bundle group~$\Gamma$.  The cases of transfinite length
invariants and the final invariant are investigated in
Sections~\ref{section:tb-transfinite-computation},
\ref{section:tb-computation-final} and~\ref{section:modified-tb}.  Readers eager
to see the transfinite case may wish to skip this section on a first
reading.

The following theorem summarizes the result of our computation of finite length
invariants.  In what follows, $\Z_d=\Z/d\Z$ denotes the finite cyclic group of
order~$d$, and $\Z_d^\times=\{r\in \Z_d \,|\, \gcd(r,d)=1\}$ denotes the
multiplicative group of units in~$\Z_d$.

\begin{theorem}
  \phantomsection\label{theorem:tb-finite-computation}
  For finite $k\ge 2$, the following hold.
  \begin{enumerate}
    \item \label{item:tb-finite-H_3}
    The third homology is given by $H_3(\Gamma/\Gamma_k) = (\Z_{2^{k-1}})^4$.

    \item \label{item:tb-finite-realization}
    The set of realizable classes in $H_3(\Gamma/\Gamma_k)$ is given
    by
    \begin{align*}
      \cR_k(\Gamma) &= \begin{cases}
        \{(a,b,c,r)\in (\Z_2)^4 \mid ac+b+r=1 \}
        &\text{for }k=2,
        \\
        (\Z_{2^{k-1}})^3\times \Z_{2^{k-1}}^\times
        &\text{for } 3\le k < \infty.
      \end{cases}
    \end{align*}

    \item \label{item:tb-finite-image-k+1}
    The map $\cR_{k+1}(\Gamma) \to \cR_k(\Gamma)$ induced by the projection
    $\Gamma/\Gamma_{k+1} \to \Gamma/\Gamma_k$ is given by
    \[
      \begin{cases}
      \begin{array}[t]{@{}r@{}c@{}l@{}}
        (\Z_4)^3\times \Z_4^\times
        &{}\longrightarrow{}
        &\{(a,b,c,r)\in (\Z_2)^4 \mid ac+b+r=1\}
        \\
        (a,b,c,r)\hphantom{\{} &{}\longmapsto{}&
        \hphantom{\{}(0,0,0,r)
      \end{array}
      &\text{for } k=2,
      \\
      \leavevmode\vrule height 3ex width 0ex
      \begin{array}[t]{@{}r@{}c@{}l@{}}
        (\Z_{2^{k}})^3\times \Z_{2^k}^\times
        &{}\longrightarrow{}
        &(\Z_{2^{k-1}})^3\times \Z_{2^{k-1}}^\times
        \\
        (a,b,c,r)\hphantom{\{} &{}\longmapsto{}&
        \hphantom{\{}(2a,2b,2c,r)
      \end{array}
      & \text{for }3\le k<\infty.
      \end{cases}
    \]
    \item \label{item:tb-finite-aut} For every automorphism $\phi$ on
    $\Gamma/\Gamma_k$, the induced bijection $\phi_*\colon \cR_k(\Gamma) \to
    \cR_k(\Gamma)$ sends $\Im\{\cR_{k+1}(\Gamma) \to \cR_k(\Gamma)\}$ onto
    itself.  Consequently, $\theta\in \cR_k(\Gamma)$ vanishes in the cokernel of
    $\cR_{k+1}(\Gamma) \to \cR_k(\Gamma)$ if and only if $\theta$ vanishes in
    the cokernel of $\cR_{k+1}(\Gamma) \to \cR_k(\Gamma)/\Aut(\Gamma/\Gamma_k)$.
    
  \end{enumerate}
\end{theorem}

From Theorem~\ref{theorem:tb-finite-computation}\ref{item:tb-finite-aut}
and Theorems
\ref{theorem:main-determination-lcsq-lift}
and~\ref{theorem:main-determination-lcsq-nonlift}, the following corollary
is immediately obtained.

\begin{corollary}
  Let $k\ge 2$ be finite.  Suppose $M$ is a closed 3-manifold with
  $\pi=\pi_1(M)$ and $f\colon \pi/\pi_k \isomto \Gamma/\Gamma_k$ is an
  isomorphism. Then $f$ lifts to an isomorphism $f\colon \pi/\pi_{k+1} \isomto
  \Gamma/\Gamma_{k+1}$ if and only if there is an isomorphism $\pi/\pi_{k+1}
  \isomto \Gamma/\Gamma_{k+1}$ (which is not required to be a lift).
\end{corollary}

Using Theorem~\ref{theorem:tb-finite-computation}\ref{item:tb-finite-aut}, we
can also obtain the following estimate of the number of isomorphism classes of
the $(k+1)$st lower central quotients of 3-manifold groups with the same $k$th
lower central quotient as that of the torus bundle.

\begin{corollary}
  \label{corollary:tb-finite-tower-estimate}
  For each finite $k\ge 2$,
  \[
    2 \le
    \# \bigg( \bigg\{ \pi/\pi_{k+1} \,\bigg|\,
    \begin{tabular}{@{}c@{}}
      $\pi=\pi_1(M)$ for a closed 3-manifold\\ $M$
      such that $\pi/\pi_k \cong \Gamma/\Gamma_k$
    \end{tabular}
    \bigg \} \bigg/ \text{isomorphism} \bigg)
    \le 7\cdot 2^{4(k-2)}+1.
  \]
\end{corollary}

\begin{proof}
  By Theorem~\ref{theorem:tb-finite-computation}\ref{item:tb-finite-image-k+1}
  and~\ref{item:tb-finite-aut}, there is a class $\theta\in \cR_k(\Gamma)$ which
  does not vanish in the cokernel of $\cR_{k+1}(\Gamma) \to
  \cR_k(\Gamma)/\Aut(\Gamma/\Gamma_k)$.  From this, it follows that there exist
  at least two isomorphism classes of $\pi/\pi_{k+1}$ with $\pi=\pi_1(M)$ for
  some closed 3-manifold $M$ such that $\pi/\pi_k \cong \Gamma/\Gamma_k$, by
  Theorem~\ref{theorem:main-determination-lcsq-nonlift}\@.  This proves the
  lower bound in the statement.

  By
  Theorem~\ref{theorem:tb-finite-computation}\ref{item:tb-finite-realization}
  and~\ref{item:tb-finite-image-k+1}, we have
  \[
    \# \cR_k(\Gamma) = (2^{k-1})^3 \cdot 2^{k-2}, \quad
    \#\Im\{\cR_{k+1}(\Gamma)\rightarrow\cR_k(\Gamma)\}
    = (2^{k-2})^4.
  \]
  By definition, $\theta\in \cR_k(\Gamma)$ is equivalent to $\theta_k(Y)$ if and
  only if $\theta$ lies in the image of $\cR_{k+1}(\Gamma)$.  So, it follows
  that
  \[
    \# (\cR_k(\Gamma)/\mathord{\sim}) \le
    \#\cR_k(\Gamma) - \#\Im\{\cR_{k+1}(\Gamma)\rightarrow\cR_k(\Gamma)\} + 1
    = 7\cdot 2^{4(k-2)}+1.
  \]

  By
  Corollary~\ref{corollary:main-tower-classifications}\ref{item:tower-classification-nonlifts},
  the number of isomorphism classes of $\pi/\pi_{k+1}$ concerned in the
  statement is bounded above by $\# (\cR_k(\Gamma)/{\approx})$, which is in turn
  bounded above by $\# (\cR_k(\Gamma)/\mathord{\sim})$.  From this, the desired
  upper bound is obtained.
\end{proof}

Indeed, by
Corollary~\ref{corollary:main-tower-classifications}\ref{item:tower-classification-lifts},
and by the upper bounded of $\#(\cR_k(\Gamma)/{\sim})$ in the last step of the
above proof, it follows that Theorem~\ref{theorem:main-tb-finite} holds, which
asserts that
\[
  2\le \# \bigg\{ \begin{tabular}{@{}c@{}}
    equivalence classes of \\
    length $k+1$ extensions of $\{\Gamma/\Gamma_\lambda\}_{\lambda\le k}$
  \end{tabular} \bigg\} \le 7\cdot
  2^{4(k-2)} +1.
\]

We remark that the estimates in
Corollary~\ref{corollary:tb-finite-tower-estimate} (and that in
Theorem~\ref{theorem:main-tb-finite}) are not sharp.  Further investigation of
the equivalence relation and automorphism action on $\cR_k(\Gamma)$ gives us
improved bounds.  We do not address this here.

The rest of this section is devoted to the proof of
Theorem~\ref{theorem:tb-finite-computation}.  We begin with the lower central
quotient computation.  For $(a,b)\in \Z^2\subset \Gamma$, we have $[t,(a,b)] =
(-2a,-2b)$.  By using this equation inductively, it follows that the $k$th lower
central subgroup of $\Gamma$ is given by
\begin{equation}
  \label{equation:lcs-Z^2-rtimes-Z}
  \Gamma_k = 2^{k-1}\Z^2 \subset \Z^2 \subset \Gamma.
\end{equation}
Consequently, the lower central quotient is given by
\[
  \Gamma/\Gamma_k = (\Z_{2^{k-1}})^2 \rtimes \Z,
\]
where $\Z=\langle t\rangle$ acts on $(\Z_{2^{k-1}})^2$ by $t(a,b)t^{-1} =
(-a,-b)$.

The remainder of this section is devoted to the proof of Theorem~\ref{theorem:tb-finite-computation}.
In Section~\ref{subsection:tb-finite-homology}, we compute the homology of $\Gamma/\Gamma_k$ and prove Theorem~\ref{theorem:tb-finite-computation}\ref{item:tb-finite-H_3}.
In Section~\ref{subsection:tb-finite-realizable-classes}, we study the cap product structure on $\Gamma/\Gamma_k$ and prove Theorem~\ref{theorem:tb-finite-computation}\ref{item:tb-finite-realization} and~\ref{item:tb-transfinite-image-kappa+1}.
In Section~\ref{subsection:tb-finite-aut}, we study the action of $\Aut(\Gamma/\Gamma_k)$ on $\cR_k(\Gamma)$ and prove Theorem~\ref{theorem:tb-finite-computation}\ref{item:tb-finite-aut}.

\subsection{Cell structure of $B(\Gamma/\Gamma_k)$ and homology}
\label{subsection:tb-finite-homology}

To compute homology of $\Gamma/\Gamma_k$, we will use cellular chain complexes.
Although spectral sequences provide an alternative approach for HNN extensions,
the cellular method turns out to be more efficient for our purpose.  We will use
the following standard facts.

\begin{enumerate}

  \item
  For the finite cyclic group $\Z_d = \langle g\mid g^d \rangle$ of order $d$,
  $B(\Z_d)$ has a cell structure with exactly one $i$-cell $e^i$ in each dimension
  $i\ge 0$.  The boundary operator of the cellular chain complex
  $C_\bullet(B(\Z_d);\Z[\Z_d])$ is given by
  \begin{equation*}
    \partial e^{2i+1} = (1-g)e^{2i}, \quad
    \partial e^{2i} =  (1+g+\cdots+g^{d-1})e^{2i-1}.
  \end{equation*}

  \item
  Let $G=A\rtimes \Z$ be an HNN extension of an abelian group $A$ determined by
  an automorphism $h\colon A\to A$, that is, $\Z=\langle t\rangle$ acts on $A$
  by $tat^{-1} = h(a)$.  For a given cell structure of $B(A)$, we may assume
  that $h$ is realized by a cellular map $h\colon B(A)\to B(A)$.  Then $B(G)$
  has an associated cell structure, whose $n$-cells are of the form $e^p \times
  \epsilon^q$ with $p+q=n$, $q=0,1$, $e^p$ a $p$-cell of~$B(A)$ and $\epsilon^q$
  ($q=0,1$) an abstract $q$-cell.  The boundary operator of $C_\bullet(BG;\Z G)$
  is given by
  \begin{equation*}
    \begin{aligned}
      \partial (e^p\times \epsilon^0) &= (\partial e^p)\times \epsilon^0, \\
      \partial (e^p\times \epsilon^1) &=
        (\partial e^p)\times \epsilon^1
        +(-1)^p (t\cdot e^p \times \epsilon^0 - h(e^p)\times \epsilon^0).
    \end{aligned}
  \end{equation*} 
\end{enumerate}

\def\pb#1#2#3{e^{#1}\times e^{#2}\times\epsilon^{#3}}

Let $d=2^{k-1}$ and write $\Z_d^2=(\Z_d)^2$ for brevity.  Take the product
$B(\Z_d^2) = B(\Z_d)\times B(\Z_d)$ of the cell complex in~(1), and construct
$B(\Gamma/\Gamma_k) = B(\Z_d^2\rtimes\Z)$ using~(2).  Cells of dimension $n$ in
$B(\Gamma/\Gamma_k)$ are of the form $e^i\times e^j\times \epsilon^q$ with
$i+j+q=n$, $q=0,1$.  The negation homomorphism $h(g)=g^{-1}$ on $\Z_d$ induces
(the chain homotopy class of) the chain map $C_\bullet(\Z_d;\Z[\Z_d]) \to
C_\bullet(\Z_d;\Z[\Z_d])$ given by
\[
  h(e^{2k-1}) = (-1)^k g^{-1}e^{2k-1}, \quad 
  h(e^{2k}) = (-1)^k e^{2k}.
\]
and the monodromy $h\colon B(\Z_d^2) \to B(\Z_d^2)$ is given by
$h(e^i\times e^j) = h(e^i)\times h(e^j)$.  Using this together with the above
(1), (2) and the product boundary formula, it is straightforward to compute the
cellular chain complex $C_\bullet(\Gamma/\Gamma_k;\Z[\Gamma/\Gamma_k])$.
Applying the augmentation $\Z[\Gamma/\Gamma_k]\to \Z$, it is seen that
$C_\bullet(\Gamma/\Gamma_k) = C_\bullet(\Gamma/\Gamma_k;\Z)$ has the following
boundary operators in dimension${}\le 4$.
\begin{gather*}
  \partial_1\colon \begin{array}[t]{@{}r@{{}\mapsto{}}l@{}}
    \pb100 & 0\\
    \pb010 & 0\\
    \pb001 & 0
  \end{array},\qquad
  \partial_2\colon \begin{array}[t]{@{}r@{{}\mapsto{}}l@{}}
    \pb200 & d\cdot\pb100\\
    \pb110 & 0\\
    \pb020 & d\cdot\pb010\\
    \pb101 & -2\pb100\\
    \pb011 & -2\pb010
  \end{array},
  \\[.5ex]
  \partial_3\colon \begin{array}[t]{@{}r@{{}\mapsto{}}l@{}}
    \pb300 & 0\\
    \pb210 & d\cdot\pb110\\
    \pb120 & -d\cdot\pb110\\
    \pb030 & 0\\
    \pb201 & \begin{array}[t]{@{}l@{}}d\cdot\pb101\\ \,\, +2\cdot\pb200\end{array}\\
    \pb111 & 0\\
    \pb021 & \begin{array}[t]{@{}l@{}}d\cdot\pb011\\ \,\, +2\cdot\pb020\end{array}\\ 
  \end{array},\qquad
  \partial_4\colon \begin{array}[t]{@{}r@{{}\mapsto{}}l@{}}
    \pb400 & d\cdot\pb300\\
    \pb310 & 0\\
    \pb220 & \begin{array}[t]{@{}l@{}}d\cdot\pb120 \\ \,\,+d\cdot\pb210\end{array}\\
    \pb130 & 0\\
    \pb040 & d\cdot\pb030\\
    \pb301 & 0\\
    \pb211 & d\cdot\pb111\\
    \pb121 & -d\cdot\pb111\\
    \pb031 & 0\\
  \end{array}.
\end{gather*}

The homology groups $H_i(\Z_d^2\rtimes\Z)$ $(i\le 3)$ are immediately obtained
from this.
\begin{align}
  \label{equation:tb-finite-H_1}
  H_1(\Z_d^2\rtimes\Z) & =\Z_2^2,
  \\
  \label{equation:tb-finite-H_2}
  H_2(\Z_d^2\rtimes\Z) &=\Z_2^2\times \Z_d,
  \\
  \label{equation:tb-finite-H_3}
  H_3(\Z_d^2\rtimes\Z) &=\Z_d^4.
\end{align}
This shows Theorem~\ref{theorem:tb-finite-computation}\ref{item:tb-finite-H_3}.
In addition, the four $\Z_d$ factors of $H_3(\Z_d^2\rtimes\Z)$ are respectively
generated by
\begin{equation}
  \label{equation:tb-finite-H_3-generators}
  \begin{aligned}
    \xi_1 &= \pb300,\\
    \xi_2 &= \pb210+\pb120,\\
    \xi_3 &= \pb030,\\
    \zeta &= \pb111.
  \end{aligned}
\end{equation}

Here, the basis element $\zeta\in H_3(\Gamma/\Gamma_k)$ is the image of the
fundamental class $[Y] \in H_3(Y)$ under $H_3(Y)\to H_3(\Gamma/\Gamma_k)$. In
other words, $\theta_k(Y)=\zeta$.  To verify this, observe that $Y$ is a
subcomplex of $B(\Gamma/\Gamma_k)$ consisting of cells $e^i\times e^j\times
\epsilon^q$ with $i, j, q\in \{0,1\}$. By computing $H_i(Y)$ using this
subcomplex, it is seen that $e^1\times e^1\times \epsilon^1$ generates
$H_3(Y)=\Z$.

Also, viewing $B(\Z_d^2)$ as a subcomplex of $B(\Z_d^2\rtimes\Z)$, it is seen
that the subgroup generated by $\xi_1$, $\xi_2$ and $\xi_3$ is the isomorphic
image of $H_3(\Z_d^2)$ under the inclusion-induced map.

The above chain level computation also enables us to compute the
projection-induced homomorphism $H_3(\Gamma/\Gamma_{k+1}) \to
H_3(\Gamma/\Gamma_k)$. First, consider the projection $\Z_{2d} \to \Z_d$.  Abuse
notation to denote the $i$-cells of $B(\Z_{rd})$ and $B(\Z_d)$ by the same
symbol~$e^i$.  A routine computation shows that the induced chain map
$C_\bullet(\Z_{2d})\to C_\bullet(\Z_d)$ is given by $e^{i} \mapsto 2^{\lfloor
i/2 \rfloor} \cdot e^{i}$. (For instance, $e^1\mapsto e^1$ while $e^2\mapsto
2e^2$.)  From this, it follows that the projection
\[
  \Gamma/\Gamma_{k+1}=\Z_{2d}^2\rtimes\Z \to
  \Gamma/\Gamma_{k}=\Z_d^2\rtimes\Z
\]
induces the chain map $C_\bullet(\Gamma/\Gamma_{k+1}) \to
C_\bullet(\Gamma/\Gamma_k)$ given by
\[
  e^i\times e^j\times \epsilon^q \mapsto
  2^{\lfloor i/2 \rfloor+\lfloor j/2 \rfloor} \cdot
  e^i\times e^j\times\epsilon^q.
\]
Therefore, $H_3(\Gamma/\Gamma_{k+1}) \to H_3(\Gamma/\Gamma_{k})$ is the
homomorphism
\begin{equation}
  \label{equation:tb-finite-induced-map-H_3}
  \xi_i \mapsto 2\cdot \xi_i \text{ for }i=1,2,3,\quad \zeta\mapsto \zeta.
\end{equation}

\subsection{Realizable classes}
\label{subsection:tb-finite-realizable-classes}

Now we compute the realizable classes in $H_3(\Gamma/\Gamma_k)$. Fix $\theta\in
H_3(\Z_d^2\rtimes\Z)=H_3(\Gamma/\Gamma_k)$ where $d=2^{k-1}$ with $k\ge 2$ as
before.  To apply Theorem~\ref{theorem:main-realization}, we will investigate
the following cap product maps.
\begin{gather}
  \label{equation:tb-finite-cap-H^2-to-H_1}
  {}\cap\theta\colon tH^2(\Z_d^2\rtimes\Z) \to tH_1(\Z_d^2\rtimes\Z)
  \\
  \label{equation:tb-finite-cap-H^1-to-H_2}
  {}\cap\theta\colon H^1(\Z_d^2\rtimes\Z) \to H_2(\Z_d^2\rtimes\Z)
  / K_{k-1}(\Gamma/\Gamma_k)
\end{gather}
Here, $K_{k-1}(\Gamma/\Gamma_k)$ is the kernel of
$H_2(\Z_d^2\rtimes\Z)=H_2(\Gamma/\Gamma_k) \to H_2(\Gamma/\Gamma_{k-1})$.

\begin{step-named}[Case 1] Suppose $k\ge 3$, that is, $d=2^{k-1}$ is divisible
  by~$4$. 
\end{step-named}

Recall that $H_3(\Z_d^2\rtimes\Z)$ has basis $\{\xi_1,\xi_2,\xi_3,\zeta\}$
described in~\eqref{equation:tb-finite-H_3-generators}.  Let $\theta\in
H_3(\Z_d^2\rtimes\Z)$ be a class which is a linear combination of $\xi_1$,
$\xi_2$ and~$\xi_3$.  Since each $\xi_i$ is of the form $\bullet \times \bullet
\times \epsilon^0$ in~\eqref{equation:tb-finite-H_3-generators}, $\xi_i$ lies
in the image of the inclusion-induced map $i_*\colon H_3(\Z_d^2) \to
H_3(\Z_d^2\rtimes\Z)$.  Write $\theta=i_*(z)$ for some $z\in
H_3(\Z_d^2)$.  Consider the following commutative diagram.
\[
  \begin{tikzcd}[sep=large]
    \llap{$\Z_2^2={}$} tH^2(\Z_d^2\rtimes\Z)
    \ar[r,"{}\cap\theta"] \ar[d,"i^*"']
    & tH_1(\Z_d^2\rtimes\Z) \rlap{${}=\Z_2^2$}
    \\
    \llap{$\Z_d^2={}$} tH^2(\Z_d^2) \ar[r,"{}\cap z"']
    & tH_1(\Z_d^2) \rlap{${}=\Z_d^2$} \ar[u,"i_*"',two heads]
  \end{tikzcd}
\]
Here, $tH_1(\Z_d^2\rtimes \Z)=\Z_2^2$ by~\eqref{equation:tb-finite-H_1},
$H_1(\Z_d^2)=\Z_d^2$ obviously, so $tH^2(\Z_d^2\rtimes \Z) =
\Ext(H_1(\Z_d^2\rtimes \Z),\Z) = \Z_2^2$ and $tH^2(\Z_d^2)=\Z_d^2$.  Let $c\in
tH^2(\Z_d^2\rtimes\Z)$.  Since $2c=0$ and all order 2 elements in
$tH^2(\Z_d^2)=\Z_d^2$ are multiples of~$d/2$, $i^*(c)$ is a multiple of~$d/2$.
So, $i_*(i^*(c) \cap z) = c\cap\theta$ is a multiple of~$d/2$, which is a
multiple of $2$ since $d=2^{k-1}$ with $k\ge 3$. It follows that $c\cap
\theta=0$, since it lies in $tH_1(\Z_d^2\rtimes\Z) = \Z_2^2$.  This shows that
the cap product~\eqref{equation:tb-finite-cap-H^2-to-H_1} is zero.
Also, the cap product~\eqref{equation:tb-finite-cap-H^1-to-H_2} is zero since
$H^1(\Z_d^2)=0$ and the following diagram commutes.
\[
  \begin{tikzcd}[sep=large]
    H^1(\Z_d^2\rtimes\Z)
    \ar[r,"{}\cap\theta"] \ar[d,"i^*"']
    & H_2(\Z_d^2\rtimes\Z)
    \\
    \llap{$0={}$} H^1(\Z_d^2) \ar[r,"{}\cap z"']
    & H_2(\Z_d^2)  \ar[u,"i_*"']
  \end{tikzcd}
\]

Now, consider a class of the form $\theta=r\zeta$ with $r\in \Z$. Since $\zeta$
is the image of the fundamental class $[Y]\in H_3(\Gamma)$, $\zeta$ is
realizable, that is, $\zeta\in\cR_k(\Gamma)$. By
Theorem~\ref{theorem:main-realization}, the cap
product~\eqref{equation:tb-finite-cap-H^2-to-H_1} is an isomorphism for
$\theta=\zeta$.  From this it follows that
\eqref{equation:tb-finite-cap-H^2-to-H_1} is an isomorphism for
$\theta=r\zeta$ if and only if $r$ is odd, since $tH_1(\Z_d^2\rtimes\Z)$ is a
2-group by~\eqref{equation:tb-finite-H_2}. Also, the cap
product~\eqref{equation:tb-finite-cap-H^1-to-H_2} is surjective for the
realizable class $\theta=\zeta\in \cR_k(\Gamma)$, by
Theorem~\ref{theorem:main-realization}\@.  From this it follows that
\eqref{equation:tb-finite-cap-H^1-to-H_2} is surjective for
$\theta=r\zeta$ if $r$ is odd, since $H_2(\Z_d^2\rtimes\Z)$ is a $2$-group.

Combine the above conclusions, to obtain the following: for an arbitrary class
\[
  \theta = a\xi_1 + b\xi_2 + c\xi_3 + r\zeta \in H_3(\Z_d^2\rtimes\Z),
\]
the above \eqref{equation:tb-finite-cap-H^2-to-H_1} is an
isomorphism and \eqref{equation:tb-finite-cap-H^1-to-H_2} is surjective
if and only if $r$ is odd.  Applying Theorem~\ref{theorem:main-realization},
this proves
Theorem~\ref{theorem:tb-finite-computation}\ref{item:tb-finite-realization}
for $k\ge 3$.

\begin{step-named}[Case 2]
  Suppose $k=2$, that is, $d=2^{k-1}=2$.
\end{step-named}

In this case, first note that $\Gamma/\Gamma_{k-1}=\Gamma/\Gamma_1$ is trivial
by definition, and thus the cap
product~\eqref{equation:tb-finite-cap-H^1-to-H_2} is onto the trivial
group.  So, it suffices to determine when the cap
product~\eqref{equation:tb-finite-cap-H^2-to-H_1} is an isomorphism.

Observe that the semi-direct product $\Z_2^2\rtimes\Z$ is equal to the ordinary
product $\Z_2^2\times\Z$ since $-a=a$ in~$\Z_2$.  This enables us to compute the
cap product directly using the standard product cell structures. To prevent
confusion from the semi-direct product case, denote the $i$-cell of $B(\Z)=S^1$
by $u^i$ ($i=0,1$), while cells of $B(\Z_2)$ are denoted by $e^i$ as before.  It
is well known that
\[
  \Delta(e^i)=\sum_{p+q=i} (-1)^{pq} e_p\times e_q, \quad
  \Delta(u^i) = \sum_{p+q=i} u^p\times e^q
\]
are cellular approximations of the diagonal maps $B(\Z_2) \to B(\Z_2)\times
B(\Z_2)$ and $B(\Z) \to B(\Z)\times B(\Z)$, and thus the chain level cup product
of $B(\Z_d)$ and $B(\Z)$ defined using them are given by
\[
  (e^i)^* \cup (e^j)^* = (-1)^{ij} \cdot (e^{i+j})^*, \qquad
  (u^i)^* \cup (u^j)^* = (u^{i+j})^*.
\]
Here and in what follows, for brevity, we use the convention that $e^i=0$ for
$i<0$ and $u^i=0$ for $i\notin\{0,1\}$.  Using this notation, the cap product is
given by
\[
  (e^i)^* \cap e^j = (-1)^{i(j-i)} \cdot e^{j-i}, \qquad
  (u_i)^* \cap u^j = u^{j-i}.
\]
Therefore, the cap product on the product $B(\Z_2^2\times\Z)$ is as follows:
\begin{equation}
  \label{equation:tb-product-cap-product}
  (e^i\times e^j\times u^p)^* \cap (e^k\times e^\ell\times u^q) 
  = (-1)^{jk+pk+pl} \cdot e^{k-i} \times e^{\ell-j} \times u^{q-p}.
\end{equation}

Note that the product cell structure we use here is different from the HNN
extension cell structure we used in Case~1.  To compute the cap product for the
basis elements in~\eqref{equation:tb-finite-H_3-generators} which are
expressed in terms of the cells $e^i\times e^j\times \epsilon^q$, we need to
rewrite them in terms of the product cells $e^i\times e^j\times u^q$.  The three
basis elements $\xi_1$, $\xi_2$ and $\xi_3$
in~\eqref{equation:tb-finite-H_3-generators} are already in this form, since
$\epsilon^0$ is identical with $u^0$.  To make the computation for
$\zeta=e^1\times e^2\times \epsilon^1$ simpler, consider the projection
$\Z^2\rtimes\Z \to \Z_2^2\rtimes \Z = \Z_2^2\times \Z$.  It is straightforward
to verify that this induces the homotopy class of a chain map
\[
  C_\bullet(B(\Z^2\rtimes\Z);\Z[\Z^2\rtimes\Z]) \to
  C_\bullet(B(\Z_2^2\times\Z);\Z[\Z_2^2\times\Z])
\]
which is given by $e^i\times e^j\times \epsilon^q \mapsto e^i\times e^j\times
u^q$ in dimension $i+j+q \le 1$ and by
\[
  \begin{array}{r@{{}\mapsto{}}l}
    e^1\times e^0\times \epsilon^1 
    & e^1\times e^0\times u^1 - e^2\times e^0\times u^0
    \\
    e^0\times e^1\times \epsilon^1 
    & e^0\times e^1\times u^1 - e^0\times e^2\times u^0
    \\
    e^1\times e^1\times \epsilon^1 
    & e^1\times e^1\times u^1 - e^1\times e^2\times u^0
    - e^2 \times (g\cdot e^1) \times u^0
  \end{array}
\]
in dimensions $2$ and~$3$.  So, applying the augmentation, $\zeta=e^1\times
e^1\times \epsilon^1$ is expressed, in the product complex
$C_\bullet(B(\Z_2^2\times\Z);\Z)$, as
\begin{equation}
  \label{equation:tb-finite-fundamental-class}
  \zeta = e^1\times e^1\times u^1 - e^1\times e^2\times u^0
  - e^2 \times e^1 \times u^0.  
\end{equation}

Now we are ready to compute the cap
product~\eqref{equation:tb-finite-cap-H^2-to-H_1}.  By a routine
computation, it is verified that
\begin{align*}
  H^2(\Z_2^2\rtimes\Z) &= \Z_2^2 \quad\text{with basis }
  \{(e^2\times e^0\times u^0)^*,  (e^0\times e^2\times u^0)^*\},
  \\
  tH_1(\Z_2^2\rtimes\Z) &= \Z_2^2 \quad\text{with basis }
  \{e^1\times e^0\times u^0, e^0\times e^1\times u^0\}.
\end{align*}
Let $\theta=a\xi_1+b\xi_2+c\xi_3+r\zeta\in H_3(\Z_2^2\rtimes\Z)$.
From~\eqref{equation:tb-product-cap-product}
and~\eqref{equation:tb-finite-fundamental-class}, it follows that
the cap product $\cap\,\theta$
in~\eqref{equation:tb-finite-cap-H^2-to-H_1} is given by
\[
  \begin{bmatrix}
    a & b-r \\ b-r & c
  \end{bmatrix}
\]
with respect to the above basis.  It follows that $\cap\,\theta$ is an
isomorphism if and only if $ac+b+r$ is odd.  This completes the proof of
Theorem~\ref{theorem:tb-finite-computation}\ref{item:tb-finite-realization}
for $k=2$.

Once $\cR_k(\Gamma)$ is computed as above,
Theorem~\ref{theorem:tb-finite-computation}\ref{item:tb-finite-image-k+1}
follows immediately from the description of the projection-induced homomorphism
in~\eqref{equation:tb-finite-induced-map-H_3}.

\subsection{Automorphism action on the realizable classes}
\label{subsection:tb-finite-aut}

As above, let $d=2^{k-1}$, and write $\Gamma/\Gamma_k = \Z_d^2\rtimes\Z$.
Suppose $\phi\colon \Z_d^2\rtimes\Z \to \Z_d^2\rtimes\Z$ is an automorphism.  It
induces an automorphism $\phi_*\colon H_3(\Gamma/\Gamma_k)\to
H_3(\Gamma/\Gamma_k)$, which restricts to a bijection $\phi_*\colon
\cR_k(\Gamma) \to \cR_k(\Gamma)$.  Our goal is to show
Theorem~\ref{theorem:tb-finite-computation}\ref{item:tb-finite-aut}, which says
that $\phi_*$ sends $\Im\{\cR_{k+1}(\Gamma) \to \cR_k(\Gamma)\}$ onto itself
bijectively.

As the first step, we claim that $\phi$ sends the subgroup $\Z_d^2\subset
\Z_d^2\rtimes \Z$ isomorphically onto $\Z_d^2$ itself.  To see this, observe
that $\Z_d^2$ is the kernel of the horizontal composition in the following
diagram:
\[
  \begin{tikzcd}[row sep=large]
    \Z_d^2\rtimes \Z \ar[r] \ar[d,"\phi"]
    & H_1(\Z_d^2\rtimes\Z) =\Z_d^2\times\Z \ar[r] \ar[d,"\phi_*"]
    & H_1(\Z_d^2\rtimes\Z;\Q) = \Q \ar[d,"\phi_*"]
    \\
    \Z_d^2\rtimes \Z \ar[r]
    & H_1(\Z_d^2\rtimes\Z) =\Z_d^2\times\Z \ar[r]
    & H_1(\Z_d^2\rtimes\Z;\Q) = \Q
  \end{tikzcd}
\]
Since vertical arrows are automorphisms, the claim follows from the
commutativity of the diagram.

Recall from Section~\ref{subsection:tb-finite-homology} that the
subgroup $\langle \xi_1,\xi_2,\xi_3\rangle$ is the (isomorphic) image of
$H_3(\Z_d^2)$ in $H_3(\Z_d^2\rtimes\Z)$.  So, by the claim, $\phi_* \langle
\xi_1,\xi_2,\xi_3\rangle$ is equal to $\langle \xi_1,\xi_2,\xi_3\rangle$.  From
this it follows that
\begin{equation}
  \label{equation:tb-finite-aut-on-xi_i}
  \phi_*\langle 2\xi_1,2\xi_2,2\xi_3\rangle=\langle 2\xi_1,2\xi_2,2\xi_3\rangle.
\end{equation}

In addition, recall from Section~\ref{subsection:tb-finite-homology} that
$\zeta\in H_3(\Z_d^2\rtimes\Z)=H_3(\Gamma/\Gamma_k)$ is the image of the
fundamental class $[Y] \in H_3(\Gamma)$.  The quotient homomorphism $\Gamma \to
\Gamma/\Gamma_k$ factors through $\Gamma/\Gamma_{k+1}$, and $[Y]$ is sent into
$\cR_{k+1}(\Gamma)$ by definition.  From this, it follows that $\zeta$ lies in
the image of $\cR_{k+1}(\Gamma)$.  By
Theorem~\ref{theorem:tb-finite-computation}\ref{item:tb-finite-image-k+1},
\[
  \Im\{\cR_{k+1}(\Gamma) \rightarrow \cR_k(\Gamma)\}
  = \{\eta+r\zeta\mid \eta\in \langle 2\xi_1,2\xi_2,2\xi_3\rangle,\, r\in2\Z+1\}.
\] 
So, $\phi_*(\zeta)=\eta_0+r_0\zeta$ for some $\eta_0\in \langle
2\xi_1,2\xi_2,2\xi_3\rangle$ and some odd~$r_0$.  Now, for an arbitrary
$\eta+r\zeta \in \Im\{\cR_{k+1}(\Gamma) \rightarrow \cR_k(\Gamma)\}$ with
$\eta\in \langle 2\xi_1,2\xi_2,2\xi_3\rangle$ and $r$ odd, we have
\[
  \phi_*(\eta+r\zeta) = \phi_*(\eta) + r\eta_0 +rr_0\zeta.
\]
Here, $\phi_*(\eta) + r\eta_0 \in \langle 2\xi_1,2\xi_2,2\xi_3\rangle$
by~\eqref{equation:tb-finite-aut-on-xi_i}, and $rr_0$ is obviously odd.  This
shows that $\phi_*$ sends $\Im\{\cR_{k+1}(\Gamma) \rightarrow \cR_k(\Gamma)\}$
into itself. Since $\phi_*$ is one-to-one and $\Im\{\cR_{k+1}(\Gamma)
\rightarrow \cR_k(\Gamma)\}$ is a finite set, it follows that $\phi_*$ restricts
to a bijection of $\Im\{\cR_{k+1}(\Gamma) \rightarrow \cR_k(\Gamma)\}$ onto
itself.  This completes the proof of
Theorem~\ref{theorem:tb-finite-computation}\ref{item:tb-finite-aut}.

\section{Torus bundle example: invariants of transfinite length}
\label{section:tb-transfinite-computation}

In this section, we study transfinite invariants over the torus bundle $Y$
defined in~\eqref{equation:torus-bundle}:
\[
  Y = T^2\times [0,1] / (h(x),0) \sim (x,0)
\]
where $h=\sbmatrix{-1 & 0 \\ 0 & -1}$.  As before, let $\Gamma=\pi_1(Y)$. Denote
the ring of integers localized at the prime~$2$ by
\[
  \Z_{(2)}=\{a/b\mid a\in \Z,\, b\in 2\Z+1\}.
\]
Let $\Z_{(2)}^\times =\{a/b \mid a,b\in 2\Z+1\}$ be the multiplicative group of
units.

The following theorem summarizes our computation of transfinite invariants
over the torus bundle.

\begin{theorem}
  \phantomsection\label{theorem:tb-transfinite-computation}
  For transfinite ordinals $\kappa$, the following hold.
  \begin{enumerate}

    \item \label{item:tb-transfinite-H_3}
    The third homology of $\Gamma/\Gamma_\kappa$ is given by
    \begin{align*}
      H_3(\widehat\Gamma/\widehat\Gamma_\kappa) &= 
      \begin{cases}
        \Z_{(2)} &\text{for } \kappa=\omega, \\
        (\Z_{(2)}/\Z)\times \Z &\text{for } \kappa\ge \omega+1.
      \end{cases}
    \end{align*}

    \item \label{item:tb-transfinite-realization}
    The set of realizable classes in $H_3(\Gamma/\Gamma_\kappa)$ is given by
    \[
      \cR_\kappa(\Gamma) =
      \begin{cases}
        \Z_{(2)}^\times
        &\text{for }\kappa=\omega,\\
        (\Z_{(2)}/\Z)\times\{\pm1\} &\text{for }\kappa \ge \omega+1.
      \end{cases}
    \]

    \item \label{item:tb-transfinite-image-kappa+1}
    The map $\cR_{\kappa+1}(\Gamma) \to \cR_\kappa(\Gamma)$ induced by
    $\Gamma/\Gamma_{\kappa+1} \to \Gamma/\Gamma_\kappa$ is given by
    \[
      \begin{cases}
        \begin{array}[t]{r@{}c@{}l@{}}
          (\Z_{(2)}/\Z)\times\{\pm1\}
          & {}\longrightarrow{}
          & \Z_{(2)}^\times
          \\
          (x,\epsilon)\hphantom{\{} & {}\longmapsto{} & \hphantom{\{}\epsilon
        \end{array}
        & \text{for $\kappa=\omega$,}
        \\
        \leavevmode\vrule height 3ex width 0ex
        \begin{array}[t]{r@{}c@{}l@{}}
          (\Z_{(2)}/\Z)\times\{\pm1\}
          & {}\longrightarrow{}
          & (\Z_{(2)}/\Z)\times\{\pm1\}
          \\
          (x,\epsilon)\hphantom{\{} & {}\longmapsto{} & \hphantom{\{}(x,\epsilon)
        \end{array}
        & \text{for $\kappa\ge\omega+1$.}
      \end{cases}
    \]

    \item \label{item:tb-transfinite-equiv-rel}
    On $\cR_\omega(\Gamma)=\Z_{(2)}^\times$, the equivalence relation $\sim$ is
    given by $r\sim r'$ if and only if $r=\pm r'$.  On $\cR_\kappa(\Gamma)$ with $\kappa\ge \omega+1$, $r\sim r'$ for all $r$, $r'\in \cR_\kappa(\Gamma)$.

    \item \label{item:tb-transfinite-aut}
    The automorphism group $\Aut(\widehat\Gamma/\widehat\Gamma_\omega)$ acts on
    $\cR_\omega(\Gamma)$ transitively.  Consequently
    $\cR_\omega(\Gamma)/{\approx}$ is trivial.
  \end{enumerate}
\end{theorem}

Combining
Theorem~\ref{theorem:tb-transfinite-computation}\ref{item:tb-transfinite-equiv-rel}
and~\ref{item:tb-transfinite-aut} with
Corollary~\ref{corollary:main-tower-classifications}\ref{item:tower-classification-lifts}
and Theorem~\ref{theorem:main-determination-lcsq-nonlift} respectively, the
following statements are immediately obtained.  Recall that the notion of
extensions of a transfinite lower central quotient tower and their equivalence
were introduced in Section~\ref{subsection:results-transfinite-tower}.

\begin{corollary}
  \phantomsection\label{corollary:tb-omega+1-tower-isom}
  \leavevmode\Nopagebreak
  \begin{enumerate}
    \item\label{item:tb-omega+1-tower}
    The set
    \[
      \bigg\{\begin{tabular}{@{}c@{}}
        length $\omega+1$ extensions, by 3-manifolds,
        \\[.2ex]
        of the length $\omega$ tower $\widehat\Gamma/\widehat\Gamma_\omega
        \to \cdots \to \Gamma/\Gamma_1 = \{1\}$
      \end{tabular}\bigg\} \bigg/
      \begin{tabular}{@{}c@{}}
        equivalence of \\ length $\omega+1$ extensions
      \end{tabular}
    \]
    is in one-to-one correspondence with the infinite set
    \[
      (\Z_{(2)}^\times)_{>0} := \{a/b \mid a,b\in 2\Z+1,\,a,b>0\}.
    \]
    \item\label{item:tb-omega+1-isom}
    If $M$ is a closed 3-manifold with $\pi=\pi_1(M)$ such that
    $\widehat\pi/\widehat\pi_\omega\cong \widehat\Gamma/\widehat\Gamma_\omega$,
    then $\widehat\pi/\widehat\pi_\kappa \cong
    \widehat\Gamma/\widehat\Gamma_\kappa$ for every ordinal $\kappa\ge
    \omega+1$.
  \end{enumerate}
\end{corollary}

This illustrates that the classification of tower extensions from length
$\omega$ to $\omega+1$ may have a completely different nature from the
determination of the isomorphism class of the $(\omega+1)$st lower central
quotient for a given $\omega$th lower central quotient.  For the case of the
torus bundle group $\Gamma$,
Corollary~\ref{corollary:tb-omega+1-tower-isom}\ref{item:tb-omega+1-tower} tells
us that the former has infinitely many solutions, while the latter has a unique
solution by
Corollary~\ref{corollary:tb-omega+1-tower-isom}\ref{item:tb-omega+1-isom}.  In
particular, over~$\Gamma$, $\bar\mu_\kappa(M)$ is trivial for all infinite
ordinals $\kappa$ whenever $\bar\mu_\kappa(M)$ is defined.  

We remark that our proof of Theorem~\ref{theorem:main-tb-milnor} in Section~\ref{section:modified-tb} presents modified torus bundle groups, over which there are infinitely many 3-manifolds $M$ with nontrivial~$\bar\mu_\omega(M)$.

The remaining part of this section is devoted to the proof of
Theorem~\ref{theorem:tb-transfinite-computation}. 
In Section~\ref{subsection:tb-localization}, we describe the homology localization~$\widehat\Gamma$ and its transfinite lower central series.
It turns out that the transfinite lower central series stabilizes at length $\omega+1$ with $\widehat\Gamma_{\omega+1}=\{1\}$.
In Section~\ref{subsection:tb-omega-H_3}, we study the homology and cap product structure of $\widehat\Gamma/\widehat\Gamma_\omega$ and prove Theorem~\ref{theorem:tb-transfinite-computation}\ref{item:tb-transfinite-H_3} and~\ref{item:tb-transfinite-realization} for $\kappa=\omega$.
In Sections~\ref{subsection:tb-omega+1-H_3} and~\ref{subsection:tb-omega+1-realizable-classes}, we study the homology and the cap product structure of~$\widehat\Gamma$, respectively, and prove Theorem~\ref{theorem:tb-transfinite-computation}\ref{item:tb-transfinite-H_3} and~\ref{item:tb-transfinite-realization} for $\kappa\ge\omega+1$ and Theorem~\ref{theorem:tb-transfinite-computation}\ref{item:tb-transfinite-image-kappa+1}.
In Section~\ref{subsection:tb-omega-aut}, we study the equivalence relation ${\sim}$ on $\cR_\omega(\Gamma)$ and prove Theorem~\ref{theorem:tb-transfinite-computation}\ref{item:tb-transfinite-equiv-rel} and~\ref{item:tb-transfinite-aut}.

\subsection{Homology localization of the torus bundle group}
\label{subsection:tb-localization}

We start by reviewing the computation of the homology localization
$\widehat\Gamma$ of the torus bundle group~$\Gamma$, from our earlier
work~\cite{Cha-Orr:2011-1}.  The result expresses $\widehat\Gamma$ as a colimit
of finitely presented groups.  (Such a colimit expression of the localization
exists for any finitely presented group by
Theorem~\ref{theorem:localization-basic-facts}\ref{item:localization-colimit},
but finding an explicit description is nontrivial in general.)

For a positive odd integer $\ell$, let
\begin{equation}
  \label{equation:presentation-Gamma(l)}
  \Gamma(\ell) =\langle u, v, t \mid tut^{-1}u,\, tvt^{-1}v,\,
  [u,v]^{\ell^2},\, [[u,v],u],\, [[u,v],v],\, [u,v],t] \rangle.
\end{equation}

It is straightforward to see that $\Gamma(1) = \Gamma$ and the map $\Gamma(\ell)
\to \Gamma(r\ell)$ sending $t$, $u$, and $v$ to $t$, $u^r$, and $v^r$
respectively is a well-defined inclusion for all odd $r$, $\ell\ge 1$.  The
groups $\Gamma(\ell)$ with these inclusions form a direct system.

\begin{theorem}[{\cite[Theorem~3.1]{Cha-Orr:2011-1}}]
  \label{theorem:cha-orr-colim}
  The homology localization of $\Gamma$ is given as
  \[
    \Gamma \to \widehat\Gamma = \colim_{\ell\text{ odd}} \Gamma(\ell).
  \]
\end{theorem}

Observe, from the presentation~\eqref{equation:presentation-Gamma(l)}, that
$[u,v]\in \Gamma(\ell)$ generates a finite cyclic subgroup of order $\ell^2$,
which is normal in~$\Gamma(\ell)$, and the quotient of $\Gamma(\ell)$ by this
cyclic subgroup is isomorphic to the semi direct product $\Z^2\rtimes \Z$, where
$\Z^2$ is generated by $u$, $v$ and $\Z$ is generated by~$t$ which acts on
$\Z^2$ by negation. Note that the restriction of $\Gamma(\ell) \to
\Gamma(r\ell)$ on the cyclic subgroup generated by $[u,v]$ is the homomorphism
$\Z_{\ell^2} \to \Z_{(r\ell)^2}$ given by $1\mapsto r^2$, and $\Gamma(\ell) \to
\Gamma(r\ell)$ induces a map $\Z^2\rtimes \Z \to \Z^2\rtimes \Z$ on the
quotients, which is given by $(a,b,c) \mapsto (ra,rb,c)$.  So, if we identify
$\Z_{\ell^2}$ with
\[
  \cyc{\ell^2} := (\tfrac{1}{\ell^2}\Z)/\Z
  = \{ \tfrac{a}{\ell^2} \mid a\in \Z \} / \Z,
\]
under $[u,v] \mapsto \frac{1}{\ell^2}$ and and identify $\Z^2\rtimes\Z$ with
$\frt{\ell}\rtimes\Z$ under $u\mapsto (\tfrac1\ell,0,0)$, $v\mapsto
(0,\tfrac1\ell,0)$, $t\mapsto (0,0,1)$, then we obtain the following commutative
diagram with exact rows.
\begin{equation}
  \label{equation:Gamma(l)-as-extension}
  \begin{tikzcd}[row sep=large]
    1 \ar[r] & \cyc{\ell^2} \ar[r]\ar[d,"\id",hook] & \Gamma(\ell) \ar[r]\ar[d]
    & \frt{\ell}\rtimes\Z \ar[r]\ar[d,"\id",hook] & 1
    \\
    1 \ar[r] & \cyc{(r\ell)^2} \ar[r] & \Gamma(r\ell) \ar[r]
    & \frt{r\ell}\rtimes\Z \ar[r] & 1
  \end{tikzcd}
\end{equation}
Taking colimit, we obtain the following central extension.
\begin{equation}
  \label{equation:Gamma-hat-as-extension}
  1\to \Z_{(2)}/\Z \to \widehat\Gamma \to \Z_{(2)}^2\rtimes\Z \to 1.
\end{equation}

Using this, we can compute the transfinite lower central subgroups
of~$\widehat\Gamma$.

\begin{lemma}
  \label{lemma:tb-transfinite-lcs}
  The first transfinite lower central subgroup $\widehat\Gamma_\omega$ is equal
  to the subgroup $\Z_{(2)}/\Z$.  For $\kappa\ge\omega+1$,
  $\widehat\Gamma_\kappa$ is trivial.
\end{lemma}

\begin{proof}
  We claim that $\Gamma(\ell)_\omega$ is the subgroup $\cyc{\ell^2}$.  To prove
  this, we will first verify $\Gamma(\ell)_\omega\subset \cyc{\ell^2}$. Recall
  from~\eqref{equation:lcs-Z^2-rtimes-Z} that the $r$th lower central subgroup
  $(\Z^2\rtimes\Z)_r$ of $\Gamma=\Z^2\rtimes\Z$ is equal to~$(2^{r-1}\Z)^2$. So
  $\Z^2\rtimes\Z$ is residually nilpotent.  That is, $(\Z^2\rtimes\Z)_\omega$ is
  trivial. From this and~\eqref{equation:Gamma(l)-as-extension}, it follows that
  $\Gamma(\ell)_\omega$ lies in the subgroup~$\cyc{\ell^2}$.  For the reverse
  inclusion, first verify that $u^{2^{r-1}} \in \Gamma(\ell)_r$ by induction,
  using the identity $[t,u^{2^{r-1}}] = u^{-2^r}$. So $[u^{2^{r-1}},v] =
  [u,v]^{2^{r-1}}$ lies in $\Gamma(\ell)_{r+1}$. Since $[u,v]$ has order
  $\ell^2$ and $\ell$ is odd, it implies that $[u,v]\in \Gamma(\ell)_{r+1}$.
  Since this holds for all $r$, it follows that $[u,v] \in \Gamma(\ell)_\omega$.
  In other words, $\cyc{\ell^2} \subset \Gamma(\ell)_\omega$.  This shows the
  claim that $\Gamma(\ell)_\omega=\cyc{\ell^2}$.

  From the claim, the promised conclusion $\widehat\Gamma_\omega=\Z_{(2)}/\Z$ is
  obtained by taking colimit.
  
  Since $[u,v]$ is central, $\Gamma(\ell)_{\omega+1}$ is trivial, and thus
  $\Gamma(\ell)_\kappa$ is trivial for all $\kappa\ge \omega+1$. Take colimit to
  obtain that $\widehat\Gamma_{\kappa}$ is trivial for all $\kappa\ge \omega+1$.
\end{proof}

\subsection{Third homology and realizable classes for $\kappa=\omega$}
\label{subsection:tb-omega-H_3}

The goal of this subsection is to investigate the homology and cap product structure of $\widehat\Gamma/\widehat\Gamma_\omega$ and prove Theorem~\ref{theorem:tb-transfinite-computation}\ref{item:tb-transfinite-H_3} and~\ref{item:tb-transfinite-realization} for $\kappa=\omega$.

We begin with homology computation for~$\widehat\Gamma/\widehat\Gamma_\omega$.
By~\eqref{equation:Gamma-hat-as-extension} and
Lemma~\ref{lemma:tb-transfinite-lcs}, we have
$\widehat\Gamma/\widehat\Gamma_\omega = \Z_{(2)}^2\rtimes \Z$ where $\Z$ acts on
$\Z_{(2)}^2$ by negation.
The Lyndon-Hochschild-Serre spectral
sequence for the HNN extension
\[
  1 \to \Z_{(2)}^2 \to \widehat\Gamma/\widehat\Gamma_\omega \to \Z \to 1
\]
gives the Wang exact sequence
\begin{multline}
  \label{equation:wang-sequence-omega}
  H_3(\Z_{(2)}^2)
  \to H_3(\widehat\Gamma/\widehat\Gamma_\omega) \to H_2 (\Z_{(2)}^2)
  \xrightarrow{1-t_*} H_2(\Z_{(2)}^2) \\
  \to H_2(\widehat\Gamma/\widehat\Gamma_\omega)
  \to H_1 (\Z_{(2)}^2) \xrightarrow{1-t_*} H_1(\Z_{(2)}^2)
\end{multline}
where $t_*\colon H_i (\Z_{(2)}^2) \to H_i (\Z_{(2)}^2)$ is the map induced by
negation $(a,b)\mapsto(-a,-b)$ on~$\Z_{(2)}^2$.  Using that $\Z_{(2)}^2$ is the
colimit of $(\tfrac{1}{d}\Z)^2 \cong \Z^2$, it is straightforward to compute
the following homology groups of~$\Z_{(2)}^2$:
\[
  H_3(\Z_{(2)}^2) = 0, \quad   H_2(\Z_{(2)}^2) = \Z_{(2)}, \quad
  H_1(\Z_{(2)}^2) = \Z_{(2)}^2.
\]
Moreover, $1-t_*$ on $H_2(\Z_{(2)}^2)$ is zero, while $1-t_*$ on
$H_1(\Z_{(2)}^2)$ is multiplication by~$2$.  From this
and~\eqref{equation:wang-sequence-omega}, it follows that 
\begin{equation}
  \label{equation:tb-omega-H_i}
  \begin{aligned}
    H_3(\widehat\Gamma/\widehat\Gamma_\omega) &= H_2(\Z_{(2)}^2) = \Z_{(2)}, \\
    H_2(\widehat\Gamma/\widehat\Gamma_\omega) &= H_2(\Z_{(2)}^2) = \Z_{(2)}, \\
    H_1(\widehat\Gamma/\widehat\Gamma_\omega) &= \Z_2^2 \times \Z.
  \end{aligned}
\end{equation}
This shows
Theorem~\ref{theorem:tb-transfinite-computation}\ref{item:tb-transfinite-H_3}
for $\kappa=\omega$.

Now we investigate cap products to compute the set of realizable
classes~$\cR_\omega(\Gamma)$.  First note that $\theta=1 \in \Z_{(2)} =
H_3(\widehat\Gamma/\widehat\Gamma_\omega)$ is the image of the fundamental class
$[Y]\in H_3(\Gamma)$, and thus it lies in~$\cR_\omega(\Gamma)$ by definition.
So, by Theorem~\ref{theorem:main-realization}, for $\theta=1$,
\begin{equation}
  \label{equation:tb-omega-cap-H^2-to-H_1}
  {}\cap \theta \colon tH^2(\widehat\Gamma/\widehat\Gamma_\omega) = \Z_2^2
  \to tH_1(\widehat\Gamma/\widehat\Gamma_\omega) = \Z_2^2
\end{equation}
is an isomorphism.  Also, 
\begin{equation}
  \label{equation:tb-omega-cap-H^1-to-H_2}
  {}\cap\theta\colon H^1(\widehat\Gamma/\widehat\Gamma_\omega) = \Z
  \to H_2(\widehat\Gamma/\widehat\Gamma_\omega) /
  \Ker\{H_2(\widehat\Gamma/\widehat\Gamma_\omega) \rightarrow
  H_2(\widehat\Gamma/\widehat\Gamma_k)\}.  
\end{equation}
is surjective for all finite~$k$.  That is, $\Im\{\cap\,1\} =
H_2(\widehat\Gamma/\widehat\Gamma_\omega)/{\Ker}$.

Consider an arbitrary $\theta := a/d \in \Z_{(2)} =
H_3(\widehat\Gamma/\widehat\Gamma_\omega)$ with $d$ odd.  Then, since the
codomain of~\eqref{equation:tb-omega-cap-H^2-to-H_1} is a finite abelian
2-group, the cap product ${}\cap\theta = ad\cdot (\cap\, 1)$
in~\eqref{equation:tb-omega-cap-H^2-to-H_1} is an isomorphism if and only if $a$
is odd.  Moreover, if $a$ is odd, then the cap product
in~\eqref{equation:tb-omega-cap-H^1-to-H_2} satisfies
\[
  \Im\{\cap\, a/d\} = (a/d)\cdot \Im\{\cap\, 1\} = (a/d)\cdot
  (H_2(\widehat\Gamma/\widehat\Gamma_\omega)/{\Ker}) =
  H_2(\widehat\Gamma/\widehat\Gamma_\omega)/{\Ker}
\]
where the last equality holds since $a/d$ is invertible in
$H_2(\widehat\Gamma/\widehat\Gamma_\omega)=\Z_{(2)}$.

So, by applying Theorem~\ref{theorem:main-realization}, $\theta=a/d\in
H_3(\widehat\Gamma/\widehat\Gamma_\omega)$ lies in $\cR_\omega(\Gamma)$ if and
only if $a$ is odd.  This proves
Theorem~\ref{theorem:tb-transfinite-computation}\ref{item:tb-transfinite-realization}
for $\kappa=\omega$.

\subsection{Third homology for $\kappa\ge\omega+1$}
\label{subsection:tb-omega+1-H_3}

The goal of this subsection is to investigate the homology of~$\widehat\Gamma$ and prove Theorem~\ref{theorem:tb-transfinite-computation}\ref{item:tb-transfinite-H_3} for $\kappa\ge\omega+1$.

Recall that $\widehat\Gamma/\widehat\Gamma_\kappa = \widehat\Gamma$ for
$\kappa\ge\omega+1$ by Lemma~\ref{lemma:tb-transfinite-lcs}, and that
$\widehat\Gamma = \colim \Gamma(\ell)$ by Theorem~\ref{theorem:cha-orr-colim},
where $\Gamma(\ell)$ is the group defined
by~\eqref{equation:presentation-Gamma(l)}.  We
restate~\eqref{equation:presentation-Gamma(l)} for the reader's convenience.
\begin{equation}
  \Gamma(\ell) =\langle u, v, t \mid tut^{-1}u,\, tvt^{-1}v,\,
  [u,v]^{\ell^2},\, [[u,v],u],\, [[u,v],v],\, [u,v],t] \rangle
  \tag{\ref{equation:presentation-Gamma(l)}}
\end{equation}

To understand the homology of $\widehat\Gamma$, it is useful to consider an HNN
extension described below.  Let $A(\ell)$ be the subgroup of $\Gamma(\ell)$ generated
by $u$ and~$v$, following~\cite{Cha-Orr:2011-1}.  From the
presentation~\eqref{equation:presentation-Gamma(l)}, it is immediately seen that
$A(\ell)$ is a normal subgroup of $\Gamma(\ell)$, and $\Gamma(\ell)/A(\ell)$ is the
infinite cyclic group generated by~$t$:
\begin{equation}
  \label{equation:Gamma(l)-as-HNN-extension}
  1 \to A(\ell) \to \Gamma(\ell) \to \Z \to 1
\end{equation}
Note that $\Gamma(\ell) \to \Gamma(r\ell)$ induces an isomorphism
$\Gamma(\ell)/A(\ell) \isomto \Gamma(r\ell)/A(r\ell)=\Z$ sending $t$ to~$t$.
Let $\cA=\colim A(\ell)$, and take the colimit
of~\eqref{equation:Gamma(l)-as-HNN-extension} to obtain the following:
\[
  1 \to \cA \to \widehat\Gamma \to \Z \to 1
\]
The Lyndon-Hochschild-Serre spectral sequence for this HNN extension gives the
following Wang exact sequence:
\begin{equation}
  \label{equation:wang-sequence-Gamma-hat}
  \cdots \to H_i(\cA) \xrightarrow{1-t_*} H_i(\cA) \to H_i(\widehat\Gamma)
  \to H_{i-1}(\cA) \xrightarrow{1-t_*} H_{i-1}(\cA) \to \cdots  
\end{equation}
where $t_*$ is induced by the conjugation by $t$ on~$\cA$.

To compute the homology of $\widehat\Gamma$
using~\eqref{equation:wang-sequence-Gamma-hat}, we first compute the homology
of~$\cA$. From~\eqref{equation:presentation-Gamma(l)}, it follows that $[u,v]\in
A(\ell)$ generates a finite cyclic normal subgroup of order~$\ell^2$, and
$A(\ell)$ is a central extension of this, by the free abelian group of rank two
generated by $u$ and~$v$. Since $u\mapsto u^r$, $v\mapsto v^r$ and $[u,v]\mapsto
[u^r,v^r] = [u,v]^{r^2}$ under $\Gamma(\ell) \to \Gamma(r\ell)$, we have the
following commutative diagram. 

\begin{equation}
  \label{equation:A(l)-as-extension}
  \begin{tikzcd}[row sep=large]
    1 \ar[r] & \Z_{\ell^2} \ar[r]\ar[d,"\cdot r^2"] & A(\ell) \ar[r]\ar[d]
    & \Z^2 \ar[r]\ar[d,"\cdot r"] & 1
    \\
    1 \ar[r] & \Z_{(r\ell)^2} \ar[r] & A(r\ell) \ar[r] & \Z^2 \ar[r] & 1
  \end{tikzcd}
\end{equation}
Consider the Lyndon-Hochschild-Serre spectral sequence of the top row
of~\eqref{equation:A(l)-as-extension}.
\begin{equation}
  \label{equation:spectral-sequence-A(k)}
  E^2_{p,q} = H_p(\Z^2) \otimes H_q(\Z_{\ell^2}) \Longrightarrow H_n(A(\ell))
\end{equation}
The $E^2$ and $E^\infty=E^3$ pages for $q\le 3$ are as follows:
\begin{equation}
  \label{equation:E^2-E^3-for-A(k)}
  E^2 = \, \begin{array}{|@{}l@{}}
    \begin{tikzcd}[sep=tiny,inner sep=-5pt]
      \Z_{\ell^2} & \Z_{\ell^2}^2 & \Z_{\ell^2} \mathstrut \\
      0 & 0 & 0 \\
      \Z_{\ell^2} & \Z_{\ell^2}^2 & \Z_{\ell^2} \\
      \Z & \Z^2 & \Z
      \ar[ull,"d^2_{2,0}"',outer sep=-2pt,near start,shorten >=-1ex,shorten <=-.5ex] \mathstrut 
    \end{tikzcd}
    \\ \hline
  \end{array}
  \quad,\quad
  E^\infty = E^3 = \, \begin{array}{|@{}l@{}}
    \begin{tikzcd}[sep=tiny]
      \Z_{\ell^2} & \Z_{\ell^2}^2 & \Z_{\ell^2} \mathstrut \\
      0 & 0 & 0 \\
      0 & \Z_{\ell^2}^2 & \Z_{\ell^2} \\
      \Z & \Z^2 & \ell^2\Z \mathstrut 
    \end{tikzcd}
    \\ \hline
  \end{array}\quad.
\end{equation}
All entries in~\eqref{equation:E^2-E^3-for-A(k)} are immediately obtained
from~\eqref{equation:spectral-sequence-A(k)}, possibly except $E^3_{p,q}$ for
$(p,q)=(2,0)$ and~$(1,0)$.  To verify these, observe that $E^\infty_{0,1} =
E^3_{0,1}$ mush vanish since $H_1(A(k))=\Z^2$ and $E^\infty_{1,0} = E^2_{1,0} =
\Z^2$. From this it follows that the differential $d^2_{2,0}$ is surjective, so
its kernel $E^3_{0,1}$ is the subgroup $\ell^2\Z$ of~$\Z$.

From the $E^\infty$ page, it follows that $H_2(A(\ell))$ is an extension of
$\Z_{\ell^2}^2=H_1(\Z^2)\otimes H_1(\Z_{\ell^2})$ by $\ell^2\Z \subset \Z =
H_2(\Z^2)$. From~\eqref{equation:A(l)-as-extension}, it follows that
$H_1(\Z^2)\to H_1(\Z^2)$ induced by $A(\ell) \to A(r\ell)$ is multiplication by
$r$, while $H_2(\Z^2)\to H_2(\Z^2)=\Z$ and $H_1(\Z_{\ell^2})\to
H_1(\Z_{(r\ell)^2})$ are multiplication by~$r^2$.
So~\eqref{equation:A(l)-as-extension} gives rise to the following diagram:
\begin{equation}
  \label{equation:exact-sequence-H_2(A(k))}
  \begin{tikzcd}[row sep=large]
    & \makebox[0mm][c]{$H_1(\Z^2)\otimes H_1(\Z_{\ell^2})$}\ar[d,equal] &
    & \Z \rlap{${}= H_2(\Z^2)$} \ar[d,"\cup" description,-,dash pattern=on 0cm off 10cm]
    \\[-4.5ex]
    0 \ar[r] & \Z_{\ell^2}^2 \ar[r]\ar[d,"\cdot r^3"] & H_2(A(\ell)) \ar[r]\ar[d]
    & \ell^2\Z \ar[r]\ar[d,"\cdot r^2","\cong"'] & 0
    \\
    0 \ar[r] & \Z_{(r\ell)^2}^2 \ar[r] & H_2(A(r\ell)) \ar[r] & (r\ell)^2\Z \ar[r] & 0
    \\[-4.5ex]
    & \makebox[0mm][c]{$H_1(\Z^2)\otimes H_1(\Z_{(r\ell)^2})$}\ar[u,equal] &
    & \Z \rlap{${}= H_2(\Z^2)$} \ar[u,"\cap" description,-,dash pattern=on 0cm off 10cm]
  \end{tikzcd}
\end{equation}
The top row of~\eqref{equation:exact-sequence-H_2(A(k))} for $\ell=1$ provides
an isomorphism $H_2(A(1)) = H_2(\Z^2) \isomto \Z$. The colimit map $\Z = 1^2\Z
\to \colim \ell^2\Z$ is an isomorphism, since the map $\cdot r^2$
in~\eqref{equation:exact-sequence-H_2(A(k))} is an isomorphism for all~$r$.  On
the other hand, since the map $\cdot r^3 \colon \Z_{\ell^2}^2 \to
\Z_{(r\ell)^2}^2$ is zero for all large $r$ divided by $\ell^2$,
$\colim \Z_{\ell^2}^2$ vanishes.  So, by taking the colimit
of~\eqref{equation:exact-sequence-H_2(A(k))}, we obtain an isomorphism
\begin{equation}
  \label{equation:tb-omega+1-H_2(A)}
  \Z = H_2(A(1)) \isomto \colim_{\ell\ge 1\text{ odd}} H_2(A(\ell)) = H_2(\cA).  
\end{equation}
Moreover, since $t$ on $A(1)=\Z^2$ is the negation, $t_*$ on $H_2(\cA)=\Z$ is the
identity.

To compute $H_3(\cA)$, first observe that the $E^\infty$ page
in~\eqref{equation:E^2-E^3-for-A(k)} tells us that $H_3(A(\ell))$ is an extension
of $H_3(\Z_{\ell^2})=\Z_{\ell^2}$ by $H_2(\Z^2)\otimes H_1(\Z_{\ell^2})=\Z_{\ell^2}$. Recall
the fact that $H_i(\Z_{\ell^2})=\Z_{\ell^2} \to H_i(\Z_{(r\ell)^2})=\Z_{(r\ell)^2}$ induced
by the injection $(\cdot r^2) \colon \Z_{\ell^2} \to \Z_{(r\ell)^2}$ is multiplication
by~$r^2$ for $i=1,3$.  Also, $(\cdot r)\colon \Z^2\to \Z^2$ induces
multiplication by $r^2$ on $H_2(\Z^2)$.  From this, it follows
that~\eqref{equation:A(l)-as-extension} gives rise to
\begin{equation}
  \label{equation:exact-sequence-H_3(A(k))}
  \begin{tikzcd}[row sep=large]
    & \makebox[0mm][c]{$H_3(\Z_{\ell^2})$} \ar[d,equal] & 
    & \makebox[0mm][c]{$H_2(\Z^2)\otimes H_1(\Z_{\ell^2})$} \ar[d,equal]
    \\[-4.5ex]
    0 \ar[r] & \Z_{\ell^2} \ar[r]\ar[d,"\cdot r^2"] & H_3(A(\ell)) \ar[r]\ar[d]
    & \Z_{\ell^2} \ar[r]\ar[d,"\cdot r^4"] & 0
    \\
    0 \ar[r] & \Z_{(r\ell)^2} \ar[r] & H_3(A(r\ell)) \ar[r] & \Z_{(r\ell)^2} \ar[r] & 0
    \\[-4.5ex]
    & \makebox[0mm][c]{$H_3(\Z_{(r\ell)^2})$} \ar[u,equal] & 
    & \makebox[0mm][c]{$H_2(\Z^2)\otimes H_1(\Z_{(r\ell)^2})$} \ar[u,equal]
  \end{tikzcd}
\end{equation}
Since the vertical map $\cdot r^4$ in~\eqref{equation:exact-sequence-H_3(A(k))}
is trivial if $r$ divided by $\ell$, the colimit of them is trivial.  So, by
taking the colimit of~\eqref{equation:exact-sequence-H_3(A(k))}, we obtain an
isomorphism
\begin{equation}
  \label{equation:tb-omega+1-H_3(A)}
  \Z_{(2)}/\Z = \colim_{\ell\ge 1\text{ odd}} \Z_{\ell^2}
  = \colim_{\ell\ge 1\text{ odd}} H_3(\Z_{\ell^2}) \isomto H_3(\cA).
\end{equation}
Note that the action of $t$ on $\Z_{\ell^2} \subset A(\ell)$ is trivial, since
$t[u,v]t^{-1}=[u^{-1},v^{-1}] = [u,v]$.  It follows that $t_*$ on
$H_3(\cA)=\Z_{(2)}/\Z$ is the identity.

Now, use~\eqref{equation:tb-omega+1-H_2(A)}, \eqref{equation:tb-omega+1-H_3(A)}
and the fact that $1-t_*=0$ on both $H_2(\cA)$ and $H_3(\cA)$, to extract the
following exact sequence from the Wang
sequence~\eqref{equation:wang-sequence-Gamma-hat}.
\begin{equation}
  \label{equation:H_3-Gamma-hat-as-extension}
  \begin{tikzcd}
    0\ar[r] & \Z_{(2)}/\Z \ar[r] & H_3(\widehat\Gamma) \ar[r] & \Z \ar[r] & 0
    \\[-2.5ex]
    & H_3(\cA) \ar[u,equal] & & H_2(\cA) \ar[u,equal]
  \end{tikzcd}
\end{equation}
So, $H_3(\widehat\Gamma)$ is isomorphic to $(\Z_{(2)}/\Z) \times \Z$. To provide
a fixed identification, we use a splitting described below. Recall that
$\Gamma(1)=\Gamma$, so $H_3(\Gamma(1)) = H_3(\Gamma)=H_3(Y)$ where $Y$ is the
torus bundle. Compare~\eqref{equation:H_3-Gamma-hat-as-extension} with the Wang
sequence associated to the exact
sequence~\eqref{equation:Gamma(l)-as-HNN-extension} for $\ell=1$, to obtain the
following commutative diagram:
\[
  \begin{tikzcd}[row sep=large]
    & 0 \ar[r] & H_3(Y) \ar[r,"\cong"]\ar[d] & H_2(A(1)) \ar[r]\ar[d,"\cong"] & 0 
    \\
    0 \ar[r] & H_3(\cA) \ar[r] & H_3(\widehat\Gamma) \ar[r] & H_2(\cA) \ar[r] & 0
  \end{tikzcd}
\]
From this, it is straightforward to see that the composition of the inverses of the two isomorphisms and $H_3(Y) \to H_3(\widehat\Gamma)$ is a splitting.  It gives an identification 
\begin{equation}
  \label{equation:tb-omega+1-H_3}
  H_3(\widehat\Gamma) = (\Z_{(2)}/\Z) \times \Z
\end{equation}
such that $(0,1)\in (\Z_{(2)}/\Z) \times \Z$ represents the image of the
fundamental class~$[Y]$.  This shows
Theorem~\ref{theorem:tb-transfinite-computation}\ref{item:tb-transfinite-H_3}
for $\kappa\ge\omega+1$.

For use in the next subsection, we compute~$H_2(\widehat\Gamma)$ here.  By
taking the abelianization, it is straightforward to see that $H_1(\cA) =
\Z_{(2)}^2$ and $t_*$ on $H_1(\cA)$ is the negation.  So, from the Wang
sequence~\eqref{equation:wang-sequence-Gamma-hat}, it follows that
\begin{equation}
  \label{equation:tb-omega+1-H_2}
  H_2(\widehat\Gamma)=H_2(\cA)=H_2(A(1))=H_2(\Z^2)=\Z.
\end{equation}

\subsection{Realizable classes for $\kappa\ge \omega+1$}
\label{subsection:tb-omega+1-realizable-classes}

In this subsection, we study the set of realizable classes $\cR_\kappa(\Gamma)$ for $\kappa\ge\omega+1$, to prove Theorem~\ref{theorem:tb-transfinite-computation}\ref{item:tb-transfinite-realization} for $\kappa\ge\omega+1$ and Theorem~\ref{theorem:tb-transfinite-computation}\ref{item:tb-transfinite-image-kappa+1}.

To determine realizable classes in $H_3(\widehat\Gamma/\widehat\Gamma_\kappa) =
H_3(\widehat\Gamma)$ for $\kappa\ge \omega+1$, consider the cap product
\begin{equation}
  \label{equation:tb-omega+1-cap-H^2-to-H_1}
  {}\cap\theta\colon tH^2(\widehat\Gamma) \to tH_1(\widehat\Gamma)
  = tH_1(\Gamma) = \Z_2^2.
\end{equation}
If $\theta\in H_3(\widehat\Gamma)=(\Z_{(2)}/\Z)\times\Z$ lies in the subgroup
$\Z_{(2)}/\Z$, then $k\theta = 0$ for some odd $k>0$.  Since
$tH_1(\widehat\Gamma)$ is a 2-group, it follows
that~\eqref{equation:tb-omega+1-cap-H^2-to-H_1} is zero for $\theta\in
\Z_{(2)}/\Z$. On the other hand, $\theta=(0,1) \in
H_3(\widehat\Gamma)=(\Z_{(2)}/\Z)\times\Z$ lies in $\cR_\kappa(\Gamma)$ since
$\theta$ is the image of the fundamental class~$[Y]$.   So
\eqref{equation:tb-omega+1-cap-H^2-to-H_1} is an isomorphism for $\theta=(0,1)$
by Theorem~\ref{theorem:main-realization}.  Since $tH_1(\widehat\Gamma)$ is a
finite abelian 2-group, it follows that
\eqref{equation:tb-omega+1-cap-H^2-to-H_1} is an isomorphism for $\theta=(0,r)$
if and only if $r\in \Z$ is odd.  Combining these observations, it follows that
\eqref{equation:tb-omega+1-cap-H^2-to-H_1} is an isomorphism for $\theta=(x,r)
\in H_3(\widehat\Gamma)=(\Z_{(2)}/\Z)\times\Z$ if and only if $r\in\Z$ is odd.

Now consider the cap product
\begin{equation}
  \label{equation:tb-omega+1-cap-H^1-to-H_2}
  {}\cap\theta\colon H^1(\widehat\Gamma) \to H_2(\widehat\Gamma)
  /\Ker\{H_2(\widehat\Gamma) \rightarrow
  H_2(\widehat\Gamma/\widehat\Gamma_\omega)\}.
\end{equation}
By~\eqref{equation:tb-omega+1-H_2}, $H_2(\widehat\Gamma)=H_2(\Z^2) = \Z$.
By~\eqref{equation:tb-omega-H_i}, $H_2(\widehat\Gamma/\widehat\Gamma_\omega) =
H_2(\Z_{(2)}^2) = \Z_{(2)}$.  Since $\Z^2\hookrightarrow \Z_{(2)}^2$ induces the
standard inclusion on these~$H_2$, $\Ker\{H_2(\widehat\Gamma) \rightarrow
H_2(\widehat\Gamma/\widehat\Gamma_\omega)\}$ is trivial, and thus the codomain
of~\eqref{equation:tb-omega+1-cap-H^1-to-H_2} is equal to
$H_2(\widehat\Gamma)=\Z$.  From this, it follows that the cap
product~\eqref{equation:tb-omega+1-cap-H^1-to-H_2} is zero for $\theta = (x,0)
\in H_3(\widehat\Gamma)=(\Z_{(2)}/\Z)\times\Z$,  since $x$ is torsion. By
Theorem~\ref{theorem:main-realization}, the cap
product~\eqref{equation:tb-omega+1-cap-H^1-to-H_2} is surjective for
$\theta=(0,1)$, since this $\theta$ lies in~$\cR_\kappa(\Gamma)$. So, for a
general class $\theta=(x,r) \in H_3(\widehat\Gamma)=(\Z_{(2)}/\Z)\times\Z$,
\eqref{equation:tb-omega+1-cap-H^1-to-H_2} is surjective if and only if
$r=\pm1$.  By Theorem~\ref{theorem:main-realization}, it is the case if and only
if $\theta\in \cR_\kappa(\Gamma)$. This proves
Theorem~\ref{theorem:tb-transfinite-computation}\ref{item:tb-transfinite-realization}
for $\kappa= \omega+1$.  For $\kappa>\omega+1$, the computation proceeds along
the same lines.  The only exception is that we need to replace
$\Ker\{H_2(\widehat\Gamma) \rightarrow
H_2(\widehat\Gamma/\widehat\Gamma_\omega)\}$
in~\eqref{equation:tb-omega+1-cap-H^1-to-H_2} by $\Ker\{H_2(\widehat\Gamma)
\rightarrow H_2(\widehat\Gamma/\widehat\Gamma_\lambda)\}$ with $\lambda<\kappa$.
But, since the kernel is already trivial for $\lambda=\omega$ by the above
computation, the same argument applies to the case of $\kappa>\omega+1$ as well.
This completes the proof of
Theorem~\ref{theorem:tb-transfinite-computation}\ref{item:tb-transfinite-realization}
for $\kappa\ge \omega+1$.

To compute the map $\cR_{\omega+1}(\Gamma) \to \cR_\omega(\Gamma)$ induced by
the projection $\widehat\Gamma = \widehat\Gamma/\widehat\Gamma_{\omega+1} \to
\widehat\Gamma/\widehat\Gamma_\omega$, recall
from~\eqref{equation:tb-omega+1-H_3} that
$H_3(\widehat\Gamma)=(\Z_{(2)}/\Z)\times\Z$ where the $\Z$ factor is identified
with~$H_2(\cA)=H_2(\Z^2)$ via the Wang
sequence~\eqref{equation:wang-sequence-Gamma-hat}.  Also, recall
from~\eqref{equation:tb-omega-H_i} that
$H_3(\widehat\Gamma/\widehat\Gamma_\omega)=H_2(\Z_{(2)})=\Z_{(2)}$.  So,
$H_3(\widehat\Gamma)=(\Z_{(2)})\times\Z \to
H_3(\widehat\Gamma/\widehat\Gamma_\omega)=\Z_{(2)}$ is given by $(a,r) \mapsto
r$, and $\cR_{\omega+1}(\Gamma) \to \cR_\omega(\Gamma)$ is the restriction. This
shows
Theorem~\ref{theorem:tb-transfinite-computation}\ref{item:tb-transfinite-image-kappa+1}
for $\kappa=\omega$.  Since $\widehat\Gamma/\widehat\Gamma_\kappa =
\widehat\Gamma$ for $\kappa\ge\omega+1$,
Theorem~\ref{theorem:tb-transfinite-computation}\ref{item:tb-transfinite-image-kappa+1}
for $\kappa\ge\omega+1$ is obviously true.

\subsection{Equivalence relation and automorphism action for $\kappa=\omega$}
\label{subsection:tb-omega-aut}

In this subsection, we investigate the equivalence relation ${\sim}$
on~$\cR_\omega(\Gamma)$ and prove Theorem~\ref{theorem:tb-transfinite-computation}\ref{item:tb-transfinite-equiv-rel} and~\ref{item:tb-transfinite-aut}.

Recall that $\cR_\omega(\Gamma) = \Z_{(2)}^\times \subset
H_3(\widehat\Gamma/\widehat\Gamma_\omega) = \Z_{(2)}$ by
Theorem~\ref{theorem:tb-transfinite-computation}\ref{item:tb-transfinite-H_3}
and~\ref{item:tb-transfinite-realization}. Fix $\theta=\frac pq\in
\cR_\omega(\Gamma)$, where $p$,~$q\in 2\Z+1$.

To determine the equivalence class of~$\theta$ as a subset of
$\cR_\omega(\Gamma)$, we will use an automorphism of
$\widehat\Gamma/\widehat\Gamma_\omega$, which is equal to $\Z_{(2)}^2\rtimes \Z$
by Lemma~\ref{lemma:tb-transfinite-lcs}.  Define $\phi_{p/q}\colon
\Z_{(2)}^2\rtimes \Z \to \Z_{(2)}^2\rtimes \Z$ by $\phi_{p/q}(a,b,r) = (\frac
pq\cdot a,b,r)$ for $a,b\in\Z_{(2)}$, $r\in\Z$.  It is straightforward to verify
that $\phi_{p/q}$ is an automorphism with inverse $\phi_{p/q}^{-1}=\phi_{q/p}$.
We claim that $\phi_{p/q}$ induces $1\mapsto \frac pq=\theta$ on
$H_3(\Z_{(2)}^2\rtimes\Z) = \Z_{(2)}$.  To see this, observe that the
restriction of $\phi_{p/q}$ on the subgroup $\Z_{(2)}^2$ induces an automorphism
of $H_2(\Z_{(2)}^2) = \Z_{(2)}$ given by $1\mapsto \frac pq$. Since
$H_3(\Z_{(2)}^2\rtimes\Z) = H_2(\Z_{(2)}^2)$ by~\eqref{equation:tb-omega-H_i},
the claim follows from this.

To avoid confusion, for a closed 3-manifold $M$ with $\pi=\pi_1(M)$ equipped
with an isomorphism $f\colon \widehat\pi/\widehat\pi_\omega \isomto
\widehat\Gamma/\widehat\Gamma_\omega$, denote the invariant $\theta_\omega(M)$
by $\theta_\omega(M,f)$ temporarily.  Then, for $M=Y$, the above claim implies
that $\theta_\omega(Y,\phi_{p/q}) = \theta$, since $1\in
\Z_{(2)}=H_3(\widehat\Gamma/\widehat\Gamma_\omega)$ represents the image of the
fundamental class~$[Y]$.   So, by definition, the equivalence class $I_\theta =
\{\theta'\mid \theta'\sim \theta\}$ of $\theta$ in $\cR_\omega(\Gamma)$ is equal
to the image of the following composition.
\[
  \begin{tikzcd}[sep=large]
    \cR_{\omega+1}(\Gamma) \ar[r]
    & \cR_\omega(\Gamma) \ar[r,"(\phi_{p/q})_*"']
    & \cR_\omega(\Gamma)
    \\[-4.5ex]
    |[overlay]| (\Z_{(2)}/\Z)\times\{\pm1\} \ar[u,equal]
    & \Z_{(2)}^\times \ar[u,equal]
  \end{tikzcd}
\]
Here, the projection-induced map $\cR_{\omega+1}(\Gamma) \to \cR_\omega(\Gamma)$
is $(a,\pm1)\mapsto \pm1$ by
Theorem~\ref{theorem:tb-transfinite-computation}\ref{item:tb-transfinite-image-kappa+1}.
From this, it follows that $I_\theta=\{\theta,-\theta\}$.  This completes the
proof of
Theorem~\ref{theorem:tb-transfinite-computation}\ref{item:tb-transfinite-equiv-rel}.

In addition, using the above argument, it is straightforward to show
Theorem~\ref{theorem:tb-transfinite-computation}\ref{item:tb-transfinite-aut},
which asserts that the action of $\Aut(\widehat\Gamma/\widehat\Gamma_\omega)$ on
the set of realizable classes $\cR_\omega(\Gamma)$ is transitive. Indeed, for an
arbitrary $\theta=p/q \in \Z_{(2)}^\times =\cR_\omega(\Gamma)$, since the
above automorphism $\phi_{p/q}$ on $\widehat\Gamma/\widehat\Gamma_\omega$
satisfies $\phi_{p/q}(1)=\theta$, it follows that $\theta$ and $1$ have the same
orbit. So the action is transitive.

\section{Torus bundle example: the universal \texorpdfstring{$\theta$}{theta}-invariant}
\label{section:tb-computation-final}

We continue the study of the localization of the fundamental group $\Gamma$ of the torus bundle $Y$ defined in~\eqref{equation:torus-bundle}.
The goal of this section is to understand the final invariant $\widehat\theta$ defined over~$\widehat\Gamma$ and prove Theorem~\ref{theorem:main-tb-final}, which we state again below for the reader's convenience.
Recall that $\widehat\cR(\Gamma)$ is the set of realizable values of~$\widehat\theta$.
Theorem~\ref{theorem:main-tb-final} says:
\emph{
  $\widehat\cR(\Gamma)/\Aut(\widehat\Gamma)$ is infinite.
  This detects the existence of infinitely many distinct homology cobordism classes of closed 3-manifolds $M$ with $\pi=\pi_1(M)$, such that $\widehat\pi \cong \widehat\Gamma$, and thus, $\theta_\kappa(M)$ is defined and vanishes in $\Coker\{\cR_{\kappa+1}(\Gamma) \to \cR_\kappa(\Gamma)\}$ for all ordinals~$\kappa$.
  In particular, for every ordinal $\kappa$, the Milnor invariant $\bar\mu_\kappa(M)$ vanishes for these 3-manifolds~$M$.
}

We begin with computation of realizable classes in~$H_3(\widehat\Gamma)$.
Recall that $H_3(\widehat\Gamma)=(\Z_{(2)}/\Z)\times \Z$ by Theorem~\ref{theorem:tb-transfinite-computation}\ref{item:tb-transfinite-H_3}.

\begin{theorem}
  \label{theorem:tb-final-realization}
  $\widehat\cR(\Gamma) = (\Z_{(2)}/\Z) \times \{\pm1\} \subset
  H_3(\widehat\Gamma)$.
\end{theorem}

\begin{proof}
  In the argument used to prove
  Theorem~\ref{theorem:tb-transfinite-computation}\ref{item:tb-transfinite-H_3}
  in Section~\ref{subsection:tb-omega+1-realizable-classes}, we have shown that
  a homology class $\theta\in H_3(\widehat\Gamma)=(\Z_{(2)}/\Z)\times \Z$ lies
  in $(\Z_{(2)}/\Z) \times \{\pm1\}$ if and only if $\cap\,\theta\colon
  tH^2(\widehat\Gamma) \to tH_1(\widehat\Gamma)$ is an isomorphism and
  $\cap\,\theta\colon H^1(\widehat\Gamma) \to H_2(\widehat\Gamma)$ is an
  epimorphism.  By Theorem~\ref{theorem:main-realization-final-inv}, it follows
  that $\widehat\cR(\Gamma) = (\Z_{(2)}/\Z) \times \{\pm1\}$.
\end{proof}

The next theorem describes the action of $\Aut(\widehat\Gamma)$ on
$H_3(\widehat\Gamma)$ and~$\widehat\cR(\Gamma)$.   To state the result, we use
the following notation.
For a group $G$, denote the abelianization by~$G_{ab}$.
Recall from Section~\ref{subsection:tb-omega+1-H_3} that
$\widehat\Gamma$ is an HNN extension of a subgroup~$\cA$ such that
$\cA_{ab}=\Z_{(2)}^2$ with basis~$\{u,v\}$.  We will show, in
Lemma~\ref{lemma:tb-aut-general-form}, that if $f\colon \widehat\Gamma \to
\widehat\Gamma$ an automorphism, then $f$ induces an automorphism $f_\cA\in
\operatorname{GL}(2,\Z_{(2)})$ on $\cA_{ab}=\Z_{(2)}^2$ satisfying $\det f_\cA =
\pm1$, and $f$ induces an automorphism $f_\Z$ on the quotient
$\widehat\Gamma/\cA = \Z$. Define 
\[
  \delta_f := \det f_\cA  \in \{1,-1\}, \quad
  \epsilon_f := f_\Z(1) \in \{1,-1\}.
\]
One readily sees that $\Aut(\widehat\Gamma) \to \{-1,1\}^2\cong \Z_2^2$ given by
$f\mapsto (\delta_f,\epsilon_f)$ is a surjective group homomorphism onto the
Klein 4-group.  Indeed, for a given pair $(\delta,\epsilon) \in \{1,-1\}^2$, the
automorphism $\Gamma\to \Gamma$ defined by $u\mapsto u^{\delta}$, $v \mapsto v$,
$t \mapsto t^\epsilon$ gives rise to an automorphism $f\colon
\widehat\Gamma \to \widehat\Gamma$ satisfying $(\delta_f,\epsilon_f) =
(\delta,\epsilon)$.

\begin{theorem}
  \label{theorem:tb-final-aut}
  Suppose $f$ is an automorphism on~$\widehat\Gamma$.  Then the induced
  automorphism $f_*$ on $H_3(\widehat\Gamma)=(\Z_{(2)}/\Z) \times \Z$ is given
  by $f_*(a,n) = (\delta_f\cdot a, \delta_f\cdot \epsilon_f\cdot n)$.
  Consequently, there are bijections
  \begin{gather*}
    H_3(\widehat\Gamma)/\Aut(\widehat\Gamma)
    \approx \{(a,n) \in \Z_{(2)} \times \Z \mid 0\le a < 1/2,\, n\ge 0\},
    \\
    \widehat\cR(\Gamma)/\Aut(\widehat\Gamma) 
    \approx \{a\in \Z_{(2)} \mid 0\le a < 1/2\}.
  \end{gather*}
\end{theorem}

The first statement says that the natural map $\Aut(\widehat\Gamma) \to
\Aut(H_3(\widehat\Gamma))$ factors through the Klein 4-group $\{1,-1\}^2$, via
$f\mapsto (\delta_f, \epsilon_f)$.  The two bijections in
Theorem~\ref{theorem:tb-final-aut} are immediately obtained from the first
statement.

The first sentence of Theorem~\ref{theorem:main-tb-final}, which asserts that $\widehat\cR(\Gamma)/\Aut(\widehat\Gamma)$ is infinite, is an immediate consequence of Theorem~\ref{theorem:tb-final-aut}.
Also, from this, the second statement of Theorem~\ref{theorem:main-tb-final} follows immediately by Theorems~\ref{theorem:homology-cob-final-inv}\@.

The remaining part of this section is devoted to the proof of the
first statement of Theorem~\ref{theorem:tb-final-aut}.

Recall that $\Gamma(\ell)$ is the subgroup of $\widehat\Gamma$ given
by~\eqref{equation:presentation-Gamma(l)}, $\Gamma(1)=\Gamma$, and
$\widehat\Gamma$ is the colimit of~$\Gamma(\ell)$.  If $f\colon
\widehat\Gamma\to\widehat\Gamma$ is an automorphism, then for each odd $\ell\ge
1$, $f(\Gamma(\ell))\subset \Gamma(r\ell)$ for some odd~$r\ge 1$, since
$\Gamma(\ell)$ is finitely generated. The restriction $f\colon \Gamma(\ell)\to
\Gamma(r\ell)$ induces isomorphisms on $H_1$ and $H_2$, since so does the
colimit inclusion $\Gamma(\ell) \to \widehat\Gamma$ for every~$\ell$.  So
$f\colon \Gamma(\ell)\to \Gamma(r\ell)$ is 2-connected. Conversely, if $f\colon
\Gamma(\ell)\to \Gamma(r\ell)$ is 2-connected, then it induces an automorphism
$f\colon \widehat\Gamma = \widehat{\Gamma(\ell)} \isomto
\widehat{\Gamma(r\ell)} = \widehat\Gamma$, by
Theorem~\ref{theorem:localization-basic-facts}\ref{item:localization-2-connected-map}.

This leads us to investigate 2-connected homomorphisms $f\colon
\Gamma(\ell)\to \Gamma(r\ell)$.  We begin with a characterization.  Recall from
the presentation~\eqref{equation:presentation-Gamma(l)} that $\Gamma(\ell)$ has
generators $u$, $v$ and~$t$.  Let $A(\ell)$ be the subgroup generated by $u$
and~$v$, as done in Section~\ref{subsection:tb-omega+1-H_3}.

\begin{lemma}
  \label{lemma:tb-aut-general-form}
  A homomorphism $f\colon \Gamma(\ell) \to \Gamma(r\ell)$ is 2-connected if and
  only if $f$ is given by
  \begin{equation}
    \label{equation:tb-aut-general-form}
    \begin{aligned}
      f(t) &= t^\epsilon u^p v^q [u,v]^j \\
      f(u) &= u^a v^b [u,v]^m \\
      f(v) &= u^c v^d [u,v]^n
    \end{aligned}
  \end{equation}
  where $\epsilon,a,b,c,d,j,m,n$ are integers satisfying
  \begin{equation}
    \label{equation:tb-aut-exponent-condition}
    \begin{gathered}
      \epsilon =\pm 1, \quad ad-bc = \pm r^2,
      \\
      2m \equiv aq-bp+ab, \quad
      2n \equiv cq-dp+cd \mod{(r\ell)^2}.
    \end{gathered}
  \end{equation}
\end{lemma}

Often we will abuse the notation to denote by $f$ the automorphism of
$\widehat\Gamma$ induced by a 2-connected map $f\colon \Gamma(\ell)\to
\Gamma(r\ell)$. Note that if $f$ is given
by~\eqref{equation:tb-aut-general-form}, then it induces automorphisms
$\frac1r\sbmatrix{a & c \\ b & d}$ on $H_1(\cA)=\Z_{(2)}^2$ and $1\mapsto
\epsilon$ on $\Z=\widehat\Gamma/\cA$.  So, we have
\begin{equation}
  \label{equation:general-form-epsilons}
  \delta_f = \frac{ad-bc}{r^2},\quad \epsilon_f = \epsilon.
\end{equation}

\begin{proof}[Proof of Lemma~\ref{lemma:tb-aut-general-form}] Observe that any
  $g\in \Gamma(\ell)$ can be written as $g = t^\epsilon u^p v^q [u,v]^j$, by
  using the defining relations in~\eqref{equation:presentation-Gamma(l)}.  Also,
  $t^\epsilon u^p v^q [u,v]^j$ lies in the subgroup $A(\ell)$ if and only if
  $\epsilon=0$.  We claim that $f$ sends $A(\ell)$ to $A(r\ell)$.  From the
  claim, it follows that $f(t)$, $f(u)$ and $f(v)$ are of the form
  of~\eqref{equation:tb-aut-general-form} for some exponents (without
  enforcing~\eqref{equation:tb-aut-exponent-condition} for now). 

  To show the claim, consider the \emph{first rational derived subgroup} of a
  group $G$, which is defined to be the kernel of the natural map $G\to
  H_1(G)\otimes \Q$. That is, it is the minimal normal subgroup of $G$ such that
  the quotient is abelian and torsion free.  It is straightforward to see that
  the first rational derived subgroup is characteristic. In our case, for
  $\Gamma(\ell)$, the first rational derived subgroup is equal to $A(\ell)$.  So
  $f(A(\ell))\subset A(r\ell)$, as claimed.

  Next, we claim that a map of the free group on $t$, $u$ and $v$ to
  $\Gamma(r\ell)$ given by~\eqref{equation:tb-aut-general-form} kills relations
  of $\Gamma(\ell)$ if and only if $ad-bc \equiv 0 \bmod{r^2}$ and
  \[
    2m \equiv aq-bp+ab,\quad 2n \equiv cq-dp+cd \mod{(r\ell)^2}.
  \]
  The claim is shown by a routine computation. The map sends the relation
  $tut^{-1}u$ to
  \[
  t^\epsilon u^p v^q [u,v]^j \cdot u^a v^b [u,v]^m \cdot
  (t^\epsilon u^p v^q [u,v]^j)^{-1} \cdot u^p v^q [u,v]^\ell =
  [u,v]^{2m-aq+bp+ab}
  \]
  which is trivial in $\Gamma(r\ell)$ if and only if $(r\ell)^2$ divides
  $2m-aq+bp+ab$. Similarly $(r\ell)^2$ divides $2n-cq+dp+cd$ if and only if
  $tvt^{-1}v$ is sent to the identity.  Also the relation $[u,v]^{\ell^2}$ of
  $\Gamma(\ell)$ is sent to $[u^av^b,u^cv^d]^{\ell^2} = [u,v]^{\ell^2(ad-bc)}$,
  which is trivial if and only if $r^2$ divides $ad-bc$, since $[u,v]$ has order
  $(r\ell)^2$ in~$\Gamma(r\ell)$. This proves the claim.

  Recall that $H_1(\Gamma(\ell)) = (\Z_2)^2 \times \Z$ where the factors are
  generated by $u$, $v$, and $t$ respectively. So, when $f$ is the homomorphism
  given by~\eqref{equation:tb-aut-general-form}, $f_*\colon H_1(\Gamma(\ell))\to
  H_1(\Gamma(r\ell))$ is represented by
  \[
  \begin{bmatrix}
    a & c & p \\
    b & d & q \\
    0 & 0 & \epsilon
  \end{bmatrix}.
  \]
  Therefore, $f$ induces an isomorphism on $H_1$ if and only if
  $\epsilon = \pm1$ and $ad-bc$ is odd.

  To investigate the induced map on $H_2$, first note that
  $A(\ell)_{ab}$ is equal to $\Z^2$ generated by $u$ and~$v$. We will use the
  fact that $H_2(\Gamma(\ell))$ can be identified with the subgroup $\ell^2 \Z
  \subset \Z = H_2(A(\ell)_{ab})$.  This can be proven by investigating the Wang
  sequence for the HNN extension~\eqref{equation:Gamma(l)-as-HNN-extension}.  An
  alternative proof is as follows.  Recall that
  $H_2(\Gamma)=H_2(\widehat\Gamma)=\Z$ by~\eqref{equation:tb-omega+1-H_2}.
  Since $\Gamma=\Gamma(1)\to \Gamma(\ell)$ and $\Gamma(\ell) \to \widehat\Gamma$ are
  2-connected, it follows that $H_2(\Gamma(\ell))$ is equal to $H_2(\Gamma(1))$, which is
  equal to $H_2(A(1))=H_2(\Z^2)$ by~\eqref{equation:tb-omega+1-H_2}.  Note that
  $\Z^2=A(1) \to A(\ell)_{ab}=\Z^2$ is scalar multiplication by~$\ell$.  So,
  $H_2(\Gamma(\ell))$ is the subgroup $\ell^2 \Z \subset \Z = H_2(A(\ell)_{ab})$.
  
  Now, observe that $H_2(A(\ell)_{ab})=\Z \to H_2(A(r\ell)_{ab})=\Z$ induced by
  $f$ given by~\eqref{equation:tb-aut-general-form} is equal to multiplication
  by $ad-bc$. From this, it follows that $f$ induces an epimorphism
  $H_2(\Gamma(\ell)) \to H_2(\Gamma(r\ell))$ if and only if $ad-bc=\pm r^2$.
\end{proof}

Using Lemma~\ref{lemma:tb-aut-general-form}, we will investigate the action of
$\Aut(\widehat\Gamma)$ on the torsion part of $H_3(\widehat\Gamma)$.

\begin{lemma}
  \label{lemma:tb-aut-on-tH_3}
  Suppose $f\colon \widehat\Gamma \to \widehat\Gamma$ is an automorphism.  Then
  the induced automorphism $f_*$ on the torsion subgroup
  $tH_3(\widehat\Gamma)=\Z_{(2)}/\Z$ is multiplication by $\delta_f\in \{\pm1\}$.
\end{lemma}

\begin{proof}
  Fix an arbitrary odd~$\ell\ge 1$.  By
  Lemma~\ref{equation:tb-aut-general-form}, the given automorphism $f$ on
  $\widehat\Gamma$ restricts to a 2-connected homomorphism $f|_{\Gamma(\ell)} \colon
  \Gamma(\ell) \to \Gamma(r\ell)$ for some odd~$r\ge 0$, and $f|_{\Gamma(\ell)}$ is of the form
  of~\eqref{equation:tb-aut-general-form}.  We have
  \[
    f|_{\Gamma(\ell)}([u,v]) = [u^a v^b,u^c v^d] = [u,v]^{ad-bc} =
    [u,v]^{\delta_f\cdot r^2}.
  \]
  Recall from Section~\ref{subsection:tb-omega+1-H_3} that $[u,v]\in \Gamma(\ell)$
  generates a subgroup that we identified with~$\Z_{\ell^2}$. So, $f|_{\Gamma(\ell)}$
  induces $\Z_{\ell^2} \to \Z_{(r\ell)^2}$ given by $1\mapsto \delta_f\cdot
  r^2$. It induces the inclusion $H_3(\Z_{\ell^2}) = \Z_{\ell^2} \to
  H_3(\Z_{(r\ell)^2}) = \Z_{(r\ell)^2}$ given by $1\mapsto \delta_f\cdot
  r^2$. This is the map $\frac1{\ell^2} \mapsto
  \delta_f\cdot\frac1{\ell^2}$, when $H_3(\Z_{\ell^2}) = \Z_{\ell^2}$ and
  $H_3(\Z_{(r\ell)^2}) = \Z_{(r\ell)^2}$ are regarded as subgroups of
  $\Z_{(2)}/\Z$ using~\eqref{equation:tb-omega+1-H_3(A)}.
  By~\eqref{equation:H_3-Gamma-hat-as-extension}
  and~\eqref{equation:tb-omega+1-H_3(A)}, it follows that the induced map
  $f_*\colon tH_3(\widehat\Gamma) \to tH_3(\widehat\Gamma)=\Z_{(2)}/\Z$ is
  multiplication by~$\delta_f$.
\end{proof}

By~\eqref{equation:tb-omega+1-H_3}, the $\Z$ factor of
$H_3(\widehat\Gamma)=(\Z_{(2)}/\Z)\times\Z$ is generated by the image of the
fundamental class $[Y] \in H_3(Y)=H_3(\Gamma)$.  The rest of this section is
devoted to understand the action of $\Aut(\widehat\Gamma)$ on this generator.
Since every automorphism of $\widehat\Gamma$ is induced by a 2-connected map
$f\colon \Gamma=\Gamma(1)\to \Gamma(r)$, it suffices to investigate $f_*[Y]\in
H_3(\Gamma(r))$.

Our strategy is to simplify $f$ given in Lemma~\ref{lemma:tb-aut-general-form}
without altering~$f_*[Y]$.  We begin with elimination of the $[u,v]^j$
factor in $f(t)$ in the general form~\eqref{equation:tb-aut-general-form}.

\begin{lemma}
  \label{lemma:tb-aut-eliminate-j}
  Let $f\colon \Gamma(1)\to \Gamma(r)$ be a 2-connected map given
  by~\eqref{equation:tb-aut-general-form}.  Let $f'\colon \Gamma(1)\to \Gamma(r)$ be the
  map
  \begin{equation}
    \label{equation:f-with-j-eliminated}
    f'(t) = t^\epsilon u^p v^q, \quad
    f'(u) = u^a v^b [u,v]^m, \quad
    f'(v) = u^c v^d [u,v]^n.
  \end{equation}
  Then $f$ is 2-connected, and $f$ and $f'$ induce the same homomorphism $f_* =
  f'_*\colon H_3(\Gamma(1))\to H_3(\Gamma(r))$.
\end{lemma}

\begin{proof}
  By Lemma~\ref{lemma:tb-aut-general-form}, the assignment
  \eqref{equation:f-with-j-eliminated} gives a well-defined 2-connected
  homomorphism, since the conditions
  in~\eqref{equation:tb-aut-exponent-condition} do not involve the exponent~$j$.

  Recall that $B\Gamma(1)=Y=T^2\times[0,1]/(h(z),0)\sim (z,1)$ where $h\colon
  T^2\to T^2=S^1\times S^1$ is the monodromy
  $h(\zeta,\xi)=(\zeta^{-1},\xi^{-1})$. Here $S^1$ is regarded as the unit
  circle in~$\C$.  Use $(1,1,0)$ as a basepoint of~$B\Gamma(1)$.  Choose maps
  $B\Gamma(1)\to B\Gamma(r)$ realizing $f$ and $f'$, and denote them by $f$ and
  $f'$, abusing the notation.

  Let $T^3=T^2\times S^1$, and use $(1,1,1)\in T^3$ as a basepoint.  Denote by
  $x$, $y$ the standard basis of $\pi_1(T^2)=\Z^2$, and denote by $s$ the
  generator of $\pi_1(S^1)=\Z$, so that $x$, $y$ and $s$ form a basis
  of~$\pi_1(T^3)$.   The element $[u,v]\in \Gamma(r)$ is central
  by~\eqref{equation:presentation-Gamma(l)}, and $f(u)$, $f(v)\in \Gamma(r)$
  commute since $u$, $v\in \Gamma(1)$ commute.  It follows that there is a map
  $g\colon T^3\to B\Gamma(r)$ which induces $\pi_1(T^3)=\Z^3 \to \Gamma(r)$
  given by $x\mapsto f(u)^{-1}$, $y\mapsto f(u)^{-1}$ and $s\mapsto [u,v]^j$. By
  homotopy if necessary, we may assume $g|_{T^2\times 1}$ is equal to
  $f'|_{T^2\times 1}$, since $f=f'$ on $u$ and $v$.  Define $F\colon
  B\Gamma(1)\to B\Gamma(r)$ to be the composition
  \begin{align*}
    F\colon B\Gamma(1) & \overset{}{=\joinrel=} T^2\times[0,1]/(h(z),0)\sim (z,1)
    \\[1ex]
    & \xrightarrow[q]{}
    \big( T^2\times[0,\tfrac12]/(h(z),0)\sim(z,\tfrac12) \big)
      \cupover{\vtop to0mm{\hbox{$\scriptstyle T^2\times\tfrac12$}\vss}}
      \big(T^2\times[\tfrac12,1]/(z,\tfrac12)\sim(z,1) \big)
    \\
    & \overset{}{=\joinrel=} Y\cupover{T^2} T^3
      \xrightarrow{f'\cup g} B\Gamma(r)
  \end{align*}
  where $q$ is the quotient map induced by $(z,t)\mapsto (z,t)$.  Observe that
  the induced map $F\colon \Gamma(1) \to \Gamma(r)$ satisfies $F(u)=f'(u)=f(u)$,
  $F(v)=f'(v)=f(v)$, and $F(t)=f'(t)g(s)=f(t)$.  It follows that $f$ and $F$ are
  homotopic. Therefore, on $H_3$, we have
  \[
  f_*[Y] = F_*[Y] = f'_*[Y]+g_*[Y]\in
  H_3(\Gamma(r)).
  \]

  So it suffices to prove that $g_*\colon H_3(T^3)\to H_3(\Gamma(r))$ is zero.
  To show this, first observe that $g\colon \pi_1(T^3)\to \Gamma(r)$ sends
  $\pi_1(T^3)=\Z^3$ to the subgroup~$A(r)$.  In addition, it induces a morphism
  of central extensions:
  \begin{equation}
    \label{equation:tb-T^3-diagram-g}
    \begin{tikzcd}[row sep=large,ampersand replacement=\&]
      0 \ar[r] \& \Z \ar[r]\ar[d,"\cdot j"] \&
      \Z^3 \ar[r]\ar[d,"g"] \& \Z^2 \ar[r]\ar[d,"\sbmatrix{a & c \\ b & d}"]
      \& 0
      \\
      0 \ar[r] \& \Z_{r^2} \ar[r] \& A(r) \ar[r] \& \Z^2 \ar[r] \& 0
    \end{tikzcd}
  \end{equation}
  Here the top row corresponds to the trivial fibration $S^1 \hookrightarrow T^3
  \to T^2$, and the bottom row is the exact sequence
  in~\eqref{equation:A(l)-as-extension}.  The leftmost and rightmost vertical
  maps are multiplication by $j$ and $\sbmatrix{a & c \\ b & d}$, by the
  definition of $g$ and description \eqref{equation:tb-aut-general-form}
  of~$f$.  The map $g$ induces a morphism of the spectral sequences. In
  particular, on $E^2_{2,1}$, $g$ induces a map
  \begin{equation}
    \label{equation:tb-T^3-E^2_2,1}
    \Z = H_2(\Z^2)\otimes H_1(\Z) \to H_2(\Z^2) \otimes
    H_1(\Z_{r^2}) = \Z_{r^2}.
  \end{equation}
  This is scalar multiplication by $(ad-bc)j$, by the above descriptions of the
  vertical maps in \eqref{equation:tb-T^3-diagram-g}. By
  Lemma~\ref{lemma:tb-aut-general-form}, $ad-bc=\pm r^2$.  It follows that
  \eqref{equation:tb-T^3-E^2_2,1} is a zero map.  Since $E^2_{2,1}$ for $\Z^3$
  is equal to $H_3(T^3)$, it follows that $g_*\colon H_3(T^3)\to H_3(A(r))$ is
  zero.
\end{proof}

The next step of our reduction is described by the following lemma.

\begin{lemma}
  \label{lemma:tb-aut-clean-form}
  Suppose $f\colon \Gamma=\Gamma(1)\to \Gamma(r)$ is a $2$-connected
  homomorphism. Then there is a $2$-connected homomorphism $f'\colon \Gamma\to
  \Gamma(r)$ such that $f'(t)=t^{\epsilon_f}$,  $\delta_f = \delta_{f'}$ and
  $f_* = f'_*$ on~$H_3(\Gamma)$.
\end{lemma} 

\begin{proof}
  Let $\epsilon=\epsilon_f$, and apply
  Lemmas~\ref{lemma:tb-aut-general-form} and~\ref{lemma:tb-aut-eliminate-j} to
  assume
  \begin{equation}
    \label{equation:before-cleaned}
    f(t)=t^\epsilon u^p v^q,\quad f(u)=u^av^b[u,v]^m,\quad f(v)=u^cv^d[u,v]^n.
  \end{equation}
  
  We claim that we may assume that both $p$ and $q$ are even
  in~\eqref{equation:before-cleaned}.  To show this, consider $\phi\colon
  \Gamma\to \Gamma$ given by $\phi(t)=tu$, $\phi(u)=u$, $\phi(v)=v$.  It is a
  well-defined 2-connected homomorphism by
  Lemma~\ref{lemma:tb-aut-general-form}.  Moreover, it induces the identity
  on~$H_3(\Gamma)=\Z$.  This can be seen geometrically, by inspecting the
  fundamental class $[Y]$ under an appropriate map $xB\Gamma = Y\to Y$
  realizing~$\phi$.  Alternately, use the Wang sequence for the extension $1\to
  A(1)\to \Gamma\to \Z$ to identify $H_3(\Gamma)$ with $H_2(A(1))=H_2(\Z^2)=\Z$,
  and use that $\phi|_{A(1)}$ is the identity. Now, since $\phi_*=\id$ on $H_3$,
  it follows that $f_* = (f\circ\phi)_*$ on~$H_3(\Gamma)$.  Similarly, define a
  2-connected homomorphism $\phi'\colon \Gamma\to \Gamma$ by $\phi'(t)=tv$,
  $\phi'(u)=u$, $\phi'(v)=v$. Then $f_* = (f\circ \phi')_*$ on $H_3$, too.  We
  have that
  \begin{align*}
    (f\circ\phi)(t) &= f(tu) = t^\epsilon u^p v^q \cdot u^a v^b [u,v]^m =
    t^\epsilon u^{p+a} v^{q+b} [u,v]^{m-aq},
    \\
    (f\circ\phi')(t) &= f(tv) = t^\epsilon u^p v^q \cdot u^c v^d [u,v]^m =
    t^\epsilon u^{p+c} v^{q+d} [u,v]^{m-cq}.
  \end{align*}
  By~\eqref{equation:tb-aut-exponent-condition}, $ad-bc$ is odd.  We assume $a$
  and $d$ are odd, and $b$ is even, since arguments for other cases are
  identical. Then, composition with $\phi$ alters the parity of $p$ and
  preserves the parity of $q$, and composition with $\phi'$ alters the parity of
  $q$ (while the parity of $p$ is left uncontrolled).  So, by composition, we
  may assume that both $p$ and $q$ are even.  Note that $a$, $b$, $c$, $d$ and
  $\epsilon$ are left unchanged.  Finally apply
  Lemma~\ref{lemma:tb-aut-eliminate-j} to obtain the form
  of~\eqref{equation:before-cleaned}.  This proves of the claim. 

  Now, define $\psi\colon \Gamma(r) \to \Gamma(r)$ to be $\psi(g)=ugu^{-1}$.  Since
  conjugation induces the identity on $H_*$ (e.g., see
  \cite[p.~191]{Weibel:1994-1}), we have $(\psi\circ f)_* = f_*$
  on~$H_3$.  Also, we have
  \begin{align*}
    (\psi\circ f)(t) &= u\cdot t^\epsilon u^p v^q \cdot u^{-1} = t^\epsilon
    u^{p-2} v^q [u,v]^q,
    \\
    (\psi\circ f)(u) &= u\cdot u^a v^b [u,v]^m \cdot u^{-1}
    = u^a v^b [u,v]^{m+b},
    \\
    (\psi\circ f)(v) &= u\cdot u^c v^d[u,v]^n \cdot u^{-1}
    = u^c v^d [u,v]^{n+d}.
  \end{align*}
  Apply Lemma~\ref{lemma:tb-aut-eliminate-j}, to eliminate $[u,v]^q$ in
  $(\psi\circ f)(t)$.  This changes $p$ to~$p-2$, without altering $a$, $b$,
  $c$, $d$, $\epsilon$ and~$q$ (but $m$ and $n$ are allowed to be altered).
  Using $\psi'(g)=u^{-1}gu$ in place of $\psi$, $p$ can also be changed to
  $p+2$. Similarly, $q$ can be changed to $q\pm 2$.  Applying this repeatedly,
  we can arrange $p=q=0$.  This gives us a homomorphism $f'\colon \Gamma\to
  \Gamma(r)$ of the promised form, which satisfies $f_* = f'_*$. Since $\phi$,
  $\phi'$, $\psi$ and $\psi'$ used above have $\epsilon_\bullet=1$ and
  $\delta_\bullet = 1$, we have $\epsilon_{f'}= \epsilon_f$ and
  $\delta_{f'}=\delta_f$.
\end{proof}

As the final step of our analysis, we investigate the special case of
2-connected homomorphisms in Lemma~\ref{lemma:tb-aut-clean-form}. Let
$i\colon \Gamma=\Gamma(1) \to \colim \Gamma(\ell) = \widehat\Gamma$ be the colimit map,
and $i_*\colon H_3(\Gamma) \to H_3(\widehat\Gamma)$ be the induced map. Recall
that the $\Z$ factor of $H_3(\widehat\Gamma)=(\Z_{(2)}/\Z) \times \Z$ is
generated by the image $i_*[Y]$ of the fundamental class $[Y] \in
H_3(\Gamma)=H_3(Y)=\Z$.

\begin{lemma}
  \label{lemma:tb-aut-clean-form-action}
  Suppose $f\colon \Gamma=\Gamma(1)\to \Gamma(r)$ is a 2-connected homomorphism
  such that $f(t) = t^{\epsilon_f}$.  Then the induced map $f_*\colon
  H_3(\Gamma) \to H_3(\widehat\Gamma)$ is given by
  \[
    f_* = \delta_f\cdot\epsilon_f\cdot i_*.
  \]
\end{lemma}

\begin{proof}
  Consider the subgroup of $\Gamma(\ell)$ generated by $u$, $v$ and $t^2$, which
  corresponds a double cover.  Since $[t^2,u]=[t^2,v]=1$, this subgroup is the
  internal direct product $A(\ell)\times 2\Z$, where $A(\ell)$ is generated by
  $u$ and $v$ and the infinite cyclic group $2\Z$ is generated by~$t^2$. The
  colimit $\cA\times 2\Z = \colim (A(\ell)\times 2\Z)$ is an index two subgroup
  of $\widehat\Gamma$.   Since $f$ sends $A(1)$ to $A(r)$, $f$ lifts to a
  homomorphism $g\colon A(1)\times2\Z \to A(r)\times2\Z$.  Compose them with the
  colimit maps $A(r)\times 2\Z \to \colim(A(\ell)\times 2\Z) = \cA\times 2\Z$
  and $\Gamma(r) \to \colim_\ell \Gamma(\ell) = \widehat\Gamma$, and take $H_3$,
  to obtain the following diagram.
  \begin{equation}
    \label{equation:double-cover-diagram}
    \begin{tikzcd}[row sep=large]
      H_3(A(1)\times 2\Z) \ar[r,"g_*"] \ar[d] 
      & H_3(A(r)\times 2\Z) \ar[r,"\colim"] \ar[d] 
      & H_3(\cA \times 2\Z) \ar[d]
      \\
      H_3(\Gamma(1)) \ar[r,"f_*"'] 
      & H_3(\Gamma(r)) \ar[r,"\colim"']
      & H_3(\widehat\Gamma)
    \end{tikzcd}
  \end{equation}

  We will compare the composition of the top row with the homomorphism induced
  by the colimit $i|_{A(1)\times 2\Z}\colon A(1)\times 2\Z \to \cA\times 2\Z$.
  The key property, which is a consequence of the hypothesis
  $f(t)=t^{\epsilon_f}$, is that the lift $g$ can be written as a product: $g =
  (g|_{A(1)}) \times (\epsilon\cdot)$ where $g|_{A(1)}$ is equal to the
  restriction $f|_{A(1)}\colon A(1)\to A(r)$, and $\epsilon\cdot \colon 2\Z\to
  2\Z$ is multiplication by~$\epsilon:=\epsilon_f$. So, the induced map $g_*$ on
  $H_3$ is determined by $g|_{A(1)}$ and $\epsilon$ by the K\"unneth formula.
  More precisely, since $A(1)=\Z^2$, the composition of the top row
  of~\eqref{equation:double-cover-diagram} is equal to the composition
  \begin{multline*}
    H_3(A(1)\times2\Z) = H_2(A(1))\otimes H_1(2\Z)
    \xrightarrow{(g|_{A(1)})_*\otimes(\cdot\epsilon)} H_2(A(\ell)) \otimes H_1(2\Z)
    \\
    \xrightarrow{\colim\otimes(\cdot\epsilon)} H_2(\cA) \otimes H_1(2\Z)
    \xrightarrow{\times} H_3(\cA\times2\Z).
  \end{multline*}
  By~\eqref{equation:exact-sequence-H_2(A(k))}
  and~\eqref{equation:tb-omega+1-H_2(A)}, $H_2(\cA)$ is identified with the
  subgroup $r^2\Z\subset \Z=H_2(A(r)_{ab})$.
  In our case, the homomorphism $H_2(A(1))=\Z \to
  H_2(A(r)_{ab})=\Z$ induced by $g|_{A(1)} = f|_{A(1)}$ is multiplication by
  the determinant of $A(1)=\Z^2 \to A(r)_{ab}=\Z^2$, which is equal to $\delta_f
  \cdot r^2$ by~\eqref{equation:general-form-epsilons}. From this, it
  follows that the composition of the top row
  of~\eqref{equation:double-cover-diagram} is equal to~$\delta_f\cdot \epsilon
  \cdot (i|_{A(1)\times 2\Z})_*$.

  Consequently, the composition of the bottom row
  of~\eqref{equation:double-cover-diagram}, which is the induced homomorphism
  $f_*\colon H_3(\Gamma)\to H_3(\widehat\Gamma)$, is equal to $\delta_f\cdot
  \epsilon\cdot i_*$ \emph{on the image of} $H_3(A(1)\times 2\Z) \to
  H_3(\Gamma(1))=H_3(Y)$. Since $B(A(1)\times2\Z) \to B\Gamma(1)= Y$ is a double
  cover of the 3-manifold $Y$, the image is the subgroup generated by~$2[Y]\in
  H_3(Y)$.  So,
  \[
    2\cdot f_*[Y]= 2\cdot \delta_f\cdot \epsilon\cdot i_*[Y]
    \quad\text{in }H_3(\widehat\Gamma).
  \]
  Since $H_3(\widehat\Gamma)=(\Z_{(2)}/\Z)\times \Z$ by
  Theorem~\ref{theorem:tb-transfinite-computation}\ref{item:tb-transfinite-H_3},
  $2\theta=0$ implies $\theta=0$ for every $\theta\in H_3(\widehat\Gamma)$.  It
  follows that $f_*[Y] = \delta_f\cdot \epsilon\cdot i_*[Y]$. This
  completes the proof.
\end{proof}

Now we are ready to prove Theorem~\ref{theorem:tb-final-aut}, which asserts
that the action of $f\in \Aut(\widehat\Gamma)$ on
$H_3(\widehat\Gamma)=(\Z/\Z_{(2)}) \times \Z$ is given by $f_*(a,n) =
(\delta_f\cdot a, \delta_f \cdot\epsilon_f\cdot n)$. 

\begin{proof}[Proof of Theorem~\ref{theorem:tb-final-aut}] By
  Lemma~\ref{lemma:tb-aut-on-tH_3}, the restriction of $f_*$ on
  $tH_3(\widehat\Gamma)=\Z_{(2)}/\Z$ is multiplication by~$\delta_f$.  So
  it remains to investigate $f_*$ on the generator $(0,1)\in (\Z_{(2)}/\Z)
  \times \Z$.   By Lemmas~\ref{lemma:tb-aut-clean-form}
  and~\ref{lemma:tb-aut-clean-form-action}, we may assume that the map
  $H_3(\Gamma) \to H_3(\widehat\Gamma)$ induced by $f$ sends the fundamental
  class $[Y]$ to $\delta_f\cdot\epsilon_f \cdot i_*[Y]$.  Since $(0,1)\in
  (\Z_{(2)}/\Z) \times \Z$ is the image of $[Y]$, it follows that $f_*(0,1) =
  (0,\delta_f\cdot\epsilon_f)$.
\end{proof}

\section{Nontrivial transfinite Milnor invariants}
\label{section:modified-tb}

The goal of this section is to prove Theorem~\ref{theorem:main-tb-milnor} stated in Section~\ref{subsection:results-nontrivial-milnor}, which gives an infinite family of $3$-manifolds with vanishing Milnor invariants of finite length, but distinct nontrivial transfinite Milnor invariants of length~$\omega$.
As mentioned in
Section~\ref{subsection:results-nontrivial-milnor}, we do so by using a
family of $3$-manifolds, $\{M_d \mid d \in 2\Z+1\}$: $M_d$ is defined to be the
torus bundle $T^2\times[0,1]/(h_d(z),0)\sim(z,1)$, with monodromy
$h_d=\sbmatrix{-1 & d \\ 0 & -1}$.  Note that $M_d$ is obtained from the
original torus bundle $Y$ studied in the previous sections, by modifying the
$(1,2)$-entry of the monodromy from $0$ to~$d$.

Fix an odd integer~$d$.  We will use $M_d$ as the basepoint manifold to which
other 3-manifolds $M_r$ are compared.  That is, let $\Gamma=\pi_1(M_d)$.  Our
main goal of this section is to prove Theorem~\ref{theorem:main-tb-milnor},
which asserts the following:

\begin{enumerate}
  \item For every odd integer $r$, $\bar\mu_k(M_r)$ is defined and vanishes for
  all finite~$k$.  Moreover, $\widehat{\pi_1(M_r)}/\widehat{\pi_1(M_r)}_\omega
  \cong \widehat\Gamma/\widehat\Gamma_\omega$, so $\bar\mu_\omega(M_r)$ is
  defined.
  \item But, for odd $r$ and $s$, $\bar\mu_\omega(M_r) = \bar\mu_\omega(M_s)$ if
  and only if $|r/s|$ is a square in~$\Z_{(2)}^\times$. In particular,
  $\bar\mu_\omega(M_r)$ is nontrivial if and only if $|r/d|$ is not a square.
  \item Indeed, the set of realizable values of the Milnor invariant of length
  $\omega$, $\cR_\omega(\Gamma)/{\approx}$, is in 1-1 correspondence with
  $\Z_{(2)}^\times / {\pm}(\Z_{(2)}^\times)^2$.
\end{enumerate}

Here $\pm(\Z_{(2)}^\times)^2 := \{\pm \alpha^2 \mid \alpha\in
\Z_{(2)}^\times\}$. For every $a/b \in \Z_{(2)}^\times$ with $a$, $b\in 2\Z+1$,
we have $a/b \equiv |ab| \bmod \pm(\Z_{(2)}^\times)^2$ (multiplicatively), and
in the prime factorization of the integer $|ab|$, one can assume that each prime
has exponent at most one, modulo square. So $\cR_\omega(\Gamma)/{\approx}$ is
bijective to the set of odd positive integers which have no repeated primes in
the factorization.

To show Theorem~\ref{theorem:main-tb-milnor}, we will compute the realizable
classes and the equivalence relations ${\sim}$ and ${\approx}$ for the modified
torus bundle case.  In fact, both the arguments for computation and their
outcomes are very close to the original torus bundle $d=0$ case.  However, the
modified case has a small but important difference: the action of
$\Aut(\widehat\Gamma/\widehat\Gamma_\omega)$ on $\cR_\omega(\Gamma)$ turns out
to have smaller orbits.  See
Theorem~\ref{theorem:modified-tb-comp}\ref{item:modified-tb-aut} below and
compare it with
Theorem~\ref{theorem:tb-transfinite-computation}\ref{item:tb-transfinite-aut}.
From this the nontriviality of the length $\omega$ Milnor invariants will be
obtained.

More specifically, we will show the following.

\begin{theorem}
  \label{theorem:modified-tb-comp}
  Let $\Gamma=\pi_1(M_d)$ as above, $d$ odd.  Then, the following hold.
  \begin{enumerate}
    \item\label{item:modified-tb-realization}
    Each of
    $H_3(\widehat\Gamma/\widehat\Gamma_\omega)$,
    $H_3(\widehat\Gamma/\widehat\Gamma_{\omega+1})$,
    $\cR_\omega(\Gamma)$, $\cR_{\omega+1}(\Gamma)$,
    the map $\cR_{\omega+1}(\Gamma) \to \cR_\omega(\Gamma)$ and
    the equivalence relation $\sim$ on $\cR_\omega(\Gamma)$ 
    is identical with that given in
    Theorem~\ref{theorem:tb-transfinite-computation}: that is,
    \begin{align*}
      H_3(\widehat\Gamma/\widehat\Gamma_\omega) &= \Z_{(2)}, &
      H_3(\widehat\Gamma/\widehat\Gamma_{\omega+1}) &= (\Z_{(2)}/\Z)\times \Z,
      \\
      \cR_\omega(\Gamma) &= \Z_{(2)}^\times, &
      \cR_{\omega+1}(\Gamma) &= (\Z_{(2)}/\Z)\times\{\pm1\}.
    \end{align*}
    The map $\cR_{\omega+1}(\Gamma) \to \cR_\omega(\Gamma)$ is
    $(x,\epsilon)\mapsto \epsilon$, and on $\cR_\omega(\Gamma)=\Z_{(2)}^\times$,
    $\theta\sim \theta'$ if and only if $\theta=\pm\theta'$.
    
    \item\label{item:modified-tb-aut}
    The orbits of the action of $\Aut(\widehat\Gamma/\widehat\Gamma_\omega)$ on
    $\cR_\omega(\Gamma)=\Z_{(2)}^\times$ is given by: $\phi(\theta)=\theta'$ for
    some $\phi\in \Aut(\widehat\Gamma/\widehat\Gamma_\omega)$ if and only if
    $\theta/\theta'$ is a square.

    Consequently, on $\cR_\omega(\Gamma)$, $\theta\approx \theta'$ if and only
    if $\theta/\theta' = \pm\alpha^2$ for some $\alpha\in \Z_{(2)}^\times$.

    \item\label{item:modified-tb-value}
    For every odd integer $r$, there is an isomorphism $f\colon
    \widehat{\pi_1(M_r)}/\widehat{\pi_1(M_r)}_\omega \isomto
    \widehat\Gamma/\widehat\Gamma_\omega$ such that
    $\theta_\omega(M_r)=\theta_\omega(M_r,f)=r/d \in \Z_{(2)}^\times =
    \cR_\omega(\Gamma)$.  So $\bar\mu_\omega(M_r) = r/d = rd$ and
    $\bar\mu_\omega(M_{rd}) = r$ in $\Z_{(2)}^\times /
    {\pm}(\Z_{(2)}^\times)^2 = \cR_\omega(\Gamma)/{\approx}$.
  \end{enumerate}
\end{theorem}

Theorem~\ref{theorem:main-tb-milnor} follows immediately from
Theorem~\ref{theorem:modified-tb-comp}\ref{item:modified-tb-aut}
and~\ref{item:modified-tb-value}.

The remaining part of this section is devoted to the proof of
Theorem~\ref{theorem:modified-tb-comp}.  In Section~\ref{subsection:modified-tb-lcsq}, we compute the transfinite lower central quotients $\widehat\Gamma/\widehat\Gamma_\omega$ and~$\widehat\Gamma/\widehat\Gamma_{\omega+1}$.  In Section~\ref{subsection:modified-tb-H_*}, we prove Theorem~\ref{theorem:modified-tb-comp}\ref{item:modified-tb-realization}.  In Section~\ref{subsection:modified-tb-aut}, we prove Theorem~\ref{theorem:modified-tb-comp}\ref{item:modified-tb-aut} and~\ref{item:modified-tb-value}.

\subsection{Transfinite lower central series quotients of the localization}
\label{subsection:modified-tb-lcsq}

The group $\Gamma=\pi_1(M_d)$ is the semi-direct product $\Z^2 \rtimes \Z = \Z^2
\rtimes_{h_d} \Z$, where the generator $t$ of $\Z$ acts on $\Z^2$ by $h_d =
\sbmatrix{-1 & d \\ 0 & -1}$.  In what follows, we will compute
$\widehat\Gamma/\widehat\Gamma_\omega$ and
$\widehat\Gamma/\widehat\Gamma_{\omega+1}$.

Recall that the group $\cA=\colim_{\ell:\text{odd}} A(\ell)$ was defined in the
beginning of Section~\ref{subsection:tb-omega+1-H_3}.  One can write
\[
  \cA = \{ x^\alpha y^\beta [x,y]^\gamma \mid \alpha, \beta \in \Z_{(2)},\,
  \gamma\in \Z_{(2)}/\Z\}
\]
where the group operation is given by $x^\alpha y^\beta [x,y]^\gamma \cdot
x^\lambda y^\mu [x,y]^\zeta =
x^{\alpha+\lambda}y^{\beta+\mu}[x,y]^{\gamma+\zeta-\beta\lambda}$. The group
$\cA$ has $\Z^2$ as a subgroup, which is generated by $x$ and~$y$.  (Note that
$[x,y]$ is trivial in~$\cA$.)  Also, $\Z_{(2)}/\Z=\{[x,y]^\gamma\} \subset \cA$
is a central subgroup, and $\cA/(\Z_{(2)}/\Z)=\Z_{(2)}^2$.

Recall that, for $d=0$ case, we proved that
$\widehat\Gamma=\widehat\Gamma/\widehat\Gamma_{\omega+1}$ is equal to the
semi-direct product $\cA\rtimes_{h_0}\Z$, where the generator $t$ of $\Z$ acts
on $\cA$ is given by $h_0 = \sbmatrix{-1 & 0 \\ 0 & -1}$, that is, $t\cdot x =
x^{-1}$, $t\cdot y = y^{-1}$  See Section~\ref{subsection:tb-omega+1-H_3}.

We will prove a similar result for the modified torus bundle case.

For this purpose, we need to extend the action of $t=h_d$ on $\Z_2$
to~$\cA$.  Being an extension, $t\cdot x^n = x^{-n}$, $t\cdot y^n = x^{dn}
y^{-n}$ must be satisfied for every integer~$n$, but it can be seen that a
naive attempt to define $t\cdot x^{1/n} = x^{-1/n}$, $t\cdot y^{1/n} = x^{d/n}
y^{-1/n}$ does not give a group homomorphism $t\colon \cA \to \cA$.

Instead, we use the following lemma, which can be verified by a direct
computation.  To state it, we need the fact that the multiplication by~$2$,
$\Z_{(2)}/\Z \xrightarrow{\smash{2\cdot}} \Z_{(2)}/\Z$, is an isomorphism, so
$\gamma/2 \in \Z_{(2)}/\Z$ is well-defined for every $\gamma \in \Z_{(2)}/\Z$.

\begin{lemma}
  \label{lemma:modified-monodromy-extn}
  The map $t\colon \cA \to \cA$ defined by
    \[
    t\cdot (x^\alpha y^\beta [x,y]^\gamma) =
    x^{-\alpha+d\beta} y^{-\beta} [x,y]^{\gamma+\frac{d\beta^2}{2}}
  \]
  is a group isomorphism which extends $t=h_d\colon \Z^2 \to \Z^2$. 
\end{lemma}

Define a semi-direct product $\cA\rtimes\Z = \cA\rtimes_{h_d}\Z$ by using the
action of $t$ in the lemma.  The subgroup $\Z_{(2)}/\Z$ is central in
$\cA\rtimes\Z$, and the quotient $(\cA\rtimes\Z)/(\Z_{(2)}/\Z)$ is the
semi-direct product $\Z_{(2)}^2 \rtimes \Z = \Z_{(2)}^2 \rtimes_{h_d} \Z$, which
is defined using the action of $t=h_d$ on~$\Z_{(2)}^2$.  In what follows, we
omit $h_d$ in the semi-direct product notation.

\begin{theorem}
  \label{theorem:modified-tb-localization}
  $\widehat\Gamma/\widehat\Gamma_\omega = \Z_{(2)}^2\rtimes\Z$, and
  $\widehat\Gamma/\widehat\Gamma_{\omega+1} = \cA\rtimes\Z$.  The natural maps
  of $\Gamma=\Z^2\rtimes\Z$ into $\widehat\Gamma/\widehat\Gamma_\omega$ and
  $\widehat\Gamma/\widehat\Gamma_{\omega+1}$ are the inclusions.
\end{theorem}

Indeed, it can also be shown that $\widehat\Gamma = \cA\rtimes\Z$ and
$\widehat\Gamma_{\omega+1}=\{1\}$, by modifying the arguments used
in~\cite{Cha-Orr:2011-1}.  Since we do not use this stronger fact, we will
just provide a proof of Theorem~\ref{theorem:modified-tb-localization} only.

\begin{proof}[Proof of Theorem~\ref{theorem:modified-tb-localization}]
  
  First, we will compute $\widehat\Gamma/\widehat\Gamma_\omega$ using
  Theorem~\ref{theorem:closure-in-completion-metabelian}.  For this, we need to
  compute the classical module localization $S^{-1}\Z^2$, where $S=\{s(t)\in
  \Z[t^{\pm1}] \mid s(1)=\pm1\}$.  As a $\Z[t^{\pm1}]$-module, $\Z^2$ is
  presented by the matrix $tI-h_d = \sbmatrix{t+1 & -d \\ 0 & t+1}$. So, $\Z^2$
  is annihilated by the determinant, $t^2+2t+1$. Observe that for each $s(t)\in
  S$, $s(t)s(t^{-1})\in S$.  So, $S^{-1}\Z^2$ is equal to $T^{-1}\Z^2$, where
  $T=\{\pm s(t)s(t^{-1})\mid s(t)\in S\}$.  An element $p(t)\in T$ satisfies
  $p(t)=p(t^{-1})$, so $p(t)=a_0+\sum_{i>0} a_i(t+t^{-1})^i$ with
  $p(1)=a_0+\sum_{i>0} 2^i\cdot a_i = \pm1$.  Since $t+t^{-1}$ acts on $\Z^2$ by multiplication by~$-2$, it follows
  that multiplication by $p(t)$ on $\Z^2$ is equal to multiplication by
  $a_0+\sum_{i>0} (-2)^i\cdot a_i$, which is an odd integer. Conversely, an
  arbitrary odd integer $r$ can be written as $r=\pm1 + 4k$, so the
  multiplication by $r$ on $\Z^2$ is equal to multiplication by
  $s(t):=\pm1-k(t+t^{-1}-2)$, which lies in~$S$.  It follows that $S^{-1}\Z^2 =
  (2\Z+1)^{-1}\Z^2 = \Z_{(2)}^2$.

  Also, for the augmentation ideal $I=\{p(t)\in \Z[t^{\pm1}] \mid p(1)=0\}$, an
  element $p(t)\in I^{2k}$ is of the form $p(t)=(t-1)^{2k}q(t) = (t+t^{-1}-2)^k
  \cdot t^k q(t)$.
  Since $t+t^{-1}$ acts on $\Z^2$ by multiplication by~$-2$, it follows that $I^{2k}\Z^2 \subset 4^k \Z^2$, and consequently, $\bigcap_{k<\infty} I^k \Z^2=0$.
  So
  $\Gamma=\Z^2\rtimes \Z$ is residually nilpotent.  For later use, note that the
  same argument proves that $\Z_{(2)}^2\rtimes \Z$ is residually nilpotent too.
  
  Now, by Theorem~\ref{theorem:closure-in-completion-metabelian}, the closure in
  the completion is given by
  \[
    \widehat\Gamma/\widehat\Gamma_\omega = \overline{\Z^2\rtimes \Z}
    = S^{-1}\Z^2 \rtimes \Z = \Z_{(2)}^2 \rtimes \Z.
  \]
  This proves the first conclusion.

  To compute $\widehat\Gamma/\widehat\Gamma_{\omega+1}$, we claim
  the following:
  \begin{enumerate}
    \item $H_2(\cA\rtimes\Z) = H_2(\cA) = H_2(\Z^2) = H_2(\Z^2\rtimes\Z) = \Z$,
    \item $(\cA\rtimes\Z)_{\omega}=\Z_{(2)}/\Z=\{[x,y]^\gamma\}$,
    $(\cA\rtimes\Z)_{\omega+1}=\{1\}$. 
  \end{enumerate}
  Here, the equalities between $H_2(-)$ are induced by the inclusions of the
  groups.

  Before proving the claims, we will derive the second conclusion of the
  theorem.  Since $\Z_{(2)}/\Z$ is a central abelian subgroup of $\cA\rtimes\Z$
  and the quotient $(\cA\rtimes\Z)/(\Z_{(2)}/\Z) = \Z_{(2)}^2\rtimes\Z$ is a
  local group,
  $\cA\rtimes\Z$ is local, by~\cite[Theorem~A.2, Lemma~A.4]{Cha-Orr:2011-1}. So,
  the inclusion $\Gamma=\Z^2\rtimes\Z \hookrightarrow \cA\rtimes\Z$ induces
  $\widehat\Gamma\to\cA\rtimes\Z$.  We will apply the standard Stallings
  argument to $\widehat\Gamma\to\cA\rtimes\Z$.
  We have already shown that $\widehat\Gamma/\widehat\Gamma_\omega = \Z_{(2)}^2 \rtimes \Z$.
  Combining this with claim (2) above, it follows that
  $\widehat\Gamma/\widehat\Gamma_\omega \cong
  (\cA\rtimes\Z)/(\cA\rtimes\Z)_\omega$. Since the composition $H_2(\Gamma) \to
  H_2(\widehat\Gamma) \to H_2(\cA\rtimes\Z)$ is an isomorphism by the first
  claim, $H_2(\widehat\Gamma) \to H_2(\cA\rtimes\Z)$ is surjective.  So, by
  Stallings' work~\cite{Stallings:1965-1},
  $\widehat\Gamma/\widehat\Gamma_{\omega+1} \cong
  (\cA\rtimes\Z)/(\cA\rtimes\Z)_{\omega+1}$.  By the second claim, it follows
  that $\widehat\Gamma/\widehat\Gamma_{\omega+1} \cong \cA\rtimes\Z$.

  Therefore, to complete the proof, it only remains to show the claims. We begin
  with the first claim, which concerns~$H_2$.  In fact, $H_*(\cA\rtimes\Z)$ can
  be computed using the Wang sequence
  \begin{equation}
    \label{equation:wang-sequence-modified-tb}
    \cdots \to H_2(\cA) \xrightarrow{1-t_*} H_2(\cA) \to H_2(\cA\rtimes\Z)
    \to H_{1}(\cA) \xrightarrow{1-t_*} H_{1}(\cA) \to \cdots  
  \end{equation}
  similarly to Section~\ref{subsection:tb-omega+1-H_3}.  (Our
  \eqref{equation:wang-sequence-modified-tb} here is analogous
  to~\eqref{equation:wang-sequence-Gamma-hat}.)
  By~\eqref{equation:tb-omega+1-H_2(A)}, $H_2(\cA)=H_2(\Z^2)=\Z$.  Since $h_d$
  has determinant one, $t_* = \id$ on $H_2(\cA)$ and thus $1-t_*=0$.  Also,
  $H_1(\cA)=\Z_{(2)}^2$, and $t_*$ on $H_1(\cA)$ is given by~$h_d$.  So,
  $1-t_*=\sbmatrix{2 & -d \\ 0 & 2}$ on $H_1(\cA)$ and this is injective.
  Therefore, from \eqref{equation:wang-sequence-modified-tb}, it follows that
  \begin{equation}
    \label{equation:modified-tb-H_2}
    H_2(\cA\rtimes\Z)=H_2(\cA)=H_2(\Z^2)=\Z.
  \end{equation}
  We remark that while our monodromy $h_d$ ($d$ odd) is different from the $d=0$
  case, $H_2(\cA\rtimes\Z)$ remains the same as the $d=0$ case given
  in~\eqref{equation:tb-omega+1-H_2}.
  
  For the second claim, we proceed similarly to the proof of
  Lemma~\ref{lemma:tb-transfinite-lcs}.  We have already shown that
  $(\cA\rtimes\Z) / (\Z_{(2)}/\Z) = \Z_{(2)}^2\rtimes \Z$ is residually
  nilpotent.  So, $(\cA\rtimes\Z)_\omega\subset \Z_{(2)}/\Z$.  For the reverse
  inclusion, observe that $[x^\beta,t]=x^{-2\beta}$, so by induction,
  $x^{2^k\beta}\in (\cA\rtimes\Z)_{k+1}$.  Thus $[x,y]^{2^k\beta} =
  [x^{2^k\beta},y]$ lies in $(\cA\rtimes\Z)_{k+2}$ for all~$\beta\in\Z_{(2)}$.
  For every $\gamma\in \Z_{(2)}$, there is $\beta\in \Z_{(2)}$ such that
  $2^k\beta \equiv \gamma \bmod \Z$, since $2$ is invertible in~$\Z_{(2)}/\Z$.
  It follows that $[x,y]^\gamma = [x,y]^{2^k\beta} \in (\cA\rtimes\Z)_{k+2}$.
  Since it holds for every~$k$, $[x,y]^\gamma \in (\cA\rtimes\Z)_\omega$.  This
  shows that $(\cA\rtimes\Z)_\omega = \Z_{(2)}/\Z$.  Finally, since
  $[x,y]^\gamma$ is central in $\cA\rtimes\Z$,
  $(\cA\rtimes\Z)_{\omega+1}=\{1\}$.  This completes the proof of the claims.
\end{proof}

\subsection{Homology and realizable classes}
\label{subsection:modified-tb-H_*}

We will prove
Theorem~\ref{theorem:modified-tb-comp}\ref{item:modified-tb-realization}. To
compute $H_3(\widehat\Gamma/\widehat\Gamma_\omega)$, we use the Wang sequence
for $\widehat\Gamma/\widehat\Gamma_\omega = \Z_{(2)}^2\rtimes\Z$, similarly to
Section~\ref{subsection:tb-omega-H_3}.  Indeed, the Wang sequence was already
given in~\eqref{equation:wang-sequence-omega}:
\begin{multline*}
  0=H_3(\Z_{(2)}^2)
  \to H_3(\widehat\Gamma/\widehat\Gamma_\omega) \to H_2 (\Z_{(2)}^2)
  \xrightarrow{1-t_*} H_2(\Z_{(2)}^2)
  \to H_2(\widehat\Gamma/\widehat\Gamma_\omega)
  \\
  \to H_1 (\Z_{(2)}^2) \xrightarrow{1-t_*} H_1(\Z_{(2)}^2)
  \to H_1(\widehat\Gamma/\widehat\Gamma_\omega) \to \Z \to 0
\end{multline*}
Here, the difference from Section~\ref{subsection:tb-omega-H_3} is that $t_*$ is
induced by $h_d=\sbmatrix{-1 & d \\ 0 & -1}$.  So, $1-t_*$ on
$H_2(\Z_{(2)}^2)=\Z_{(2)}$ is zero, and $1-t_*$ on $H_1(\Z_{(2)}^2)=\Z_{(2)}^2$
is $\sbmatrix{2 & -d \\ 0 & 2}$.  It follows that
\begin{equation}
  \label{equation:modified-tb-omega-H_i}
  \begin{aligned}
    H_3(\widehat\Gamma/\widehat\Gamma_\omega) &= H_2(\Z_{(2)}^2) = \Z_{(2)}, \\
    H_2(\widehat\Gamma/\widehat\Gamma_\omega) &= H_2(\Z_{(2)}^2) = \Z_{(2)}, \\
    H_1(\widehat\Gamma/\widehat\Gamma_\omega) &= \Z_4 \times \Z.
  \end{aligned}
\end{equation}

Note that $H_i(\widehat\Gamma/\widehat\Gamma_\omega)$ remains the same as that
of original torus bundle in Section~\ref{subsection:tb-omega-H_3} for $i=2,3$,
while $H_1(\widehat\Gamma/\widehat\Gamma_\omega)$ is altered since $d$ is odd.
Compare \ref{equation:modified-tb-omega-H_i} with~\eqref{equation:tb-omega-H_i}.
But, $H_1(\widehat\Gamma/\widehat\Gamma_\omega)$ is still a finite abelian
2-group. By this, the analysis of the cap
products~\eqref{equation:tb-omega-cap-H^2-to-H_1}
and~\eqref{equation:tb-omega-cap-H^1-to-H_2} (which uses
Theorem~\ref{theorem:main-realization}) in Section~\ref{subsection:tb-omega-H_3}
applies to our case without any modification.  This shows that
$\cR_\omega(\Gamma) = \Z_{(2)}^\times$.

To compute $H_3(\widehat\Gamma/\widehat\Gamma_{\omega+1})$, we proceed similarly
to Section~\ref{subsection:tb-omega+1-H_3}.  For
$\widehat\Gamma/\widehat\Gamma_{\omega+1} = \cA\rtimes\Z$, we have the Wang
sequence~\eqref{equation:wang-sequence-Gamma-hat}
\begin{equation*}
  \cdots \to H_i(\cA) \xrightarrow{1-t_*} H_i(\cA)
  \to H_i(\widehat\Gamma/\widehat\Gamma_{\omega+1})
  \to H_{i-1}(\cA) \xrightarrow{1-t_*} H_{i-1}(\cA) \to \cdots  
\end{equation*}
where $t_*$ is again induced by $h_d=\sbmatrix{-1 & d \\ 0 & -1}$. We have
$H_3(\cA)=H_3(\Z_{(2)}/\Z)=\Z_{(2)}/\Z$ by~\eqref{equation:tb-omega+1-H_3(A)}.
Since the subgroup $\Z_{(2)}/\Z\subset\cA$ is generated by $[x,y]^\gamma$ on
which our $t$ acts trivially, $t_*$ on $H_3(\cA)$ is the identity.  Also, since
$H_2(\cA)=H_2(\Z^2)=\Z$ by~\eqref{equation:tb-omega+1-H_2(A)}, $t_*$ on
$H_2(\cA)$ is the identity too. It follows that
$H_3(\widehat\Gamma/\widehat\Gamma_{\omega+1}) = (\Z_{(2)}/\Z)\rtimes \Z$, the
same as~\eqref{equation:tb-omega+1-H_3} in
Section~\ref{subsection:tb-omega+1-H_3}.

Also, for $\theta\in H_3(\widehat\Gamma/\widehat\Gamma_{\omega+1})$, the
analysis of the cap products \eqref{equation:tb-omega-cap-H^2-to-H_1} and
\eqref{equation:tb-omega-cap-H^1-to-H_2} in
Section~\ref{subsection:tb-omega+1-realizable-classes} is carried out for our
case
\begin{align*}
  {}\cap\theta\colon & tH^2(\widehat\Gamma/\widehat\Gamma_{\omega+1})
  \to tH_1(\widehat\Gamma/\widehat\Gamma_{\omega+1})
  = tH_1(\Gamma)
  \\
  {}\cap\theta\colon & H^1(\widehat\Gamma/\widehat\Gamma_{\omega+1})
  \to H_2(\widehat\Gamma/\widehat\Gamma_{\omega+1})
  /\Ker\{H_2(\widehat\Gamma/\widehat\Gamma_{\omega+1}) \rightarrow
  H_2(\widehat\Gamma/\widehat\Gamma_\omega)\}
\end{align*}
without modification, using that $tH_1(\Gamma)=\Z_4$ is a finite abelian
2-group.  This shows that $\cR_{\omega+1}(\Gamma) = \Z_{(2)}/\Z \times\{\pm1\}
\subset H_3(\widehat\Gamma/\widehat\Gamma_{\omega+1})$.

Note that we have shown that $\cR_{\omega}(\Gamma)$ and $\cR_{\omega+1}(\Gamma)$
are the same as those of the original torus bundle case ($d=0$).  So, by the
argument in the last paragraph, $\cR_{\omega+1}(\Gamma) \to
\cR_{\omega}(\Gamma)$ is also the same as the the original torus bundle case.

To complete the proof of
Theorem~\ref{theorem:modified-tb-comp}\ref{item:modified-tb-realization}, it
remains to determine the equivalence relation ${\sim}$ on
$\cR_\omega(\Gamma)=\Z_{(2)}^\times \subset
H_3(\widehat\Gamma/\widehat\Gamma_\omega)=\Z_{(2)}$.

Let $\theta=a/b \in \cR_\omega(\Gamma)=\Z_{(2)}^\times$ with $a$, $b$ odd
integers.  To compute the equivalence class of $\theta$, we will first find a
3-manifold realizing~$\theta$.  Recall that $M_d$ is the modified torus bundle,
with monodromy $h_d = \sbmatrix{-1 & d \\ 0 & -1}$, and that
$\Gamma=\pi_1(M_d)$.  For another odd integer $r$ which will be specified later,
consider the 3-manifold $M_r$.  By
Theorem~\ref{theorem:modified-tb-localization} aplied to $r$ instead of $d$, we
have $\widehat{\pi_1(M_r)}/\widehat{\pi_1(M_r)}_\omega =
\Z_{(2)}^2\rtimes_{h_r}\Z$.  Because the following observation will also be used
later, we state it as a lemma.

\begin{lemma}
  \label{lemma:modified-tb-diag-aut}
  Let $\alpha$, $\beta \in \Z_{(2)}^\times$.  Then
  $\phi=\phi_{\alpha,\beta}\colon \Z_{(2)}^2\rtimes_{h_r}\Z \to
  \Z_{(2)}^2\rtimes_{h_d}\Z$ given by $\phi(a,b,n) = (\alpha\cdot a, \beta\cdot
  b, n)$ is a group isomorphism if and only if $d\beta = r\alpha$.  When it is the case, the induced isomorphism
  \[
    \phi_*\colon H_3(\widehat{\pi_1(M_r)}/\widehat{\pi_1(M_r)}_\omega)
    = \Z_{(2)}\to H_3(\widehat\Gamma/\widehat\Gamma_\omega) = \Z_{(2)}
  \]
  is multiplication by~$\alpha\beta$. 
\end{lemma}
\begin{proof}
  Since the monodromies are $h_d=\sbmatrix{-1 & d \\ 0 & -1}$ and
  $h_r=\sbmatrix{-1 & r \\ 0 & -1}$ and $\phi=\sbmatrix{\alpha & 0 \\ 0 &
  \beta}$ on $\Z_{(2)}^2$, our $\phi$ is an isomorphism between the semi-direct
  products if and only if the matrix identity $h_d\phi = \phi h_r$ holds.  From
  this, the first conclusion follows immediately, using the condition $d\ne 0$.

  Since $H_3(\widehat{\pi_1(M_r)}/\widehat{\pi_1(M_r)}_\omega) = H_2(\Z_{(2)}^2)
  = \Z_{(2)}$ by~\eqref{equation:modified-tb-omega-H_i} and since the
  restriction $\phi|_{\Z_{(2)}^2}$ is $\sbmatrix{\alpha & 0 \\ 0 & \beta}$, the
  induced map $\phi_*$ on $H_3$ is the multiplication by $\det\phi|_{\Z_{(2)}^2}
  = \alpha\beta$.
\end{proof}

For our purpose, let $r=abd$, $\alpha=1/b$ and $\beta=a$.  By
Lemma~\ref{lemma:modified-tb-diag-aut}, $\phi=\phi_{\alpha,\beta}$ is an
isomorphism, and  $\phi_*$ on $H_3$ is multiplication by~$a/b$.  Since the
fundamental class $[M_r]$ is equal to $1\in \Z_{(2)} =
(\widehat{\pi_1(M_r)}/\widehat{\pi_1(M_r)}_\omega)$, it follows that the value
of the invariant $\theta(M_r) = \theta(M_r, \phi)$ defined using the isomorphism
$\phi$ is equal to the class $\theta=a/b \in \cR_\omega(\Gamma)$. Therefore, the
equivalence class of $\theta$, with respect to ${\sim}$, is equal to the image
of the composition
\[
  \begin{tikzcd}[row sep=1.5ex]
    \cR_{\omega+1}(\pi_1(M_r)) \ar[r] \ar[d,equal]
    & \cR_\omega(\pi_1(M_r)) \ar[r,"\cong", "\phi_*"'] \ar[d,equal]
    & \cR_\omega(\Gamma) \ar[d,equal]
    \\
    |[overlay]|
    (\Z_{(2)}/\Z)\times\{\pm1\} 
    & \Z_{(2)}^\times
    & \Z_{(2)}^\times
  \end{tikzcd}
\]
by Definition~\ref{definition:milnor-equivalence-rel}.  Since the first arrow is
$(x,\pm1) \mapsto \pm1$ and the second arrow is multiplication by $\theta=a/b$,
it follows that $\theta\sim\theta'$ if and only if $\theta' = \pm\theta$.  The
completes the proof of
Theorem~\ref{theorem:modified-tb-comp}\ref{item:modified-tb-realization}.

\subsection{Automorphism action and Milnor invariants}
\label{subsection:modified-tb-aut}

Recall that $\Gamma=\pi_1(M_d)$ where $d$ is fixed.  We will prove
Theorem~\ref{theorem:modified-tb-comp}\ref{item:modified-tb-aut}.

Suppose $\phi\colon \widehat\Gamma/\widehat\Gamma_\omega \to
\widehat\Gamma/\widehat\Gamma_\omega=\Z_{(2)}^2\rtimes\Z$ is an automorphism.
Similarly to the proof of Lemma~\ref{lemma:tb-aut-general-form}, we have that
$\phi$ restricts to an automorphism on the subgropup $\Z_{(2)}^2$, since
$\Z_{(2)}^2$ is the first rational derived subgroup
of~$\widehat\Gamma/\widehat\Gamma_\omega$.  Write $\phi|_{\Z_{(2)}^2} =
\sbmatrix{\alpha & \smash{\beta} \\ \smash{\gamma} & \delta} \in
\operatorname{GL}(2,\Z_{(2)})$. For the generator $t$ of the $\Z$ factor of
$\Z_{(2)}^2\rtimes\Z$, we have that $\phi(0,t) = (v,t^\epsilon)$ for some
$v\in\Z_{(2)}^2$ and $\epsilon\in\{\pm1\}$, since $\phi$ is an automorphism on
the quotient $(\widehat\Gamma/\widehat\Gamma_\omega) / \Z_{(2)}^2 = \Z$.  Since
$\phi$ is a group homomorphism on the semi-direct product with respect to the
monodromy $h_d$, the matrix identity $\phi h_d = h_d^\epsilon \phi$ must be
satisfied.  By comparing the matrix entries, it implies that $\phi|_{\Z_{(2)}^2}
= \sbmatrix{\alpha & \beta \\ 0 & \epsilon\alpha}$. (Here one uses the
assumption that $d$ is \emph{nonzero}!)  From this, it follows that the induced
automorphism $\phi_*$ on $H_3(\widehat\Gamma/\widehat\Gamma_\omega) =
H_2(\Z_{(2)}^2)=\Z_{(2)}$ is equal to multiplication by $\epsilon\cdot
\det\phi|_{\Z_{(2)}^2} = \alpha^2$.  Note that $\alpha\in \Z_{(2)}^\times$ since
$\phi|_{\Z_{(2)}^2}$ is invertible over~$\Z_{(2)}$.

Conversely, the above computation also shows that for any square $\alpha^2 \in
\Z_{(2)}^\times$, there is an automorphism $\phi$ on
$\widehat\Gamma/\widehat\Gamma_\omega=\Z_{(2)}^2\rtimes\Z$ such that $\phi_*$ on
$H_3$ is multiplication by~$\alpha^2$.  For instance, by setting $\beta=0$ and
$\epsilon=1$, the automorphism $\phi$ given by $\phi|_{\Z_{(2)}^2} =
\sbmatrix{\alpha & 0 \\ 0 & \alpha}$ and $\phi(0,t)=(0,t)$ has that property.

From the above, Theorem~\ref{theorem:modified-tb-comp}\ref{item:modified-tb-aut}
follows immedately: for $\theta$, $\theta' \in \Z_{(2)}^\times =
\cR_\omega(\Gamma)$, $\phi(\theta)=\theta'$ for some $\phi\in
\Aut(\widehat\Gamma/\widehat\Gamma_\omega)$ if and only if $\theta/\theta'$ is
a square in~$\Z_{(2)}^\times$.  By the above computation of the equivalence
relaiton ${\sim}$ and by Definition~\ref{definition:milnor-equivalence-rel}, it
also follows that $\theta \approx \theta'$ in $\cR_\omega(\Gamma)$ if and only
if $\theta/\theta' = \pm\alpha^2$ for some $\alpha\in \Z_{(2)}^\times$.

Finally, we will prove
Theorem~\ref{theorem:modified-tb-comp}\ref{item:modified-tb-value}. Recall that
$d$ is the fixed odd integer.  Let $r$ be an arbitrary odd integer. Let
$\theta=r/d \in \cR_\omega(\widehat\Gamma/\widehat\Gamma_\omega) =
\Z_{(2)}^\times$.  Apply Lemma~\ref{lemma:modified-tb-diag-aut}, for
$(\alpha,\beta) = (1,r/d)$, to obtain the isomorphism
\[
  \phi=\phi_{\alpha,\beta}\colon
  \widehat{\pi_1(M_r)}/\widehat{\pi_1(M_r)}_\omega
  \to \widehat\Gamma/\widehat\Gamma_\omega.
\]
Furthermore, Lemma~\ref{lemma:modified-tb-diag-aut} says that $\phi_*$ on $H_3$
is multiplication by $\alpha\beta = r/d\in \Z_{(2)}^\times$. Since the
fundamental class is $[M_r]=1\in\Z_{(2)}$, we have $\theta_\omega(M_r) =
\theta_\omega(M_r, \phi)=\phi_*(1)=r/d$.  This completes the proof of
Theorem~\ref{theorem:modified-tb-comp}, the last theorem of this paper.

\section{Questions}
\label{section:questions}

We list some questions which naturally arise from this work.

\begin{enumerate}
  \item\label{item:question-finite-type}
  Can one interpret the invariants $\theta_k(M)$ and $\bar\mu_k(M)$ of
  finite length (i.e.~$k<\infty$) as Gusarov-Vassiliev finite type
  invariants in an appropriate sense?
\end{enumerate}

We remark that $\theta_k(M)$ and $\bar\mu_k(M)$ are invariant under
Habiro-Gusarov clasper surgery, which is now often called
$Y_k$-equivalence.  More precisely, the following hold.

Fix a closed 3-manifold group~$\Gamma$, and let $M$ and $M'$ be two closed 3-manifold which are $Y_{k-1}$-equivalent.  Then $\theta_k(M)$ is defined if and
only if $\theta_k(M')$ is defined, and when they are defined,
$\theta_k(M)=\theta_k(M')$ in $\cR_k(\Gamma)/\Aut(\Gamma/\Gamma_k)$ if
$M$ and $M'$ are $Y_k$-equivalent.

The following three questions are relevant.

\begin{enumerate}[resume]
  \item\label{item:question-kontsevich-integral}
  Can one extract the invariants $\theta_k(M)$ and $\bar\mu_k(M)$ of
  finite length from (some variant of) the Kontsevich integral, or related
  quantum invariants?

  \item\label{item:question-transfinite-type}
  Our results strongly suggest that there should be a notion of transfinite type invariants.
  Can one interpret the transfinite length invariants $\theta_\kappa(M)$ and $\bar\mu_\kappa(M)$ as a finite type invariant? If not, can one
  generalize the notion of finite type invariants of 3-manifolds to a
  suitable notion of ``transfinite type'' invariants, so that the invariant
  $\theta_\kappa(M)$ of transfinite length can be viewed as invariants of
  transfinite type?

  \item\label{item:question-transfinite-kontsevich}
  Can we extend the definition of the Kontsevich integral (of 3-manifolds or links) to a transfinite version?
\end{enumerate}

The following addresses the (non)triviality of the transfinite invariants of a
given length.

\begin{enumerate}[resume]
  \item\label{item:question-nontriviality-given-length}  For every (countable) ordinal $\kappa$, is there a closed 3-manifold group
  $\Gamma$ for which the sets $\cR_\kappa(\Gamma)/\mathord{\sim}$ and
  $\cR_\kappa(\Gamma)/\mathord{\approx}$ have more than one element?
\end{enumerate}

Milnor's original work~\cite{Milnor:1957-1} combined with Orr's
result~\cite{Orr:1989-1} tell us that the answer
to~\ref{item:question-nontriviality-given-length} is affirmative for
finite~$\kappa$.  See also Theorem~\ref{theorem:main-tb-finite} in this paper.
Theorems~\ref{theorem:main-tb-omega} and~\ref{theorem:main-tb-milnor} show that
the answer is affirmative for $\kappa=\omega$.

\begin{enumerate}[resume]
  \item\label{item:maximal-countable-ordinal} 
    Is there a countable ordinal $\varpi$ such that if $\Gamma$ is a $3$-manifold group and $M$ is a $3$-manifold equipped with an isomorphism $\widehat{\pi_1(M)} / \widehat{\pi_1(M)}_\varpi \to \widehat\Gamma / \widehat\Gamma_\varpi$ for which $\bar\mu_\varpi(M)$ vanishes (over $\Gamma$), then $\bar\mu_\kappa(M)$ is defined and vanishes (over $\Gamma$) for every $\kappa > \varpi$. 
  \item\label{item:question-concordance}
  Do $\theta_\kappa$ and $\bar\mu_\kappa(M)$ (with $\kappa$ either finite or
  transfinite) reveal new information on link concordance?
\end{enumerate}

Regarding~\ref{item:question-concordance}, consider the following. Fix rational numbers $a_1/b_1,\ldots,a_m/b_m \in \Q$.  For a given $m$-component link $L$,
perform Dehn filling on the exterior of the link, with slopes~$a_i/b_i$,
to obtain a closed 3-manifold.  Call it~$M_L$.  Fix a link $L_0$, and
let $Y=M_{L_0}$, $\Gamma=\pi_1(Y)$. To compare a given link $L$
with the link $L_0$, consider the invariants $\theta_\kappa(M_L)$ and
$\bar\mu_\kappa(M)$, over the group $\Gamma$, as link invariants.

It seems particularly interesting whether $\theta_\kappa$ and
$\bar\mu_\kappa$ of transfinite length gives a new nontrivial link invariant
in this way.

In addition, the finite length case may also have some interesting potential
applications.  Recall from Section~\ref{section:tb-finite-computation} that
there are examples for which the finite length invariants $\theta_k$ live in
finite abelian groups, and thus have torsion-values.

\begin{enumerate}[resume]
  \item\label{item:question-torsion-concordance-inv}
  Do $\theta_k$ and $\bar\mu_k$ of finite length give new torsion-valued
  link concordance invariants?
\end{enumerate}

The following is closely related to~\ref{item:question-torsion-concordance-inv}.
In~\cite{Conant-Schneiderman-Teichner:2012-2} (see also the
survey~\cite{Conant-Teichner-Schneiderman:2011-1} of a series of related
papers), Conant, Schneiderman and Teichner proposed a higher order version of
the classical Arf invariant for links.  It may be viewed as certain 2-torsion
valued information extracted from Whitney towers and gropes in 4-space.  A
key conjecture in the theory of Whitney towers is whether the higher order Arf
invariants are nontrivial.

\begin{enumerate}[resume]
  \item\label{item:question-higher-order-arf}
  Are the invariants $\theta_k$ and $\bar\mu_k$ related to the higher
  order Arf invariants? More specifically, can one show the conjectural
  nontriviality of the higher order Arf invariants using these invariants (of
  certain 3-manifolds associated to links)?
\end{enumerate}

Also, the existence of transfinite Milnor invariants suggest the existence of transfinite Arf-invariants.

\begin{enumerate}[resume]
  \item\label{item:transfinite-Arf} Do transfinite Arf invariants of links and $3$-manifolds exist?
  \item If so, are these determined by the invariants, or some analogue of the invariants, $\theta_\kappa$?
 \end{enumerate}

\bibliographystyle{amsalpha}
\renewcommand{\MR}[1]{}

\bibliography{research.bib}

\end{document}